\documentclass[pdftex, 10 pt, a4paper, oneside, reqno]{amsart} 

\usepackage[latin9]{inputenc} 
\usepackage[T1]{fontenc}

\usepackage{lmodern}	
\usepackage[english]{babel}
\usepackage{ngerman}
\usepackage{amsmath}
\usepackage{amssymb}
\usepackage{amsfonts}
\usepackage{amstext}
\usepackage{amsthm}
\usepackage{hyperref}
\usepackage{xcolor}
\usepackage{graphicx}
\usepackage{tabularx}
\usepackage{subfigure}

\def\D{D}
\def\H{H}

\def\d{\,\mathrm{d}}

\usepackage{remreset} 
\makeatletter
\@removefromreset{footnote}{section}
\makeatother

\usepackage[a4paper, hdivide={3.95 cm,,3.95 cm},vdivide={4.2cm,,5.3cm}]{geometry}

\begin{document}

\theoremstyle{plain}
\theoremstyle{plain}
\newtheorem{Pp}{Proposition}[section] 
	\newtheorem{Thm}[Pp]{Theorem} 
	\newtheorem*{ThmWithout}{Theorem} 
	\newtheorem{Lm}[Pp]{Lemma}
	\newtheorem{Cor}[Pp]{Corollary}
	\newtheorem*{E1}{Assumption (E1)}
	\newtheorem*{E2}{Assumption (E2)}
	\newtheorem*{E3}{Assumption (E3)}
	\newtheorem*{E4}{Assumption (E4)}
	\newtheorem*{E5}{Assumption (E5)}
\theoremstyle{definition}			
	\newtheorem*{DynAn}{The analytical dynamics (A)}
	\newtheorem*{DynAltern}{Sufficient analytical and stochastic dynamics (A)' and (S)'}
  \newtheorem*{DynSt}{The stochastic dynamics (S)}
	\newtheorem{Df}[Pp]{Definition}
	\newtheorem{Ex}[Pp]{Example}
	\newtheorem*{Ob}{Observation and Interpretation}
	\newtheorem{Cond}[Pp]{Condition}
	\newtheorem*{Ass}{Assumption}
	\newtheorem{Rm}[Pp]{Remark}
	\newtheorem{Persp}[Pp]{Remark} 
	\newtheorem{Not}[Pp]{Notation}
	\newtheorem{Int}[Pp]{Motivation and Interpretation}	
	\newtheorem{Int2}[Pp]{Interpretation}	
	\newtheorem*{Pot}{Dynamical system conditions (C0)}
	\newtheorem*{Ass2}{Ergodicity assumptions (C1)-(C3)}
	\newtheorem*{Ass3}{Hypocoercivity assumptions (C1)-(C3)}
	\newtheorem*{Ass4}{Ergodicity assumptions (C0)'-(C3)'}

\def\D{D}
\def\H{H}
\def\d{\,\mathrm{d}}
\def\BreiteTabelle{13.13cm} 
\def\Breite{1.25ex}
\def\BreiteZwei{-1.6ex}
\def\BreiteDrei{-3.5ex}
\def\df{:=}
\DeclareRobustCommand*{\normone}{|}
\DeclareRobustCommand*{\normtwo}{\|}

\title[Ergodicity for singular degenerate Kolmogorov equations]{A hypocoercivity related ergodicity method with rate of convergence for singularly distorted degenerate Kolmogorov equations and applications}

\author{Martin Grothaus}
\address{Martin Grothaus, Mathematics Department, University of Kaiserslautern, \newline 
P.O.Box 3049, 67653 Kaiserslautern, Germany. {\rm \texttt{Email:~grothaus@mathematik.uni-kl.de}},\newline
Functional Analysis and Stochastic Analysis Group, \newline
{\rm \texttt{URL:~http://www.mathematik.uni-kl.de/fuana/ }}} 

\author{Patrik Stilgenbauer}
\address{
Patrik Stilgenbauer, Mathematics Department, University of Kaiserslautern, \newline
P.O.Box 3049, 67653 Kaiserslautern, Germany. 
{\rm \texttt{Email:~stilgenb@mathematik.uni-kl.de}}, \newline
Functional Analysis and Stochastic Analysis Group, \newline
{\rm \texttt{URL:~http://www.mathematik.uni-kl.de/fuana/ }}}

\date{\today}

\subjclass[2000]{Primary 37A25; Secondary 58J65}

\keywords{Ergodicity; Rate of convergence; Degenerate diffusions; Singularly distorted diffusions; Kolmogorov backward equation; Hypocoercivity; Operator Semigroups; Generalized Dirichlet Forms;  Hypoellipticity; Poincar\'e inequality; N-particle Langevin dynamics; Spherical velocity Langevin dynamics;  Fiber lay-down;  Stratonovich SDEs on manifolds; Fokker-Planck equation}

\begin{abstract}
In this article we develop a new abstract strategy for proving ergodicity with explicit computable rate of convergence for diffusions associated with a degenerate Kolmogorov operator $L$. A crucial point is that the evolution operator $L$ may have singular and nonsmooth coefficients. This allows the application of the method e.g.~to degenerate and singular particle systems arising in Mathematical Physics. As far as we know in such singular cases the relaxation to equilibrium can't be discussed with the help of existing approaches using hypoellipticity, hypocoercivity or stochastic Lyapunov type techniques. The method is formulated in an $L^2$-Hilbert space setting and is based on an interplay between Functional Analysis and Stochastics. Moreover, it implies an ergodicity rate which can be related to $L^2$-exponential convergence of the semigroup. Furthermore, the ergodicity method shows up an interesting analogy with existing hypocoercivity approaches. In the first application we discuss ergodicity of the $N$-particle degenerate Langevin dynamics with singular potentials.  The dual to this equation is also called the kinetic Fokker-Planck equation with an external confining potential. In the second example we apply the method to the so-called (degenerate) spherical velocity Langevin equation which is also known as the fiber lay-down process arising in industrial mathematics.
\end{abstract}

\maketitle
\section{Introduction} \label{Section_Introduction_Ergodicity}

Studying the decay to equilibrium of degenerate kinetic equations or diffusions is still an active and demanding mathematical research area lying in between modern Stochastics and Functional Analysis. Especially in the last decade, many results concerning the exponential relaxation to equilibrium of the \textit{kinetic (degenerate) Fokker-Planck equation with an external confining potential} have been obtained in case the underlying potential is sufficiently smooth and nonsingular. For analytical approaches see e.g.~\cite{HeNi04}, \cite{HN05}, \cite{HeNi06}, \cite{Her07}, \cite{Vil09}, \cite{Dua11}, \cite{DMS13} or \cite{GS12B}. Therein tools from hypoellipticity and hypocoercivity are applied to the previous equation as well as to other kinetic models implying an exponential rate of convergence. Since the kinetic (degenerate) Fokker-Planck equation describes the evolution of the probability density of the \textit{Langevin equation}, also stochastic approaches are available for studying the exponential longtime behavior of this dynamics directly. The interested reader is referred e.g.~to \cite{MS02} or \cite{Wu01} where methods based on stochastic Lyapunov type techniques are used. Moreover, consider article \cite{BCG08} in which extended Lyapunov-Poincar\'{e} inequalites are developed and applied to the kinetic Fokker-Planck equation. For further studies about the longtime behavior of this dynamics see also article \cite{Bau13} which is based on generalized Bakry-\'{E}mery conditions.

However, in Statistical Mechanics and Mathematical Physics the underlying potential in the Langevin dynamics is usually of Lennard-Jones type. Hence it is singular and analyzing the decay to equilibrium can't be discussed with the abovementioned methods. Suitable tools to handle this situation are provided in \cite{CG10}. In the latter article ergodicity with a rate of convergence for the so-called \textit{N-particle Langevin dynamics with singular potentials} is established. As explained in \cite{CG10}, exponential convergence of the $N$-particle Langevin semigroup in $L^2$ would imply such an ergodicity rate. In this sense, the $N$-particle Langevin dynamics is called exponentially ergodic therein. The method used in \cite{CG10} is based on the strategy from \cite{GK08} in which ergodicity with rate of convergence of the so-called two-dimensional fiber lay-down process is proven.

In the underlying article we aim to generalize the method from \cite{GK08} and \cite{CG10} to an abstract setting. We develop a new abstract method in a Hilbert space framework which is suitable for proving ergodicity with explicit computable rate of convergence for various diffusion processes associated with a degenerate non-coercive Kolmogorov (backward) evolution operator $L$.  Again the generator may have singular and nonsmooth coefficients. As in the theory of hypocoercivity, it decomposes into a symmetric dissipative part $S$ and an antisymmetric conservative term $A$ such that the combination of both implies the phenomenon of relaxation to equilibrium. Our method applies to evolution equations with order of degeneracy equal to one, i.e., $S$ usually acts in the velocity and $A$ mixes space and velocity variables. We mention that the method is further based on an interplay between Functional Analysis and Stochastics and relies on martingale methods. While the analytical part of the dynamics can be constructed using techniques from the theory of operator semigroups, the existence of the stochastic part of the dynamics can usually be guaranteed in the applications via the theory of generalized Dirichlet-forms (see \cite{Fuk80}, \cite{Fuk94}, \cite{MR92}, \cite{St99} or \cite{Tru05}) or by using tools from \cite{BBR06}. As in \cite{CG10} the ergodicity rate can again be related to exponential convergence of the semigroup in an $L^2$-space and the convergence to equilibrium of the underlying dynamics may therefore reasonably be called $L^2$-exponentially ergodic, see Remark \ref{Rm_comparison_ergodicity_exponenential_rate}.

Afterwards, we discuss the application of our abstract method. At first, we show how the $N$-particle Langevin dynamics from \cite{CG10} fits into our setting. The equation is introduced in detail below. As consequence, we get back the ergodicity result derived in  \cite[Sec.~4.2]{CG10} and obtain a further specification for the rate of convergence. In the second example we apply the method to the generalized version of a fiber lay-down process. The latter is introduced in \cite{KMW12} and in \cite{GKMS12}. Consider these articles as well as \cite{GKMW07}, \cite{GS12}, and \cite{GS12B} for further motivation and the industrial application of this model. As explained in \cite{GS12} we remark that the generalized fiber lay-down dynamics can alternatively be seen as the analogue of the Langevin equation for a particle moving with velocities of constant Euclidean norm. For this model we then finally obtain again ergodicity with an explicit rate of convergence. In particular, we are able to generalize the result from \cite{GK08} to arbitrary dimensions.

It is interesting to note that we require conditions in the abstract setting similar to the assumptions made in the hypocoercivity setting of Dolbeault, Mouhot and Schmeiser (see \cite[Sec.~1.3]{DMS13}) which has itself been extended later on by the authors of the underlying article, see \cite{GS12B}. Moreover, the hypocoercivity approach is stronger than the ergodicity method in the sense that it implies an exponential rate of convergence of the semigroup in the $L^2$-space directly; the ergodicity approach instead describes the time averages of the dynamics. However, the mentioned hypocoercivity setting (and other similar ones) do not apply in such singular situations so far, see Remark \ref{Rm_comparison_ergodicity_exponenential_rate} below. Moreover, up to the best of the author's knowledge, the approach developed in the underlying article seems to be the first abstract ergodicity method in the existing literature which allows to cover relevant particle systems with singular interactions arising for instance in Mathematical Physics. Although the conditions in the abstract hypocoercivity method mentioned above are similar to the conditions required in our abstract ergodicity method, we emphasize that both methods are complementary to each other.

This paper is organized as follows. In Section \ref{section_Results} we present the ergodicity method in the abstract setting and discuss the two applications mentioned above. All proofs, however, are postponed to Section \ref{Regular sub-Markovian semigroups and associated laws}. Therefore, basic definitions and notations for understanding the framework are shortly explained within Section \ref{section_Results}. Further details and complete definitions are then given in Section \ref{Regular sub-Markovian semigroups and associated laws} in full detail. We finally mention that the results in this article are obtained from the PhD thesis of the second named author; see \cite[Ch.~3]{Sti14}.

Altogether, the main results obtained in this article are summarized as follows.
\begin{itemize} 
\item Developing a new abstract strategy for proving ergodicity with an explicit rate of convergence for diffusions associated with a degenerate non-coercive Kolmogorov evolution operator $L$. The important point is that $L$ may have singular nonsmooth coefficients, see Theorem \ref{main_ergodicity_theorem_with_rate_of_convergence} and Theorem \ref{Cor_main_ergodicity_theorem_with_rate_of_convergence}.
\item The method applies to degenerate and singular particle systems\index{singular particle system, singular\\$N$-particle Langevin dynamics,\\singular diffusion}\footnote{When using the expression singular particle system (or analogously singular $N$-particle Langevin dynamics or singular diffusion) the word singular refers to a singular interaction potential.} arising in Mathematical Physics. Methods in the existing literature (e.g.~based on hypoellipticity, hypocoercivity or Lyapunov type techniques) do not apply in this situation so far.
\item Applying the framework to the $N$-particle Langevin dynamics with singular interaction potentials, see Theorem \ref{Ergodicity_theorem_N_particle_Langevin}. This reproduces the result from \cite{CG10} and provides a further specification for the rate of convergence.
\item Applying the method to the fiber lay-down dynamics, see Theorem \ref{Ergodicity_theorem_generalized_fiber_lay_down}. This generalizes the result from \cite{GK08} to arbitrary dimensions.
\end{itemize}

\section{Overview of the results} \label{section_Results}

As described in the introduction, in this section we present all our results. We start with the abstract strategy for proving ergodicity (with explicit computable rate of convergence) for singular degenerate Kolmogorov diffusions and discuss the applications afterwards. All proofs are postponed to Section \ref{Regular sub-Markovian semigroups and associated laws}. Basic definitions and notations for understanding the whole framework are shortly explained within the underlying section. Further details and complete definitions are then given in Section \ref{Regular sub-Markovian semigroups and associated laws}. 

\subsection{The abstract ergodicity method} \label{Results_abstract_ergodicity_method}

Some comments on the notations: In the following, all considered operators are assumed to be linear.  We assume that the reader is familiar with basic definitions and statements concerning the theory of operator semigroups. Beautiful references on the subject are e.g.~\cite{Paz83} or \cite{Gol85}. Below a strongly continuous contraction semigroup is always abbreviated by s.c.c.s. 
Now, as mentioned in the introduction, the ergodicity method presented below is the generalization of the method from \cite{GK08} (and from \cite{CG10}) to an abstract setting.\index{$L$, $S$, $A$, $\hat{L}$, $D$, $P$, $\mu$, $\mathbb{P}$, $\hat{\mathbb{P}}$ in\\ergodicity setting}

\begin{DynAn} \index{(A) Analytic dynamics ergodicity\\method}
We require the following conditions for the analytic part of the dynamics, i.e., the underlying Kolmogorov operator. \\[\Breite]
\begin{tabularx}{\BreiteTabelle}{lX}
(A1) & \textit{State space:} $E$ is a separable metric space equipped with its Borel $\sigma$-algebra $\mathcal{B}(E)$. Let $\mu$ be a probability measure on $(E,\mathcal{B}(E))$. The Hilbert space $H$ is defined as
\begin{align*}
H\df L^2(E,\mu)
\end{align*}
endowed with the usual scalar product $(\cdot,\cdot)_H$ and induced norm  $\|\cdot\|_H$ or $\|\cdot\|$.\\[\Breite]
(A2) & \textit{The Kolmogorov generator:}  Let $D \subset H$ be a dense linear subspace of $H$ which is an algebra. Let $(L,D)$ be a linear operator on $H$ of the form 
\begin{align*}
L=S-A \quad \mbox{on }D
\end{align*}
where $(S,D)$ is a symmetric and nonpositive definite operator and $(A,D)$ is an antisymmetric operator on $H$.\\[\Breite]
(A3) & \textit{Invariant measure:} Let $\mu$ be an invariant measure for $(L,D)$ and for $(\hat{L},D)$.\index{invariant measure} This means that 
\begin{align*} 
\int_E Lf\, \mathrm{d}\mu=0=\int_E \hat{L}f\, \mathrm{d}\mu \quad \mbox{for all $f \in \D$}. \smallskip
\end{align*} \\[\BreiteDrei]
\end{tabularx}
\end{DynAn}

Above $(\hat{L},D)$ with $\hat{L}=S+A$ denotes the adjoint of $(L,D)$ on $D$ in $H$. Via assuming (A2) note that (A3) equivalently means that $\mu$ should be invariant for $(S,D)$ and invariant for $(A,D)$. Note that the previously introduced operators with predomain $D$ are closable, since they are densely defined and dissipative. The closures\index{$(L,\overline{D}^L)$, $(S,\overline{D}^S)$, $(A,\overline{D}^A)$ etc.} of these operators on $H$ with predomain $D$ are denoted by 
\begin{align*}
(\hat{L},\overline{D}^{\hat{L}}),~~(L,\overline{D}^L), ~~(S,\overline{D}^S)~~\mbox{and}~~(A,\overline{D}^A).
\end{align*}
Furthermore, the orthogonal projection on the kernel $\mathcal{N}(S)$ of $(S,\overline{D}^S)$ is denoted by
\begin{align*}
P \colon H \to \mathcal{N}(S).
\end{align*}


\phantomsection\label{pageref_Conditions_analytic_A} Before introducing the assumptions concerning the underlying stochastic part of the dynamics we need some more notations, see also Section \ref{subsection_laws} and Section \ref{subsection_laws_and_associated_semigroups} for further details. So, assume the situation from (A). Let $\mathbb{P}$ be a probability law\index{probability law $\mathbb{P}$} on $C([0,\infty);E)$, this is, a probability measure on $C([0,\infty);E)$ where $C([0,\infty);E)$\index{$C([0,\infty);E)$} denotes the space of continuous paths on $[0,\infty)$ taking values in $E$. Assume that $\mathbb{P}$ admits $\mu$ as invariant measure, i.e., $\mathbb{P} \circ X_t^{-1}=\mu$\index{invariant measure of a law} for all $t \geq 0$ where $X_t$ denotes the evaluation of paths at time $t$. For each $f \in D$ introduce $M^{[f],L}\df(M_t^{[f],L})_{t \geq 0}$\index{$M^{[f],L}$} by
\begin{align*}
M_t^{[f],L}\df f(X_t)-f(X_0) - \int_0^t Lf(X_s) \d s \quad \mbox{for all }t \geq 0.
\end{align*}
Furthermore, define $N^{[f],L}\df(N_t^{[f],L})_{t \geq 0}$\index{$N^{[f],L}$} by
\begin{align*}
N_t^{[f],L}\df\left(M_t^{[f],L}\right)^2 - 2 \int_0^t \Gamma_L(f,f)(X_s) \d s\quad \mbox{for all } t \geq 0,~f \in D.
\end{align*}
Here the \textit{carr\'{e} du champ}\index{carr\'{e} du champ/ square\\field operator $\Gamma_L$} $\Gamma_L \colon D \times D \to L^1(\mu)$ (or the \textit{square-field operator}) is given by 
\begin{align*}
\Gamma_L(f,g) \df \frac{1}{2} \left(L(fg) - f Lg - g Lf\right) \quad  \mbox{for all }f,g \in D.
\end{align*}
We remark that $g(X_t)$ and $\int_0^t g(X_s)\, \mathrm{d}s$ are $\mathbb{P}$-a.s.~well-defined (i.e., independent of the $\mu$-version one chooses for g), $\mathbb{P}$-integrable and $\mathcal{F}^0_t$-measurable for each $g \in L^1(\mu)$ and each $t \geq 0$, see Lemma \ref{Lm_well_definedness_integrability} below. Here $(\mathcal{F}^0_t)_{t \geq 0}$ is the elementary filtration\index{elementary filtration $(\mathcal{F}_t^0)_{t \in I}$} generated by the paths. Thus all the terms in $M_t^{[f],L}, N_t^{[f],L}$ are in particular $\mathbb{P}$-a.s.~well-defined for $f \in D$, $t \geq 0$. Moreover, it follows that $M^{[f],L}, N^{[f],L}$ are $(\mathcal{F}^0_t)_{t \geq 0}$-adapted, $\mathbb{P}$-integrable and $M^{[f],L}$ is even square integrable for all $f \in D$, see again Lemma \ref{Lm_well_definedness_integrability}. Finally, the same definitions can be introduced and the same properties are satisfied in case $(L,D)$ is replaced by $(\hat{L},D)$ above. Now the stochastic assumption reads as follows.

\begin{DynSt} In\index{(S) Stochastic dynamics ergodicity\\method} the situation from (A) we assume the following condition.\\[\Breite]
\begin{tabularx}{\BreiteTabelle}{lX}
(S) & \textit{Stochastic dynamics and the martingale problem:} Let $\mathbb{P}$ and $\hat{\mathbb{P}}$ be probability laws on $C([0,\infty);E)$ having $\mu$ as invariant measure such that 
\begin{align*}
\hat{\mathbb{P}}_{T} = {\mathbb{P}}_{T} \circ \tau_T^{-1}\quad  \mbox{for all } T \geq 0.
\end{align*}
Assume that $N^{[f],L}$ is an $(\mathcal{F}^0_t)_{t \geq 0}$-martingale under $\mathbb{P}$ and that $N^{[f],\hat{L}}$ is an $(\mathcal{F}^0_t)_{t \geq 0}$-martingale under $\hat{\mathbb{P}}$ for all $f \in D$.
\end{tabularx}
\end{DynSt} 

Here $\mathbb{P}_{T}$\index{$\mathbb{P}_T$} and $\hat{\mathbb{P}}_{T}$ denote the image laws of $\mathbb{P}$ and $\hat{\mathbb{P}}$, respectively, w.r.t.~the restriction of paths to $C([0,T];E)$ for $T \geq 0$ and $\tau_T$ is the time-reversal\index{time-reversal $\tau_T$} on $C([0,T];E)$, see Section \ref{Regular sub-Markovian semigroups and associated laws} for definitions. With the previous assumptions at hand one obtains the following corollary. Expectation w.r.t.~$\mathbb{P}$ and $\hat{\mathbb{P}}$ is denoted by $\mathbb{E}$ and $\hat{\mathbb{E}}$, respectively. 

\begin{Cor} \label{Pp_computation_quadratic_variation} Assume the situation from (A) and (S). Let $T \geq 0$. Then 
\begin{align*}
\mathbb{E}\left[ \left( M_t^{[f],L} \right)^2 \right] = -2 \,t\, \left(Sf,f\right)_H = \hat{\mathbb{E}}\left[ \left( M_t^{[f],\hat{L}} \right)^2 \right] \quad \mbox{for all } f \in D \mbox{ and all }t \geq 0. 
\end{align*}
\end{Cor}

This corollary is proven on page \pageref{proof_Corollary_quadratic_variation}. We remark that it seems also natural to assume additionally that $M^{[f],L}$ and $M^{[f],\hat{L}}$, $f \in D$, are $(\mathcal{F}^0_t)_{t \geq 0}$-martingales under $\mathbb{P}$, or $\hat{\mathbb{P}}$ respectively. However, this property is not used in the proofs for the abstract setting below and is therefore not explicitly required. Nevertheless, we emphasize that one usually needs it in order to be able to verify that $N^{[f],L}$ and $N^{[f],\hat{L}}$ are indeed martingales in concrete applications; we further remark that these martingale problems are usually satisfied in case the Kolmogorov operator $L$ is associated with a manifold-valued Stratonovich SDE. However, such a classical stochastic approach requires smooth or at least continuous type assumptions on the coefficients of the operator $L$. Since we are interested in more general situations, in particular in singularly distorted diffusions, we choose another approach which is heavily based on Functional Analysis. Therefore, the laws $\mathbb{P}$ and $\hat{\mathbb{P}}$ are constructed in our applications via using modern tools from the theory of generalized Dirichlet forms or by using existence results from \cite{BBR06}. In these cases the laws are associated with a conservative sub-Markovian s.c.c.s. Such an associatedness property then implies that the martingale problem is fulfilled automatically, see Theorem \ref{Thm_Martingale_Problem_via_regular semigroups} below. Before stating it, let us introduce the following sufficient conditions.

\begin{DynAltern} \index{(A)', (S)' alternative dynamics\\ergodicity method}
{~}\\[\Breite]
\begin{tabularx}{\BreiteTabelle}{lX}
(A)' & \textit{The semigroup:} Assume the situation of (A1) and (A2) and let $(T_{t,2})_{t \geq 0}$ be a $\mu$-invariant sub-Markovian s.c.c.s.~on $H$ (which is then also conservative and regular) with associated generator $(L_2,D(L_2))$ which extends $(L,D)$ and further assume that its adjoint $(\hat{L}_2,D(\hat{L}_2))$ on $H$ extends $(\hat{L},D)$. \\[\Breite]
(S)' & \textit{Stochastic dynamic and associatedness with the semigroup:} Let $\mathbb{P}$ and $\hat{\mathbb{P}}$ be probability laws on $C([0,\infty);E)$ such that $\hat{\mathbb{P}}_{T} = {\mathbb{P}}_{T} \circ \tau_T^{-1}$ for all $T \geq 0$ and assume that $\mathbb{P}$ is associated with $(T_{t,2})_{t \geq 0}$.
\end{tabularx}
\end{DynAltern}

In the previous assumptions sub-Markovian\index{semigroup!sub-Markovian} means that $0 \leq T_{t,2}f \leq 1$ for all $t \geq 0$ whenever $0 \leq f \leq 1$ and conservativity\index{semigroup!conservative} means that $T_{t,2}1=1$ for all $t \geq 0$. Moreover, $\mu$-invariance\index{semigroup!$\mu$-invariance} is defined as\label{Def_conservativity_mu_invariance}
\begin{align*}
\mu(T_{t,2}f) = \mu(f) \quad \mbox{for all } t \geq 0 \mbox{ and all }f \in L^2(E,\mu).
\end{align*}
Moreover, regularity means that the associated adjoint s.c.c.s.~on $H$ is assumed to be sub-Markovian as well and finally, associatedness of $\mathbb{P}$ w.r.t.~$(T_{t,2})_{t \geq 0}$ means that for all nonnegative $f_1,\ldots,f_n \in L^\infty(E,\mu)$, $0 \leq t_1 \leq \cdots \leq t_n$, $n \in \mathbb{N}$, it holds\index{associatedness!semigroup and law}
\begin{align*}
\mathbb{E}\left[f_1(X_{t_1}) \cdots f_n(X_{t_n})\right]=\mu\left( T_{t_1,2} \left( f_1 T_{t_2-t_1,2} \left(f_2 \cdots T_{t_{n-1}-t_{n-2},2} \left( f_{n-1} T_{t_{n}-t_{n-1},2} f_n\right)\right)\right)\right).
\end{align*}
Let us mention that a law which is associated with the semigroup $(T_{t,2})_{t \geq 0}$ from (A)' is already unique. Now one obtains the following theorem. For the proof see page \pageref{proof_Theorem_A_strich_und_s_strich}.

\begin{Thm} \label{Thm_Martingale_Problem_via_regular semigroups}
Assume the situation from (A)' and (S)'. Then Condition (A3) is fulfilled, $\hat{\mathbb{P}}$ is associated with the dual semigroup $(\hat{T}_{t,2})_{t \geq 0}$ of $(T_{t,2})_{t \geq 0}$ on $H$ (which is also a regular conservative $\mu$-invariant sub-Markovian s.c.c.s.) and both laws $\mathbb{P}$ and $\hat{\mathbb{P}}$ admit $\mu$ as invariant measure. Finally, $M^{[f],L}$ and $N^{[f],L}$ are $(\mathcal{F}^0_t)_{t \geq 0}$-martingales under $\mathbb{P}$ for all $f \in D$ and moreover, $M^{[f],\hat{L}}$ and $N^{[f],\hat{L}}$ are $(\mathcal{F}^0_t)_{t \geq 0}$-martingales under $\hat{\mathbb{P}}$ for all $f \in D$. In particular, Conditions (A) and (S) are satisfied.
\end{Thm}

For the rest of this section we assume Conditions (A) and (S) without further mention them again. Let us introduce now the first ergodicity condition.

\begin{E1} (Microscopic coercivity and microscopic dynamic)\index{Microscopic coercivity} First let $(S,D)$ be essentially selfadjoint on $H$.\index{(E1)-(E4) ergodicity conditions} Furthermore, assume that there exists a constant $\Lambda_m > 0$ \index{$\Lambda_m$, $\Lambda_M$, $c_1$, $c_2$, $c_3$ (constants\\ergodicity method)}such that
\begin{align*}
-\left(Sf,f\right)_H \geq \Lambda_m \,\|(I-P)f\|^2 \quad \mbox{for all }f \in D.
\end{align*}
\end{E1}

By (E1) and the fact that $\mu$ is invariant w.r.t.~$(S,D)$, note that one obtains the conservativity condition $1 \in \overline{D}^S$ and $S1=0$. By assuming (E1) and using Proposition \ref{Pp_computation_quadratic_variation}, one can prove the following statement which will be one of the main ingredients to prove the final ergodicity theorem. The proof is given on  page \pageref{proof_main_ingredient_proposition}.

\begin{Pp} \label{Pp_property_from_microscopic}
Assume (E1). Let $f \in \mathcal{N}(S)^\bot$. Then it holds
\begin{align*}
\mathbb{E} \left[ \left( \frac{1}{t} \int_0^t f(X_s) \d s \right)^2 \right] \leq \frac{2}{t\, \Lambda_m} \|f\|^2.
\end{align*}
The same statement holds in case $\mathbb{E}$ is replaced by $\hat{\mathbb{E}}$.
\end{Pp}

Up to now, note that everything was completely symmetric, i.e., each statement and assumption for $(L,D)$ and $\mathbb{P}$ is formulated and satisfied also in the dual case for $(\hat{L},D)$ and $\hat{\mathbb{P}}$. In the following, formulations are given with preference on $(L,D)$. However, we emphasize that when formulating the conditions below (or more precisely, only (E2)) in the analogous way for $(\hat{L},D)$, then the final statements of Theorem \ref{main_ergodicity_theorem_with_rate_of_convergence} or Theorem \ref{Cor_main_ergodicity_theorem_with_rate_of_convergence} are satisfied in case $\mathbb{P}$ is replaced by $\hat{\mathbb{P}}$ therein. 

Let us go on by introducing first the following technical condition. Assume that
\begin{align} \label{Technical_Condition_Ergodicity}
P(D) \subset \overline{D}^A,\quad AP(D) \subset \overline{D}^A \cap \overline{D}^S \cap \overline{D}^L\quad \mbox{and} \quad L_{| AP(D)} =S-A.
\end{align}
First of all, we define \index{$D_P$, $H_P$}
\begin{align*}
D_P\df P(D)\quad  \mbox{and} \quad H_P\df P(H).
\end{align*}
$H_P$ is again a Hilbert space endowed with the scalar product of $H$. Due to \eqref{Technical_Condition_Ergodicity} we can introduce $G\colon D_P \to H_P$\index{$G=PA^2P$ (ergodicity method)} as
\begin{align*}
G\df PA^2P \quad \mbox{on } D_P.
\end{align*}
This means that $Gf=PA^2Pg$ if $f=Pg$, $g \in D$. This is clearly well-defined. Note that 
\begin{align} \label{Eq_cauchy_APf}
\left( Gf,f\right)_{H_P} = - \|APf \|^2\quad   \mbox{for all } f \in D_P.
\end{align}
So, $(G,D_P)$ is dissipative and densely defined on $H_P$, hence closable on $H_P$. Its closure on $H_P$ is denoted by $(G,\overline{D_P}^G)$. While almost all operators in this section are considered on $H$, we emphasize that $(G,\overline{D_P}^{G})$ is understood as an operator living on $H_P$. We shall mention that the operator $G$ in our concrete examples describes the macroscopic dynamics, i.e., can be obtained by using a suitable macroscopic scaling limit of $L$. Therefore, consider the applications below for interpretation. 

Below we need a Kato-boundedness condition of $LAP$ by $G$. Therefore, note that also $LAP \colon D_P \to H$ can be defined on $D_P=P(D)$ in the obvious way as
\begin{align*}
LAPf=LAPg \quad \mbox{for all } f=Pg \mbox{ with } g \in D.
\end{align*}
Then the desired Kato-boundedness condition reads as follows.

\begin{E2} (Kato-boundedness) Assume the technical condition from \eqref{Technical_Condition_Ergodicity}. \index{(E1)-(E4) ergodicity conditions}Assume that the operator $(LAP,D_P)$ is $G$-bounded on $D_P$. This means that there exists $c_1, c_2 \in [0,\infty)$\index{$\Lambda_m$, $\Lambda_M$, $c_1$, $c_2$, $c_3$ (constants\\ergodicity method)} such that
\begin{align*}
\|LAPf\| \leq c_1 \| Gf \|  + c_2 \|f\| \quad \mbox{for all } f \in D_P.
\end{align*}
\end{E2}

Via this condition we can extend the operator $(LAP,D_P)$ to $\overline{D_P}^G$ as follows. Therefore, assume (E2) and let $f \in \overline{D_P}^G$. Thus there exists $f_n \in D_P$, $n \in \mathbb{N}$, such that 
\begin{align*}
f_n \to f \quad \mbox{and}\quad Gf_n \to Gf \quad \mbox{as } n \to \infty
\end{align*}
with convergence in $H$. So, (E2) in particular implies that $LAPf_n$, $n \in \mathbb{N}$, is a Cauchy sequence in $H$. We define\index{$LAP$ and $\left[LAP\right]$}
\begin{align*}
\left[LAP\right]f\df \lim_{n \to \infty} LAPf_n \in H.
\end{align*}
This is independent of the choice of $(f_n)_{n \in \mathbb{N}}$, so well-defined. Note that $(\left[LAP\right],\overline{D_P}^G)$ extends $(LAP,D_P)$. Now we introduce the next condition. 

\begin{E3} (Macroscopic coercivity and macroscopic dynamic) \index{Macroscopic coercivity}Assume \eqref{Technical_Condition_Ergodicity} and assume \index{(E1)-(E4) ergodicity conditions}that $(G,D_P)$ is essentially selfadjoint on $H_P$. Furthermore, assume that there exists a constant $\Lambda_M > 0$\index{$\Lambda_m$, $\Lambda_M$, $c_1$, $c_2$, $c_3$ (constants\\ergodicity method)} such that
\begin{align*}
-\left(Gf,f\right)_{H} \geq \Lambda_M \,\|f-\left(f,1\right)_H\|^2_{H} \quad \mbox{for all }f \in D_P.
\end{align*}
Finally, assume that $1 \in H_P$ and even $1 \in \overline{D_P}^G$, $G1=0$.
\end{E3}

Here the requirement $1 \in H_P$ is necessary in order to guarantee that $1 \in \overline{D_P}^G$ makes sense. Note that $1 \in H_P$ means $P1=1$, this is, $1 \in \overline{D}^S$ and $S1=0$. As a consequence of (E1)-(E3) we get the following lemma,  essentially used in order to prove Theorem \ref{main_ergodicity_theorem_with_rate_of_convergence}.

\begin{Lm} \label{Lm_after_H4}
\begin{itemize}
\item[(i)]
Assume (E1) and (E2). Then for all $g \in \overline{D_P}^G$ we have
\begin{align}  \label{relation2_G_LAP}
-Gg = P\left[LAP\right]g \quad \mbox{and} \quad \left(\left[LAP\right]g,Gg\right)_H=-\|Gg\|^2_H.
\end{align}
\item[(ii)] Assume (E1) and (E3). Then we have
\begin{align*}
f \in \mathcal{N}(G)^\bot \quad \mbox{for all } f \in H_P \mbox{ with } \left(f,1\right)_{H}=0.
\end{align*}
\end{itemize}
\end{Lm}

This is proven on page \pageref{proof_Lemma_before_maintheorem}. We finally arrive at the desired ergodicity theorem which gives a concrete rate of convergence. The proof can be found on page \pageref{proof_main_theorem_nr1}.

\begin{Thm} \label{main_ergodicity_theorem_with_rate_of_convergence}
In the situation from (A) and (S), assume Conditions (E1)-(E3) with the constants $\Lambda_m$, $\Lambda_M$, $c_1$ and $c_2$. Let $f \in L^2(E,\mu)$ be arbitrary and let $t >0$. We obtain ergodicity with rate of convergence
\begin{align*}
\left\| \frac{1}{t} \int_0^t f(X_s) \d s  - \mathbb{E}_\mu\left[f\right] \right\|_{L^2(\mathbb{P})} &\leq  \sqrt{2} \, \left(\frac{1}{t} \, \kappa_1 + \frac{1}{\sqrt{t}} \, \kappa_2 \right) \left\| f(X_0)-\mathbb{E}_\mu\left[f\right] \right\|_{L^2(\mathbb{P})}
\end{align*}
where
\begin{align*}
\kappa_1=\frac{\sqrt{2}}{\sqrt{\Lambda_M}}\quad \mbox{and} \quad \kappa_2= \frac{c_1+1}{\sqrt{\Lambda_m}} + \frac{c_2}{\sqrt{\Lambda_m}\, \Lambda_M} + \sqrt{\frac{c_1}{\sqrt{\Lambda_M}} + \frac{c_2}{\Lambda_M\,\sqrt{\Lambda_M}}}\,.
\end{align*}
Here $\mathbb{E}_\mu\left[f\right]=\int_E f \d \mu$\index{$\mathbb{E}_\mu[\cdot]$}. 
\end{Thm}

The rate of convergence can even further be specified with the help of an algebraic relation as introduced next.

\begin{E4} (Algebraic relation) Assume \eqref{Technical_Condition_Ergodicity}. Assume that there exists  $c_3 \geq 0$\index{$\Lambda_m$, $\Lambda_M$, $c_1$, $c_2$, $c_3$ (constants\\ergodicity method)} such that\index{(E1)-(E4) ergodicity conditions} 
\begin{align*}
\left(SAPf,APf\right)_H=c_3 \left(Gf,f\right)_H \quad \mbox{for all }f \in D_P.
\end{align*}
This is fulfilled for instance if $SAP(D) \subset \overline{D}^A$ and 
\begin{align*}
PA\,SAP=- c_3\,PA^2P\quad \mbox{on }D.
\end{align*}
\end{E4}

Then from the proof of Theorem \ref{main_ergodicity_theorem_with_rate_of_convergence} we directly obtain the upcoming corollary, see page \pageref{proof_main_theorem_corollary} for details.

\begin{Thm} \label{Cor_main_ergodicity_theorem_with_rate_of_convergence}
Additionally to the assumptions from Theorem \ref{main_ergodicity_theorem_with_rate_of_convergence} assume that (E4) holds with the respective constant $c_3$. Let $f \in L^2(E,\mu)$ be arbitrary and let $t >0$. We obtain ergodicity with rate of convergence
\begin{align*}
\left\| \frac{1}{t} \int_0^t f(X_s) \d s  - \mathbb{E}_\mu\left[f\right] \right\|_{L^2(\mathbb{P})} &\leq  \sqrt{2} \left( \frac{1}{t} \, \kappa_1 + \frac{1}{\sqrt{t}} \, \kappa_2 \right) \left\| f(X_0)-\mathbb{E}_\mu\left[f\right] \right\|_{L^2(\mathbb{P})}
\end{align*}
where the constants $\kappa_1$ and $\kappa_2$ can further be specified as
\begin{align*}
\kappa_1=\frac{\sqrt{2}}{\sqrt{\Lambda_M}}\quad \mbox{and} \quad \kappa_2= \frac{c_1+1}{\sqrt{\Lambda_m}} + \frac{c_2}{\sqrt{\Lambda_m}\, \Lambda_M} + \frac{\sqrt{c_3}}{\sqrt{\Lambda_M}}.
\end{align*}
\end{Thm}

Theorem \ref{main_ergodicity_theorem_with_rate_of_convergence} and Theorem \ref{Cor_main_ergodicity_theorem_with_rate_of_convergence} is proven in Section \ref{subsection_proofs_to_abstract_setting}. However, to give an idea how everything fits together, let us shortly sketch the proof. First it is easy to see that w.l.o.g.~one may assume that $\left(f,1\right)_H=0$. Then one can decompose $f$ in the form
\begin{align*}
f=f-Pf +Pf - \left[LAP\right]g +\left[LAP\right]g.
\end{align*}
Here $g \in \overline{D_P}^G$ can be chosen with the help of Lemma \ref{Lm_after_H4}\,(ii) such that $Pf=-Gg$. Then again by Lemma \ref{Lm_after_H4}\,(i) we have $P\left[LAP\right]g=-Gg$. So $P(f-Pf)=0$ and $P(Pf- \left[LAP\right]g )=0$. Thus the terms $f-Pf$ and $Pf- \left[LAP\right]g$ can be estimated with the help of Proposition \ref{Pp_property_from_microscopic}. Finally, the last term $\left[LAP\right]g$ can afterwards be estimated using Corollary \ref{Pp_computation_quadratic_variation} and a suitable approximation. This then yields the desired rate of convergence in terms of $\kappa_1$ and $\kappa_2$. Thus the overview of our abstract ergodicity method is completed. 

\begin{Rm} We mention that the setting developed here may also be called or considered as a setting which allows to discuss convergence to equilibrium of \textit{non-reversible} diffusions, see \cite{LNP13} for the terminology.
\end{Rm}

The results for our applications are summarized in the upcoming subsection. Let us conclude with a final remark by comparing our rate of convergence with a possible exponential rate in $L^2(\mu)$ of the semigroup $(T_{t,2})_{t \geq 0}$ when (A)' and (S)' are assumed and let us describe some advantages of our method.

\begin{Rm} \label{Rm_comparison_ergodicity_exponenential_rate} 
Assume the conditions from Theorem \ref{Thm_Martingale_Problem_via_regular semigroups} and (E1) up to (E4).  In our applications below, the semigroup $(T_{t,2})_{t \geq 0}$ always admits a stochastic representation as the transition kernel of a $\mu$-standard right process\index{$\mu$-standard right process} $\mathbf{M}$ having continuous sample paths and infinite lifetime. In these situations, one obtains a family of probability measures $(\mathbb{P}_{x})_{x \in E}$ associated\index{associatedness!with a right process} with the right process and satisfying $T_{t,2} f (x) = \mathbb{E}_{x} [f(X_t)]$ for $\mu$-a.e.~$x \in E$ and each $f \in B_b(E)$ (i.e., $f$ is bounded measurable real-valued function), $t \geq 0$. $\mathbf{M}$ is also said to be associated with $(T_{t,2})_{t \geq 0}$, see \cite[Def.~2.2.7]{Con11} for the precise definition. The desired law $\mathbb{P}$ in (S) is then constructed as the law of the right process under the probability measure $\int_E \mathbb{P}_{x}\, \mathrm{d}\mu(x)$. Here $\mathbb{E}_x$ denotes expectation w.r.t.~$\mathbb{P}_x$. We refer to \cite[Sec.~3]{CG10}, \cite[Ch.~3]{Con05} and especially to the reference \cite[Ch.~2, Ch.~6]{Con11} where the notations (including measurability issues) are made precise. Then, by using Fubini's theorem and Jensen's inequality, the estimate
\begin{align*}
\left\| \int_0^t T_{s,2} f\, \mathrm{d}s \right\|^2_{L^2(\mu)} \leq \left\| \int_0^t f(X_s) \,\mathrm{d}s \right\|^2_{L^2(\mathbb{P})},\quad f \in L^2(E,\mu),
\end{align*}
can easily be derived in our applications of interest. Hence, by the invariance of $\mu$ w.r.t.~$\mathbb{P}$ we even obtain mean ergodicity of the semigroup with the same rate of convergence as in Theorem \ref{Cor_main_ergodicity_theorem_with_rate_of_convergence}, i.e., for all $f \in L^2(\mu)$ and $t >0$ we have
\begin{align} \label{exponen_rat_convergence_semigroup_comparision_ergodicity_rate_1}
\left\| \frac{1}{t} \int_0^t T_{s,2}f \d s  - \mathbb{E}_\mu\left[f\right] \right\|_{L^2(\mu)} \leq  \sqrt{2} \left( \frac{1}{t} \, \kappa_1 + \frac{1}{\sqrt{t}} \, \kappa_2 \right) \left\| f-\mathbb{E}_\mu\left[f\right] \right\|_{L^2(\mu)}
\end{align}
with the constants $\kappa_1,\kappa_2$ as explicitly specified in Theorem \ref{Cor_main_ergodicity_theorem_with_rate_of_convergence}.

Now, for the moment, assume that the semigroup $(T_{t,2})_{t \geq 0}$ would admit an exponential rate of convergence in $L^2(\mu)$ as can be obtained e.g.~in existing hypocoercivity methods; see \cite{DMS13} or \cite{GS12B}. This means that there exists $\nu_1, \nu_2 \in (0,\infty)$ such that for each $f \in L^2(\mu)$ 
\begin{align*} 
\left\|T_{t,2}f - \mathbb{E}_\mu\left[f\right] \right\|_{L^2(\mu)} \leq \nu_1 e^{-\nu_2 \,t}  \left\|f - \mathbb{E}_\mu\left[f\right]\right\|_{L^2(\mu)} \quad \mbox{for all } t \geq 0.
\end{align*}
Assuming the latter we directly infer that for each $t > 0$ we have
\begin{align} \label{exponen_rat_convergence_semigroup_comparision_ergodicity_rate_2}
\left\|\frac{1}{t} \int_0^t T_{s,2}f \, \mathrm{d}s - \mathbb{E}_\mu\left[f\right] \right\|_{L^2(\mu)} \leq \frac{1}{t} \,\frac{\nu_1}{\nu_2}  \left\|f - \mathbb{E}_\mu\left[f\right]\right\|_{L^2(\mu)}.
\end{align}
So, following the vocabulary used in \cite{CG10}, our rate may be called \grqq $L^2$-\textit{exponentially ergodic}\index{$L^2$-exponentially ergodic} in the sense that the convergence rate \eqref{exponen_rat_convergence_semigroup_comparision_ergodicity_rate_1} corresponds to (but apparently not implies) exponential convergence of the semigroup.\grqq~Moreover, by comparing \eqref{exponen_rat_convergence_semigroup_comparision_ergodicity_rate_1} and \eqref{exponen_rat_convergence_semigroup_comparision_ergodicity_rate_2} we see that a possible exponential rate of convergence of the semigroup in $L^2(\mu)$ does not imply a much better rate than can be achieved by our ergodicity method. 

Moreover, we shall remark that our ergodicity rate gives a concrete quantitative description of the constants occurring in the rate of convergence. Up to the best of the author's knowledge, such explicit quantitative descriptions of the rate have not yet been obtained in related abstract methods for analyzing the relaxation to equilibrium of degenerate evolution equations. Finally, and this is of course the most important point, the main advantage of the ergodicity method is that it even applies in singular situations that arise in studying e.g.~$N$-particle Langevin systems in Statistical Mechanics, see Section \ref{Results_Applications_N_particle_Langevin}. In this situation namely, known results in literature on the relaxation to equilibrium of the Langevin dynamics in $L^2(\mu)$ using tools about hypoellipticity or hypocoercivity (see e.g.~\cite[Theo.~6.4]{HN05}, \cite[Theo.~10]{DMS13} or \cite[Theo.~1]{GS14}) are valid only under nonsingular and partly smooth type assumptions on the underlying potential $\Phi$. However, nonsmooth singular potentials are allowed in the ergodicity theorems below, see Theorem \ref{Ergodicity_theorem_N_particle_Langevin} and Remark \ref{Rm_Ergodicity_theorem_N_particle_Langevin}\,(iii).

\end{Rm}


\subsection{Application to the N-particle Langevin dynamics with singular potentials} \label{Results_Applications_N_particle_Langevin}

Let $d,N \in \mathbb{N}$. In the first example we consider the $N$-particle Langevin dynamics\index{N-particle Langevin dynamics} with singular potentials as constructed and analyzed in \cite{CG10} and \cite{Con11}. We recall shortly the setting and framework from \cite{CG10}, or \cite[Ch.~6]{Con11} equivalently. Consider the latter references for further motivation and interpretation. The underlying dynamics is given by the stochastic differential equation in $(\mathbb{R}^d)^N \times (\mathbb{R}^d)^N$
\begin{align} \label{Eq_N_particle_Langevin}
&\mathrm{d} x_t = \omega_t \, \mathrm{d}t \\
&\mathrm{d} \omega_t = -\alpha \, \omega_t \, \mathrm{d}t - \nabla \Psi (x_t)\, \mathrm{d}t + \sqrt{\frac{2\alpha}{\beta}}\, \mathrm{d}W_t. \nonumber
\end{align}
The constants $\alpha$ and $\beta$ are assumed to be strictly positive. $x_t=(x^{(1)}_t,\cdots,x^{(N)}_t)\in (\mathbb{R}^d)^N$ denotes the position and $\omega_t=(\omega^{(1)}_t,\cdots,\omega^{(N)}_t) \in (\mathbb{R}^d)^N$ the velocity variable of $N$ particles moving in $\mathbb{R}^d$. We shortly write $\mathbb{R}^{dN}$ instead of $(\mathbb{R}^d)^N$. $\Psi \colon \mathbb{R}^{dN} \to\mathbb{R} \cup \{\infty\}$ is the $N$-particle interacting potential. The $\mathbb{R}^{dN}$-valued standard Brownian motion $(W_t)_{t \geq 0}$ describes the stochastic perturbation of the particles and the first term in the velocity equation means friction. For the physical background see \cite[Ch.~8]{Sch06}, \cite{CKW04} or \cite{Ris89}. For notational convenience and in view of having a convenient expression of the invariant measure later on, we redefine the potential again via setting
\begin{align*}
\Phi\df \beta \, \Psi.
\end{align*}
The Kolmogorov generator\index{Kolmogorov generator!$N$-particle Langevin dynamics} associated to \eqref{Eq_N_particle_Langevin} is  now given (at first formally) by
\begin{align} \label{Eq_generator_N_particle_Langevin}
L=\omega \cdot \nabla_x - \alpha~ \omega \cdot \nabla_\omega - \frac{1}{\beta}\,\nabla_x \Phi \cdot \nabla_\omega + \frac{\alpha}{\beta}\, \Delta_\omega
\end{align}
Here $\cdot$ denotes the standard Euclidean scalar product, $\nabla_x$ and $\nabla_\omega$ the usual gradient operators in $\mathbb{R}^{dN}$ for the respective $x$- or $\omega$-direction and $\Delta_\omega$ is the Laplace-operator in $\mathbb{R}^{dN}$ in the $\omega$-direction. The measure $\mu_{\Phi,\beta}$\index{$\mu_{\Phi,\beta}$} is defined as
\begin{align*}
\mu_{\Phi,\beta} = \tfrac{1}{\sqrt{2\pi \beta^{-1}}^{Nd}} \, e^{-\Phi(x) - \beta \frac{\omega^2}{2}} \,\mathrm{d} x \otimes \mathrm{d} \omega = e^{-\Phi(x)} \,\mathrm{d} x \otimes  \nu_\beta .
\end{align*}
Above $\mathrm{d}x$ and $\mathrm{d}\omega$ denotes the Lebesgue measure in $\mathbb{R}^{dN}$, $\omega^2 \df \omega \cdot \omega$ and $\nu_\beta$ the normalized Gaussian measure on $\mathbb{R}^{dN}$ with mean $0$ and covariance matrix $\beta^{-1} I$\index{Gaussian measure}. Finally, the natural state space for the $N$-particle system is denoted by
\begin{align*}\index{$\widetilde{\mathbb{R}^{dN}}$}
E\df \widetilde{\mathbb{R}^{dN}}\times \mathbb{R}^{dN} \quad \mbox{where} \quad \widetilde{\mathbb{R}^{dN}} \df \left\{ x \in \mathbb{R}^{dN}~\big|~\Phi(x) < \infty \right\}.
\end{align*}
Then $L$ well-defined as a linear operator on $L^2(E,\mu_{\Phi,\beta})$ with predomain $D = C_c^\infty(E)$. Here $C_c^\infty$ always means the set of all infinitely often differentiable functions having compact support on the respective state space.

Starting with the generator $L$ from \eqref{Eq_generator_N_particle_Langevin}, in \cite{CG10} (or \cite[Ch.~6]{Con11}) non-exploding martingale solutions to \eqref{Eq_N_particle_Langevin} are constructed for a wide class of $N$-particle potentials $\Phi$ that are allowed to have singularities and discontinuous gradient forces. The longtime behavior is afterwards analyzed therein via considering the associated functional analytic objects. In particular, ergodicity with rate of convergence in this concrete setting has been proved making use of the method from \cite{GK08}.  

We do not present new results for this application. However, our aim is to show how the N-particle Langevin dynamics fits into our abstract method for proving ergodicity with rate of convergence. Assuming the same conditions as required to prove ergodicity in \cite{CG10} (or \cite[Ch.~6]{Con11} equivalently) we show that Assumptions (A), (S) and (E1) up to (E4) are indeed satisfied. Thus ergodicity with the associated rate of convergence follows automatically from our abstract results developed at the beginning of this section. Of course, this is expected since our method is the generalization of the concrete methods from \cite{CG10} and \cite{GK08} to the abstract setting. So, let us start first introducing the assumptions for the interacting particle potential as defined in \cite[Sec.~4.2]{CG10} or as in \cite[Sec.~6.6.2]{Con11} equivalently. They allow the construction of the analytic and stochastic part of the dynamics and read as follows. As explained in \cite{CG10}, w.l.o.g.~one may assume that $\alpha=\beta=1$. However, we stay a bit more general to see how these constants enter into the rate of convergence.

\begin{Pot}\index{(C0)-(C3) ergodicity assumptions\\$N$-particle Langevin}
We require the following conditions on our potential $\Phi$.
\begin{itemize}
\item[(i)] $\Phi \colon \mathbb{R}^{dN} \to \mathbb{R} \cup \{\infty\}$ is bounded from below and not identically $=\infty$. Moreover, $e^{-\Phi} \mathrm{d}x $ defines a probability measure on $(\mathbb{R}^{dN},\mathcal{B}(\mathbb{R}^{dN}))$.
\item[(ii)] $e^{-\Phi}$ is continuous on $\mathbb{R}^{dN}$. 
\item[(iii)] $\Phi$ is weakly differentiable on $\widetilde{\mathbb{R}^{dN}}$. Furthermore,  $\nabla_x \Phi \in L^2(\mathbb{R}^{dN}, e^{-\Phi}\mathrm{d}x)$. 
\end{itemize}
\end{Pot}

As outlined above, the following statement holds.

\begin{Pp} \label{Pp_(P)_implies_A_and_S_Langevin}
Let $\Phi$ satisfies (C0). Then the analytic and stochastic dynamical system assumptions (A) and (S) are fulfilled. Moreover, even Conditions (A)' and (S)' are satisfied. Details on the construction of the Langevin semigroup $(T_{t,2})_{t \geq 0}$ from (A)' are described in the upcoming remark.
\end{Pp}

\begin{Rm} \label{Rm_Details_construction_Langevin_semigroup}
We mention that $(L,C_c^\infty(E))$ is essentially m-dissipative on $L^1(E,\mu_{\Phi,\beta})$. Its closure generates a $\mu_{\Phi,\beta}$-invariant conservative sub-Markovian s.c.c.s.~$(T_{t,1})_{t \geq 0}$ on $L^1(E,\mu)$. This semigroup exactly corresponds to a regular sub-Markovian $\mu_{\Phi,\beta}$-invariant conservative s.c.c.s.~$(T_{t,2})_{t \geq 0}$ on $L^2(E,\mu_{\Phi,\beta})$ by Proposition \ref{Pp_regular_sccs} below. Now $(T_{t,2})_{t \geq 0}$ is precisely the $N$-particle Langevin semigroup from (A)'\index{Langevin semigroup\\on $L^2(\mu_{\Phi,\beta})$} which is associated with $\mathbb{P}$. This determines the law $\mathbb{P}$\index{law $\mathbb{P}$ associated to!Langevin dynamics} uniquely. Analogous statements are valid also for $(\hat{L},C_c^\infty(E))$, $(\hat{T}_{t,2})_{t \geq 0}$ and $\hat{\mathbb{P}}$. For all details, see \cite[Sec.~6.4]{Con11}.
\end{Rm}

Next we recall the specified conditions on the potential which are required in \cite{CG10} for proving ergodicity with rate of convergence. Analogously to \cite{CG10} we denote
\begin{align*}
H_{\Phi}\df L^2(\widetilde{\mathbb{R}^{dN}}, e^{-\Phi} \mathrm{d}x)=L^2(\mathbb{R}^{dN}, e^{-\Phi} \mathrm{d}x).
\end{align*}
and define $(G_{\Phi},D(G_{\Phi}))$\index{$(G_{\Phi},D(G_{\Phi}))$} to be the closure on $H_{\Phi}$\index{$H_\Phi$} of the operator 
\begin{align*}
\Delta_x -  \nabla_x \Phi \cdot \nabla_x
\end{align*}
with predomain $C_c^\infty(\widetilde{\mathbb{R}^{dN}})$. $H_{\Phi}$ is equipped with the standard scalar product. The ergodicity conditions from \cite{CG10} read as follows.

\begin{Ass2}\index{(C0)-(C3) ergodicity assumptions\\$N$-particle Langevin}
Now assume that $\Phi$ satisfies (C0). We further require the following conditions.\\[\Breite]
\begin{tabularx}{\BreiteTabelle}{lX}
(C1) & For all $1 \leq i,j \leq dN$ the operators $\partial_{x_{i}}\partial_{x_{j}}$ and $\left(\partial_{x_{i}} \Phi\right)  \,\partial_{x_{j}}$ are Kato-bounded on $C_c^\infty(\widetilde{\mathbb{R}^{dN}})$ by the operator $G_{\Phi}$ in $H_{\Phi}$. \\[\Breite]
(C2) & The operator $(G_{\Phi},C_c^\infty(\widetilde{\mathbb{R}^{dN})})$ is essentially selfadjoint in $H_{\Phi}$.\\[\Breite]
(C3) & $G_\Phi$ has a spectral gap\index{spectral gap}\footnote{The assumptions indeed imply a gap in the spectrum $\sigma(G_\Phi)$ of $G_\Phi$, see e.g.~\cite[Theo.~4.3]{KNR08}}, i.e., \index{$\text{\normalfont{gap}}(G_{\Phi})$}
{\begin{align*}
\text{\normalfont{gap}}(G_{\Phi}) \df \inf_{\stackrel{0 \not=f \in D(G_{\Phi})}{\left(1,f\right)_{H_{\Phi}}=0}} \frac{\left(-G_{\Phi} f,f\right)_{H_{\Phi}}}{\|f\|_{H_{\Phi}}^2} > 0.
\end{align*}}\\[\BreiteDrei]
\end{tabularx}
\end{Ass2}

Conditions (C1)-(C3) then indeed imply (E1) up to (E4). Consider Section \ref{subsection_proofs_to_Langevin_dynamics} for all verifications. Corollary \ref{Cor_main_ergodicity_theorem_with_rate_of_convergence} then finally gives the ergodicity theorem  for the $N$-particle Langevin dynamics. We mention again that the statement is not new and already proven in \cite{CG10} or \cite{Con11}. However, it now follows from our abstract ergodicity method implying also quantitative estimate of the constants appearing in the rate of convergence.

\begin{Thm} \label{Ergodicity_theorem_N_particle_Langevin}
Let $d,N \in \mathbb{N}$ and let $\alpha,\beta \in (0,\infty)$ be the constants in \eqref{Eq_N_particle_Langevin}. Assume that the potential $\Phi\colon \mathbb{R}^{dN} \to \mathbb{R} \cup \{ \infty \}$  satisfies Conditions (C0), (C1), (C2) and (C3). Let $\mathbb{P}$ the unique law constructed from the $N$-particle Langevin semigroup which admits $\mu_{\Phi,\beta}$ as invariant measure, see Proposition \ref{Pp_(P)_implies_A_and_S_Langevin}. Let $t>0$ and $f \in L^2(\mu_{\Phi,\beta})$ be arbitrary. We obtain ergodicity with rate of convergence 
\begin{align*}
\left\| \frac{1}{t} \int_0^t f(x_s,\omega_s) \d s  - \mathbb{E}_{\mu_{\Phi,\beta}}\left[f\right] \right\|_{L^2(\mathbb{P})} &\leq  \left( \frac{1}{t} \, \kappa_1 + \frac{1}{\sqrt{t}} \, \kappa_2 \right) \left\| f(X_0)-\mathbb{E}_{\mu_{\Phi,\beta}}\left[f\right] \right\|_{L^2(\mathbb{P})}
\end{align*}
where the constants $\kappa_1,\kappa_2 \in (0,\infty)$ can be specified as
\begin{align*}
\kappa_1=\sqrt{\beta}\, \frac{2}{\sqrt{\text{\normalfont{gap}}(G_{\Phi})}}\quad \mbox{and}\quad \kappa_2= \sqrt{\alpha}\sqrt{\beta} \, \frac{ 2 \sqrt{2}}{\sqrt{\text{\normalfont{gap}}(G_{\Phi})}}  +  \frac{1}{\sqrt{\alpha}} \left( A(\Phi) + \frac{B(\Phi)}{\text{\normalfont{gap}}(G_{\Phi})}\right).
\end{align*}
Here $A(\Phi) \in (0, \infty)$ and $B(\Phi) \in [0,\infty)$ only depend on the choice of $\Phi$; consider Equation \eqref{Specification_of_ni} below for the explicit expressions. 
\end{Thm}

\begin{Rm} \label{Rm_Ergodicity_theorem_N_particle_Langevin}
$ $
\begin{itemize}
\item[(i)] Let $(T_{t,2})_{t \geq 0}$ be as in Remark \ref{Rm_Details_construction_Langevin_semigroup}. As described in Remark \ref{Rm_comparison_ergodicity_exponenential_rate} we even obtain mean ergodicity of the semigroup $(T_{t,2})_{t \geq 0}$, see \eqref{exponen_rat_convergence_semigroup_comparision_ergodicity_rate_1} for the statement. Indeed, the existence of the associated $\mu$-standard right process is proven in \cite[Theo.~3.17]{CG10}.
\item[(ii)]
Consider \cite[Sec.~4.2]{CG10} for rather simple criteria implying (C1) up to (C3). Moreover, in \cite[Sec.~4.3]{CG10} specific potentials $\Phi = \Phi_E + \Phi_I$ are introduced fulfilling those criteria; here $\Phi_I$ includes Lennard-Jones type pair interactions and $\Phi_E$ gives rise to external forces driving particles back to the origin. 
\item[(iii)]
We further remark that in the nonsingular situation, the conditions on $\Phi$ from the hypocoercivity theorem for the Langevin equation (see \cite[Theo.~10]{DMS13} and \cite[Theo.~1]{GS14}) indeed imply the conditions from  Theorem \ref{Ergodicity_theorem_N_particle_Langevin} above. This is proven in Proposition \ref{Lm_sufficient_criteria_implying_P_E_fiber_case}. So, our abstract ergodicity method really allows to study much more general situations. In particular, it can be applied to physical relevant particle systems with singular interactions (see (ii)). It is an interesting problem to extend existing hypocoercivity methods, e.g.~the one from \cite[Theo.~10]{DMS13} and \cite[Theo.~1]{GS14}, to this more general singular situation.
\item[(iv)] 
The rate of convergence in dependence of $\alpha$ and $\beta$ is expected by the following heuristic considerations: Observe that for small values of $\alpha$ close to zero one has a bad or very slow decay towards $\mu_{\Phi,\beta}$ since the dynamics nearly behaves deterministic in this situation. Vice versa, in a large damping regime, the $(x_t)_{t \geq 0}$ process can be described approximately by the overdamped Langevin dynamics\index{overdamped Langevin dynamics}\footnote{see \cite[Sec.~2.2.4]{LRS10}: The scaling $\overline{t} = \frac{t}{\alpha}$, $\overline{x}_{\overline{t}}=x_t$, $\overline{W}_{\overline{t}} = \frac{1}{\sqrt{\alpha}} W_t$, $\overline{\omega}_{\overline{t}} = \alpha \, \omega_t$, $\overline{\Phi}(\overline{x}) = \Phi(x)$ formally implies $
\d \overline{x}_{\overline{t}} = \overline{\omega}_{\overline{t}}\, \d \overline{t} ,\quad \frac{1}{\alpha^2} \d \overline{\omega}_{\overline{t}} = - \overline{\omega}_{\overline{t}} \, \d \overline{t} - \frac{1}{\beta} \, \nabla \overline{\Phi}(\overline{x}_{\overline{t}})\,\d \overline{t} +  \sqrt{\frac{2}{\beta}}\, \d \overline{W}_{\overline{t}}$. We get $\frac{1}{\alpha^2} \d \overline{\omega}_{\overline{t}} \to 0$ as $\alpha \uparrow \infty$. So, setting $\frac{1}{\alpha^2} \d \overline{\omega}_{\overline{t}} = 0$ for $\alpha$ large, solving the equation w.r.t.~$\d \overline{x}_{\overline{t}} = \overline{\omega}_{\overline{t}} \, \d \overline{t} $ and rescaling yields \eqref{overdamped_Langevin_Highfriction}.} (or a macroscopic evolution) in $\mathbb{R}^d$ given as
\begin{align} \label{overdamped_Langevin_Highfriction}
\mathrm{d}x_t = -\frac{1}{\alpha\,\beta} \, \nabla \Phi(x_t) \, \mathrm{d}t + \sqrt{\frac{2}{\alpha \, \beta}}\, \mathrm{d}W_t.
\end{align}
with formal generator $L^{\text{ov}}= \frac{1}{\alpha\,\beta}\Delta -\frac{1}{\alpha\,\beta} \nabla \Phi \cdot \nabla$. If $\Phi$ fulfills a Poincar\'e inequality with constant $\Lambda$, it is then easy to see and well-known that the s.c.c.s.~in $L^2(e^{-\Phi}\d x)$ associated with $L^{\text{ov}}$ (provided it exists) is mean ergodic with rate $\frac{\alpha \, \beta}{\Lambda \, t}$. So, the convergence rate is expected to become as worse as possible when $\alpha \uparrow \infty$ (or $\beta \uparrow \infty$). The case of increasing $\beta$ finally means that $\nu_\beta$ tends to the Dirac distribution at the point $0$ in the velocity $\omega$. And due to the original representation $\mathrm{d}x_t=\omega_t  \d t$,  the $(x_t)_{t \geq 0}$ process is expected to reach its stationary distribution $e^{-\Phi} \d x$ then only very slow; altogether, we see that these phenomena on the convergence to equilibrium in dependence of $\alpha$ and $\beta$ are rigorously proven and confirmed by Theorem \ref{Ergodicity_theorem_N_particle_Langevin} above. Compare with \cite[Theo.~1]{GS14} where the same qualitative convergence behavior in dependence of $\alpha$ in a hypocoercivity setting is shown. 
\item[(v)]
We finally remark, that completely analogously as in \cite{CG10} one may also consider periodic boundaries in the state space for the position variables by replacing $(\mathbb{R}^{d})^N$ through $M^N$, where $M=\prod_{i=1}^d \mathbb{R} / {r_i \mathbb{Z}}$, $r_i > 0$, $i=1,\ldots,d$.
\end{itemize}
\end{Rm}

\subsection{Application to the (generalized) fiber lay-down dynamics} \label{Results_Applications_generalized_fiber_lay_down}

In the second application we consider the so-called generalized fiber lay-down dynamics which is described by the following manifold-valued Stratonovich stochastic differential equation with state space $\mathbb{M}=\mathbb{R}^d \times \mathbb{S}$ of the form
\begin{align} \label{Fiber_Model_Intro}
&\mathrm{d}x_t = \omega_t \, \mathrm{dt} \\
&\mathrm{d}\omega_t = - \frac{1}{d-1}(I- \omega_t \otimes \omega_t) \,\nabla \Phi (x_t) \, \mathrm{dt} + \sigma\,(I-\omega_t \otimes \omega_t) \circ \mathrm{d}W_t. \nonumber
\end{align}
The associated Kolmogorov generator $L$ reads (at first formally) as
\begin{align} \label{Fiber_Operator_Intro}
L= \omega \cdot \nabla_x - \text{grad}_{\mathbb{S}} \Psi \cdot \nabla_\omega + \frac{1}{2} \sigma^2 \, \Delta_{\mathbb{S}}\,~\mbox{ with }\,~\Psi(x,\omega)= \frac{1}{d-1} \,\nabla_x \Phi(x) \cdot \omega.
\end{align}
Here  $d \in \mathbb{N}$, $d \geq 2$, $W$ is a standard $d$-dimensional Brownian motion, $z \otimes y := z y^T$ for $z,y \in \mathbb{R}^{d}$ and $y^T$ is the transpose of $y$. $\mathbb{S}=\mathbb{S}^{d-1}$ denotes the unit sphere with respect to the Euclidean norm in $\mathbb{R}^d$, $\text{grad}_{\mathbb{S}} \,\psi \cdot \nabla_\omega$ or simply $\text{grad}_{\mathbb{S}} \,\psi$ the spherical gradient of some $\psi \in C^\infty(\mathbb{S})$ and $\Delta_{\mathbb{S}}$ the Laplace-Beltrami operator on $\mathbb{S}$. $x$ always indicates the space variable in $\mathbb{R}^d$ and $\omega$ the velocity component in $\mathbb{S} \subset \mathbb{R}^d$ where all vectors in Euclidean space are understood as column vectors. $\Phi\colon \mathbb{R}^d \to \mathbb{R}$ is a potential function specified later on and $\sigma$ a finite constant with $\sigma \geq 0$. Again $\left(\cdot,\cdot\right)_{\text{euc}}$ or $\cdot$ denotes the Euclidean scalar product and $\nabla$ (or $\nabla_x$ respectively) the usual gradient operator in $\mathbb{R}^d$. $\nabla^2_x$ is the Hessian matrix in Euclidean space.

All details on this model can e.g.~be found in \cite{GKMS12}, \cite{GS12} or \cite{GS12B} as well as in the related articles \cite{KMW12} and \cite{GKMW07}. In these articles the stochastic equation has been developed for modeling the lay-down of fibers in the industrial production process of nonwovens. As noticed in the introduction (and see \cite{GS12}) it can alternatively be seen as the analogue of the classical Langevin equation for a particle moving with spherical velocities. By using this interpretation, $\omega_t$ then denotes the attached velocity vector of constant Euclidean norm of some particle moving in $\mathbb{R}^d$ with position-coordinates prescribed by $x_t$. The term $\sigma\,(I-\omega_t \otimes \omega_t) \circ \mathrm{d}W_t$ in the velocity equation describes the stochastic pertubation of the particle given through some Brownian motion on the unit sphere with noise amplitude $\sigma$. Finally, $-\nabla_x \Phi (x_t)$ as usual models the force acting on the particle. However, this forcing term has to be tangential to $\mathbb{S}$ yielding the remaining deterministic term in the equation for $\mathrm{d}\omega_t$. We remark that the constant $\frac{1}{d-1}$ is introduced for having a convenient expression for the density of the invariant measure later on.

The measure $\mu_{\Phi}$ is now defined on $(\mathbb{M},\mathcal{B}(\mathbb{M}))$ as
\begin{align*}
\mu_{\Phi} =  e^{-\Phi(x)} \,\mathrm{d} x \otimes \nu
\end{align*}
where $\nu$ denotes the normalized surface measure on the unit sphere $\mathbb{S}$. 

As already remarked in the introduction, ergodicity with rate of convergence of the two-dimensional fiber lay-down dynamics has already been obtained in \cite{GK08}. The Kolmogorov operator of interest therein is  the two-dimensional version of the fiber lay-down generator \eqref{Fiber_Operator_Intro} equivalently formulated on the space $\mathbb{R}^2 \times \mathbb{R} / 2 \pi \mathbb{Z}$. Recall that our ergodicity method is the generalization of the concrete method from \cite{GK08} to an abstract setting. Thus, of course, we expect to obtain ergodicity with rate of convergence for the generalized fiber lay-down dynamics as well by applying our abstract ergodicity framework from Section \ref{Results_abstract_ergodicity_method}. Consider Theorem \ref{Ergodicity_theorem_generalized_fiber_lay_down} below for the final statement. Before formulating it, let us start introducing the basic conditions required for $\Phi$.

\begin{Pot} Assume that $\Phi\colon \mathbb{R}^{d} \to \mathbb{R} $ is locally Lipschitz continuous, bounded from below and that $e^{-\Phi} \mathrm{d}x$ is a probability measure on $(\mathbb{R}^{d},\mathcal{B}(\mathbb{R}^d))$
\end{Pot}

Due to local Lipschitz continuity, $\Phi$ is weakly differentiable and we fix a version of $\nabla_x \Phi$ in the following. Thus the expression for $L$ from \eqref{Fiber_Operator_Intro} is indeed well-defined, see also \cite[Sec.~3]{GS12B} for more details. Assuming (C0) one obtains the following statement. Consider Section \ref{subsection_proofs_to_generalized_fiber_lay_down} for its proof.

\begin{Pp} \label{Pp_(P)_implies_A_and_S_fiber_lay_down}
Let $\Phi$ satisfies (C0). Then the analytic and stochastic dynamical system assumptions (A) and (S) are fullfilled. Moreover, even Conditions (A)' and (S)' are satisfied. Here the desired s.c.c.s.~$(T_{t,2})_{t \geq 0}$ from (A)' is generated by the closure of the essentially m-dissipative operator $(L,C_c^\infty(\mathbb{M}))$ on $L^2(\mathbb{M},\mu_\Phi)$.
\end{Pp}

Assume that $\Phi$ fulfills $(C0)$. Analogously as in the Section \ref{Results_Applications_N_particle_Langevin}, $(G_\Phi,D(G_\Phi))$ is defined to be the closure of the operator 
\begin{align*}
\Delta_x -\nabla_x \Phi \cdot \nabla_x
\end{align*}
with predomain $C_c^\infty(\mathbb{R}^d)$ in $L^2(e^{-\Phi}\mathrm{d}x)$. Similarly as for the case of the $N$-particle Langevin dynamics we introduce conditions implying ergodicity later on.

\begin{Ass2}
Let $d \in \mathbb{N}$ with $d \geq 2$ and sssume that $\Phi$ satisfies (C0). We further require the following conditions.\\[\Breite]
\begin{tabularx}{\BreiteTabelle}{lX}
(C1)&For all $1 \leq i,j \leq d$ the operators $\partial_{x_{i}}\partial_{x_{j}}$ and $(\partial_{x_{i}} \Phi ) \,\partial_{x_{j}}$ are Kato-bounded on $C_c^\infty(\mathbb{R}^{d})$ by the operator $G_{\Phi}$ in $H_{\Phi}$.\\[\Breite]
(C2)& The operator $(G_{\Phi},C_c^\infty(\mathbb{R}^{d}))$ is essentially selfadjoint on $H_{\Phi}$.\\[\Breite]
(C3)& $G_\Phi$\index{$\text{\normalfont{gap}}(G_{\Phi})$} has a spectral gap, i.e., 
{\begin{align*}
\text{\normalfont{gap}}(G_{\Phi}) \df \inf_{\stackrel{0 \not=f \in D(G_{\Phi})}{\left(1,f\right)_{H_{\Phi}}=0}} \frac{\left(-G_{\Phi} f,f\right)_{H_{\Phi}}}{\|f\|_{H_{\Phi}}^2} > 0.
\end{align*}}\\[\BreiteDrei]
\end{tabularx}
\end{Ass2}

Assuming (C0)-(C4) one gets the final ergodicity theorem for the generalized fiber lay-down dynamics, or the spherical velocity Langevin process equivalently. The proofs are given in Section \ref{subsection_proofs_to_generalized_fiber_lay_down}. As desired, the theorem below contains the ergodicity result and the estimation for the rate of convergence derived in \cite{GK08} for the two-dimensional version of the fiber lay-down dynamics as special case.

\begin{Thm} \label{Ergodicity_theorem_generalized_fiber_lay_down}
Let $d \in \mathbb{N}$, $d \geq 2$ and $\sigma \in (0,\infty)$ the noise amplitude in \eqref{Fiber_Model_Intro}. Assume that $\Phi\colon \mathbb{R}^{d} \to \mathbb{R}$  satisfies (C0), (C1), (C2) and (C3).  Then there exists a unique probability law $\mathbb{P}$\index{law $\mathbb{P}$ associated to!fiber lay-down dynamics} associated on $(C([0,\infty),\mathbb{M})$ with the fiber lay-down semigroup $(T_{t,2})_{t \geq 0}$, see Proposition \ref{Pp_(P)_implies_A_and_S_fiber_lay_down}. Let $t >0$ and $f \in L^2(\mathbb{M}, \mu_\Phi)$ be arbitrary. We obtain
\begin{align*}
\left\| \frac{1}{t} \int_0^t f(x_s,\omega_s) \d s  - \mathbb{E}_{\mu_\Phi}\left[f\right] \right\|_{L^2(\mathbb{P})} &\leq  \left( \frac{1}{t} \, \kappa_1 + \frac{1}{\sqrt{t}} \, \kappa_2 \right) \left\| f(x_0,\omega_0)-\mathbb{E}_{\mu_\Phi}\left[f\right] \right\|_{L^2(\mathbb{P})}
\end{align*}
Here the constants $\kappa_1,\kappa_2$ can be specified as
\begin{align*}
\kappa_1=\frac{2\, \sqrt{d}}{\sqrt{\text{\normalfont{gap}}(G_{\Phi})}}\quad \mbox{and}\quad \kappa_2= \sigma \, \frac{ 2 \,\sqrt{d\,(d-1)}}{\sqrt{\text{\normalfont{gap}}(G_{\Phi}) }}  +  \frac{1}{\sigma} \left( A(\Phi) + \frac{B(\Phi)}{\text{\normalfont{gap}}(G_{\Phi})}\right)
\end{align*}
where $A(\Phi) \in (0, \infty)$, $B(\Phi) \in [0,\infty)$ only depend on $\Phi$ (and on the dimension $d$); see \eqref{Specification_of_mi} for explicit expressions. 
\end{Thm}

\begin{Rm}
Note that for small and large values of $\sigma$ the rate of convergence increases and an optimal rate of convergence is expected for $\sigma$ lying in an intermediate range of values. The same characteristic behavior of rate of convergence in dependence of $\sigma$ is derived in the hypocoercivity setting in \cite{DKMS11}, see also \cite{GKMS12} and \cite{GS12B} for the $d$-dimensional case. Finally, consider \cite{GK08} and \cite{DKMS11} for some numerical simulations that confirm the rate of convergence in dependence of $\sigma$.
\end{Rm}

Finally, we shall give sufficient criteria implying (C0)-(C4). We show that the same conditions for $\Phi$ as already assumed in the hypocoercivity setting in \cite{GS12B}, \cite[Theo.~10]{DMS13} or \cite[Theo.~1]{GS14} are sufficient for the method considered here. For simple criteria on the existence of Poincar\'e inequalities itself we refer e.g.~to \cite[A.~19]{Vil09}, \cite{BBCG08} or \cite{Wan99}. We emphasize that potentials of the form $\Phi(x)=\sum_{i=1}^d a_i \,x_i^2$, $a_i > 0$, which are relevant for the fiber lay-down application, satisfy (after normalization) the necessary conditions below.

\begin{Pp} \label{Lm_sufficient_criteria_implying_P_E_fiber_case}
Let $d \in \mathbb{N}$ (and $d \geq 2$ in the fiber lay-down case). Assume that the potential $\Phi\colon \mathbb{R}^d \rightarrow \mathbb{R}$ is bounded from below, satisfies $\Phi \in C^{2}(\mathbb{R}^d)$ and that $e^{-\Phi} \mathrm{d}x$ is a probability measure on $(\mathbb{R}^d,\mathcal{B}(\mathbb{R}^d))$. Moreover, the measure $e^{-\Phi}\mathrm{d}x$ is assumed to satisfy a Poincar\'e inequality of the form
\begin{align*}
\big\|\nabla f \big\|^2_{L^2(e^{-\Phi}\mathrm{d}x)} \geq \Lambda  \, \left\|\, f - \int_{\mathbb{R}^d} f \, e^{-\Phi}\mathrm{d}x \,\right\|^2_{L^2(e^{-\Phi}\mathrm{d}x)} \mbox{~for all $f \in C_c^\infty(\mathbb{R}^d)$}
\end{align*}
where $\Lambda \in (0,\infty)$. Furthermore, assume that there exists $c < \infty$ such that
\begin{align*}
\left| \nabla^2 \Phi (x) \right| \leq c \left( 1+ \left| \nabla \Phi(x) \right|\right) \quad \mbox{for all } x \in \mathbb{R}^d.
\end{align*}
Then indeed Conditions (C0), (C1), (C2) and (C3) from the Langevin dynamics (see Section \ref{Results_Applications_N_particle_Langevin} with $N=1$) or from the fiber lay-down case from above are fulfilled.
\end{Pp}

\section{Definitions and proofs} \label{Regular sub-Markovian semigroups and associated laws}

This section is devoted to give the proofs to all assertions made in Section \ref{section_Results}. In this context, we need to introduce and recapitulate first some basic definitions and knowledge about probability laws and regular sub-Markovian semigroups in Subsection \ref{subsection_laws} and Subsection \ref{subsection_laws_and_associated_semigroups}. The both last named subsections do not contain new material and are essentially based on results developed in the PhD-thesis of Florian Conrad, see \cite{Con11} and see also \cite{CG10}. We remark that Subsection \ref{subsection_laws_and_associated_semigroups} below is only needed for the proof of Theorem \ref{Thm_Martingale_Problem_via_regular semigroups} and recall that Theorem \ref{Thm_Martingale_Problem_via_regular semigroups} is a result for verifying the martingale property in Assumption (S). So, for understanding the abstract ergodicity framework, Theorem \ref{Thm_Martingale_Problem_via_regular semigroups} is not necessarily required. Consequently, the interested reader may skip Section \ref{subsection_laws_and_associated_semigroups} in first reading and may directly switch to Section \ref{subsection_proofs_to_abstract_setting} in which all proofs for the abstract ergodicity setting can be found.

\subsection{Basics about probability laws}\label{subsection_laws}

Let us start with standard definitions of probability laws, see e.g.~\cite[Ch.~3]{EK86}. The notion of time reversal is taken from \cite[Sec.~2.1.3]{Con11}. 

\begin{Df} \label{Df_probability_law}
In the following, let $(E,r)$ be a  separable metric space equipped with its Borel $\sigma$-algebra $\mathcal{B}(E)$ generated by the open sets. Let $I=[0,\infty)$ or $I=[0,T]$ for\index{$C([0,\infty);E)$} some $T>0$. $C(I;E)$$\index{$C(I;E)$}$ denotes the space of continuous paths $\omega=(\omega_t)_{t \in I}$ taking values in $E$. $C(I;E)$ can itself be equipped with a metric (inducing  uniform convergence on compact intervals) via
\begin{align*}
d(\omega,\nu)=\int_0^\infty e^{-s} \sup_{0 \leq t \leq s} \left[r(\omega_t,\nu_t) \wedge 1\right] \d s \quad \mbox{for all}~\omega,\nu \in C([0,\infty);E)
\end{align*}
in case $I=[0,\infty)$ (see \cite[Eq.~(10.5)]{EK86}) and via
\begin{align*}
d(\omega,\nu)=\sup_{0 \leq t \leq T} \left[r(\omega_t,\nu_t)\right]\quad \mbox{for all }\omega,\nu \in C([0,T];E)
\end{align*}
in case $I=[0,T]$. Here $\wedge$ denotes the minimum\index{$\wedge$}. The associated $\sigma$-field on $C(I;E)$ generated by the open sets is denoted by $\mathcal{F}_C$\index{$\mathcal{F}_C$}. In the following, $(\mathcal{F}_t^0)_{t \in I}$ denotes the filtration\index{evaluation of paths $X_t$} generated by the paths\index{elementary filtration $(\mathcal{F}_t^0)_{t \in I}$}, i.e., $\mathcal{F}_t^0=\sigma(X_s\,|\,0 \leq s \leq t)$\index{$\sigma \{ X_s~|~0 \leq s \geq t\}$} for each $t \in I$. Here
\begin{align*}
X_s \colon C(I;E) \to E,\quad X_s(\omega) \df \omega_s\quad \mbox{for all }s \geq 0.
\end{align*}
It is well-known that $\mathcal{F}_C = \mathcal{F}_\infty^0\df \sigma \{ X_s~|~s \geq 0\}$ in case $I=[0,\infty)$ and $\mathcal{F}_C = \mathcal{F}_T^0$ in case $I=[0,T]$, see e.g.~\cite[Prob.~3.25]{EK86}. A \textit{probability law} $\mathbb{P}$\index{probability law $\mathbb{P}$} is defined to be a probability measure on $\left(C(I;E),\mathcal{F}_C\right)$ and its corresponding expectation is denoted by $\mathbb{E}[\cdot]$ or by $\mathbb{E}_{\mathbb{P}}[\cdot]$\index{$\mathbb{E}[\cdot]$ or $\mathbb{E}_{\mathbb{P}}[\cdot]$}. For any such law $\mathbb{P}$ its \textit{initial probability distribution $\mu$}\index{initial distribution $\mu$ of a law $\mathbb{P}$} is defined to be $\mathbb{P} \circ X_0^{-1} $. $\mathbb{P}$ is said to have the \textit{invariant measure $\mu$}\index{invariant measure of a law} if $\mathbb{P} \circ X_t^{-1}=\mu$ for all $t \in I$. Now let $\mathbb{P}$ be a law on $C([0,\infty);E)$ and let $T>0$. Define
\begin{align*} \index{$\text{restr}_{[0,T]}$}
\text{restr}_{[0,T]} \colon C([0,\infty);E) \to C([0,T];E), \quad \omega= \left( \omega_t \right)_{t \geq 0} \mapsto \left( \omega_t \right)_{t \in [0,T]}.
\end{align*}
It is a continuous, hence measurable mapping. We denote the pushforward measure on $C([0,T];E)$ of $\mathbb{P}$ under $\text{restr}_{[0,T]}$ by $\mathbb{P}_{T}$\index{$\mathbb{P}_T$}, i.e., $\mathbb{P}_T = \mathbb{P} \circ \text{restr}_{[0,T]}^{-1}$. The mapping 
\begin{align*}
\tau_T\colon C([0,T];E) \rightarrow C([0,T];E),\quad \tau_{T}\left((\omega_s)_{s \in [0,T]}\right)\df(\omega_{T-s})_{s \in [0,T]}
\end{align*}
is called the \textit{time-reversal}\index{time-reversal $\tau_T$}. It is also continuous, hence measurable. Moreover, it is easy to see that $\tau_T$ is $\mathcal{F}_T^0 / \mathcal{F}_T^0$-measurable; The probability law $\mathbb{P}_{T} \circ \tau_T^{-1}$\index{$\mathbb{P}_{T} \circ \tau_T^{-1}$} on $C([0,T];E)$ is called the \textit{time-reversal of $\mathbb{P}_{T}$}.
\end{Df}

With these definitions at hand we obtain the following lemma. It contains statements similar as the ones used in the beginning of the proof of \cite[Lem.~2.1.8]{Con11}, or see the proof of \cite[Lem.~5.1]{CG10} equivalently. However, the assumptions in the upcoming lemma are completely different compared to \cite[Lem.~2.1.8]{Con11}. Below $\int_0^t \mathrm{d}s$\index{$\int_0^t \mathrm{d}s$} denotes the Lebesgue integral $\int_{[0,t]} \mathrm{d}s$.

\begin{Lm} \label{Lm_well_definedness_integrability}
Let $E$ be a separable metric space and let $\mu$ be a probability measure on $(E,\mathcal{B}(E))$. Let $\mathbb{P}$ be a probability law on $C([0,\infty);E)$ which admits $\mu$ as invariant measure.
\begin{itemize}
\item[(i)] Let $t \geq 0$ and choose $f \in \mathcal{L}^1(E,\mu)$. One has that $f(X_t)$ is integrable w.r.t.~$\mathbb{P}$ and
\begin{align*}
\mathbb{E}\left[  f(X_t) \right] = \int_E f \, \d \mu \leq  \|f\|_{L^1(E,\mu)}.
\end{align*}
In particular, the $\mathcal{F}^{0}_t$-measurable random variable $f(X_t)$ is $\mathbb{P}$-a.s.~well-defined for each $f \in L^1(E,\mu)$. This means that any two $\mu$-versions of $f$ yield $\mathbb{P}$-a.s.~the same random variable $f(X_t)$. 
\item[(ii)] Let $t \geq 0$ and let $f \in \mathcal{L}^1(E, \mu)$. Then the mapping
\begin{align}  \label{numerate_mapping}
[0,t] \times C([0,\infty);E) \ni (s,\omega) \mapsto f(\omega_s) \in \mathbb{R}
\end{align}
is $\mathcal{B}([0,t]) \otimes \mathcal{F}^{0}_t$-measurable and integrable w.r.t.~$\mathrm{d} s \otimes \mathbb{P}$. Hence by Fubini's theorem $\int_0^t f(X_s)\d s$ exists $\mathbb{P}$-a.s., extends to an $\mathcal{F}^{0}_t$-measurable, $\mathbb{P}$-integrable random variable which satisfies
\begin{align*}
\mathbb{E}\left[ \int_0^t f(X_s) \d s \right]  = \int_0^t \mathbb{E}\left[ f(X_s)\right] \d s  \leq   t\, \|f\|_{L^1(E,\mu)}.
\end{align*}
In particular, $\int_0^t f(X_s) \d s$ is also $\mathbb{P}$-a.s.~well-defined for each $f \in L^1(E,\mu)$. 
\item[(iii)] Let $t \geq 0$ and choose $f \in \mathcal{L}^2(E,\mu)$. Then $f(X_t)$ is square integrable w.r.t.~$\mathbb{P}$, $\mathcal{F}^{0}_t$-measurable and the mapping
\begin{align} \label{label_mapping_progr_measurable}
[0,t] \times C([0,\infty);E) \ni (s,\omega) \mapsto f(\omega_s) \in \mathbb{R}
\end{align}
is $\mathcal{B}([0,t]) \otimes \mathcal{F}^{0}_t$-measurable and square integrable w.r.t.~$\mathrm{d} s \otimes \mathbb{P}$. Furthermore, $\int_0^t f(X_s)\d s$ exists $\mathbb{P}$-a.s., extends to an $\mathcal{F}^{0}_t$-measurable, square integrable random variable w.r.t.~$\mathbb{P}$ and satisfies
\begin{align}\label{eq_property_from_microscopic}
\|f(X_t)\|_{L^2(\mathbb{P})} =  \|f\|_{L^2(E,\mu)}, \quad \left\|\int_0^t f(X_s) \, \mathrm{d}s\right\|_{L^2(\mathbb{P})} \leq t\, \|f\|_{L^2(E,\mu)}.
\end{align}
In particular, $f(X_t)$ and $\int_0^t f(X_s) \d s$ are also $\mathbb{P}$-a.s.~well-defined for all $f \in L^2(E,\mu)$. 
\end{itemize}
All statements are valid in the same form when $\mathbb{P}$ is assumed to be a probability law on $C([0,T];E)$ for an arbitrary $T>0$. In this case $[0,\infty)$ and $t \geq 0$ must correspondingly be replaced by $[0,T]$ and  $t \in [0,T]$ above.
\end{Lm}

\begin{proof}
We prove (i). Clearly, $f(X_t) \in \mathcal{L}^1(\mathbb{P})$ and $\mathbb{E}\left[f(X_t)\right] =  \int_E f\, \d \mu$ by the transformation rule of image measures since $\mathbb{P} \circ X_t^{-1} = \mu$. Thus (i) follows. We prove (ii). One knows that $(X_s)_{s \geq 0}$ is $(\mathcal{F}^0_s)_{s \geq 0}$-progressively measurable on $(C([0,\infty);E)$, see e.g.~\cite[Prob.~2.1]{EK86}. Hence also $(f(X_s))_{s \geq 0}$ is $(\mathcal{F}^0_s)_{s \geq 0}$-progressively measurable. Thus the mapping in \eqref{numerate_mapping} satisfies the stated measurability property. And by part (i) we have 
\begin{align*}
\int_0^t \mathbb{E} \big[ |f|(X_s) \big]  \d s = t\,\|f\|_{L^1(E, \mu)} < \infty.
\end{align*}
So, (ii) is shown. Finally, we prove (iii). Define $g \df f^2 \in \mathcal{L}^1(E,\mu)$. Note that even $f \in \mathcal{L}^1(E,\mu)$ since $\mu$ is a probability measure. Part (i) applied to $f$ and $g$ yields that $f(X_s)$ is square integrable w.r.t.~$\mathbb{P}$, $\mathcal{F}^{0}_s$-measurable and
\begin{align} \label{Hilfszeile}
\|f(X_s)\|_{L^2(\mathbb{P})}^2 = \|g(X_s)\|_{L^1(\mathbb{P})} = \|g\|_{L^1(E,\mu)} =  \|f\|_{L^2(E,\mu)}^2\quad \mbox{for all } s \geq 0.
\end{align}
Now part (ii) applied to $f \in \mathcal{L}^1(E,\mu)$ in particular yields that the mapping in \eqref{label_mapping_progr_measurable} is $\mathcal{B}([0,t]) \otimes \mathcal{F}^{0}_t$-measurable and integrable w.r.t.~$\mathrm{d} s \otimes \mathbb{P}$, hence $\int_0^t f(X_s)\d s$ is $\mathcal{F}^{0}_t$-measurable by Fubini's theorem. Part (ii) applied now to $g$ yields that the mapping in \eqref{label_mapping_progr_measurable} is even square integrable w.r.t.~$\mathrm{d} s \otimes \mathbb{P}$. Finally, by  using Jensen's inequality (or H"{o}lder's inequality) we can infer that the inequality
\begin{align*}
\left(\int_0^t f(X_s) \d s\right)^2 \leq t \int_0^t f^2(X_s) \d s =  t \int_0^t g(X_s) \d s 
\end{align*}
holds $\mathbb{P}$-a.s.~on $(C([0,\infty);E),\mathcal{F}_t^0,\mathbb{P})$. By part (ii) note that $\int_0^t g(X_s) \d s \in L^1(\mathbb{P})$ and by additionally using \eqref{Hilfszeile} we get
\begin{align*}
\mathbb{E} \left[ \left(\int_0^t f(X_s) \d s\right)^2  \right] \leq t^2 \|f\|_{L^2(E,\mu)}^2.
\end{align*}
This finishes part (iii) of the proof. The proof in case $\mathbb{P}$ is a law on $C([0,T];E)$ is analogous.
\end{proof}

\begin{Rm} 
Let $E$ be a separable metric space and let $\mathbb{P}$ be a law on $C([0,\infty);E)$ having the invariant probability measure $\mu$. $\mathbb{P}$ induces the law $\mathbb{P}_t$ on $C([0,t];E)$ for $t>0$. Choose $f \in \mathcal{L}^1(E,\mu)$. By Lemma \ref{Lm_well_definedness_integrability} one can define the integrable random variable $\int_0^t f(X_s) \, \mathrm{d}s$ once on $(C([0,\infty);E),\mathcal{F}_t^0,\mathbb{P})$ and once on $(C([0,t];E),\mathcal{F}_t^0,\mathbb{P})$. One easily verifies the natural identity 
\begin{align} \label{Rm_identity_restriction_integral_ergodicity}
\int_0^t f(X_s) \, \mathrm{d}s = \int_0^t f(X_s) \, \mathrm{d}s \circ \text{restr}_{[0,t]}
\end{align}
which is satisfied $\mathbb{P}$-a.s.~on $(C([0,\infty);E),\mathcal{F}_t^0,\mathbb{P})$.
\end{Rm}

\subsection{Basics about regular sub-Markovian semigroups} \label{subsection_laws_and_associated_semigroups}

In this section we closely follow \cite{Con11}, see especially \cite[Sec.~2.1]{Con11} together with \cite[Sec.~2.1.3]{Con11}. We adapt some definitions and statement to our relevant situation but do not need all assumptions in the generality as introduced in the last mentioned reference. Especially, we are only considering laws that automatically correspond to conservative diffusion processes and do not adjoin a cemetry $\Delta$ to the underlying state space $E$. Moreover, we only assume $E$ to be a metric space later on. Here the word \textit{diffusion}\index{diffusion} formally means a process with continuous sample paths and \textit{conservative}\index{conservative} means that the process stays on $E$ for all times, i.e., does not die. We start with the definition of regular sub-Markovian s.c.c.s. 

\begin{Df}
Let $(E,\mathcal{B},\mu)$ be a $\sigma$-finite measure space. Let $(T_t)_{t \geq 0}$ be a sub-Markovian s.c.c.s.~on $L^2(E,\mu)$. $(T_t)_{t \geq 0}$ is called \textit{regular} iff the adjoint semigroup\index{semigroup!regular} $(\hat{T}_t)_{t \geq 0}$ is sub-Markovian on $L^2(E,\mu)$ as well. Hence $(\hat{T}_t)_{t \geq 0}$ is also a sub-Markovian s.c.c.s.~on $L^2(E,\mu)$.
\end{Df}

Here a s.c.c.s.~$(T_t)_{t \geq 0}$ on $L^p(E,\mu)$, $p \in [1,\infty)$, is called sub-Markovian\index{semigroup!sub-Markovian} if $0 \leq T_t f \leq 1$ for all $f \in L^p(E,\mu)$ with $0 \leq f \leq 1$. Here the ordering relation $\leq$ is clearly understood $\mu$-a.e. Then one obtains the following result.

\begin{Pp} \label{Pp_regular_sccs}
Let $(E,\mathcal{B},\mu)$ be a $\sigma$-finite measure space. There is a one-to-one correspondence between sub-Markovian s.c.c.s.~$(T_{t,1})_{t \geq 0}$ on $L^1(E,\mu)$ and regular sub-Markovian s.c.c.s.~$(T_{t,2})_{t \geq 0}$ on $L^2(E,\mu)$. Furthermore, the semigroups are related via
\begin{align} \label{relation_T1_and_T2}
T_{t,1}=T_{t,2}\quad \mbox{on } L^1(E,\mu) \cap L^\infty(E,\mu).
\end{align} 
\end{Pp}
For the proof of the proposition see \cite[Lem.~1.3.11(i)]{Con11}, \cite[Lem.~1.3.11(ii)]{Con11} and \cite[Lem.~1.3.14(ii)]{Con11}. Moreover, let $(L_i,D(L_i))$, $i=1,2,$ be the generator of $(T_{t,i})_{t \geq 0}$. If $f \in D(L_2)$ such that $f$ and $L_2f$ are both elements from $L^1(E,\mu)$ it follows that $f \in D(L_1)$ and $L_1f=L_2f$, see \cite[Lem.~1.3.11(iii)]{Con11} or \cite[Lem.~1.3.14(iii)]{Con11}.\label{pageref_generator_L1_and_L2}

Furthermore, we need the definition of associatedness of a probability law with a sub-Markovian s.c.c.s.~as defined in \cite[Def.~2.1.3]{Con11}, or as in \cite{CG10} equivalently. For the rest of this section, $E$ is always a separable metric space equipped with the associated Borel $\sigma$-algebra $\mathcal{B}(E)$ generated by the open sets.

\begin{Df} \label{Df_associatedness_probabilitylaw_semigroup}\index{associatedness!semigroup and law}
Let $E$ be a separable metric space. Let $\mathbb{P}$ be a probability law on $C([0,\infty);E)$ with initial probability distribution $\mu$. Let $(T_t)_{t \geq 0}$ be a sub-Markovian s.c.c.s.~on $L^p(E,\mu)$ for an $p \in [1,\infty)$. Then $\mathbb{P}$ is said to be associated with $(T_t)_{t \geq 0}$\index{semigroup!associatedness with a law $\mathbb{P}$} if for all nonnegative $f_1,\ldots,f_n \in L^\infty(E,\mu)$, $0 \leq t_1 \leq \cdots \leq t_n$, $n \in \mathbb{N}$, it holds
\begin{align*}
\mathbb{E}\left[f_1(X_{t_1}) \cdots f_n(X_{t_n})\right]=\mu\left( T_{t_1} \left( f_1 T_{t_2-t_1} \left(f_2 \cdots T_{t_{n-1}-t_{n-2}} \left( f_{n-1} T_{t_{n}-t_{n-1}} f_n\right)\right)\right)\right).
\end{align*}
Here $\mu(f)\df\int_E f \d \mu$\index{$\mu(f)$} for some $f \in L^1(E,\mu)$. Associatedness on $C([0,T];E)$ for some $T \geq 0$ is defined in the analogous way.
\end{Df}

\begin{Rm}\label{Rm_associatedness}
$ $
\begin{itemize}
\item[(i)] Of course, one should formulate the associatedness condition first only for nonnegative functions from $\mathcal{L}^\infty(E,\mu)$. However, the right hand side respects $\mu$-equivalence classes of functions from $\mathcal{L}^\infty(E,\mu)$. Thus the associatedness condition makes sense also for functions from $L^\infty(E,\mu)$. Moreover, a sufficient criterion for verifying the associatedness condition is given in \cite[Lem.~2.1.4]{Con11}.
\item[(ii)] By \eqref{relation_T1_and_T2} and Proposition \ref{Pp_regular_sccs} note that a probability law $\mathbb{P}$ is associated with some regular sub-Markovian s.c.c.s.~$(T_{t,2})_{t \geq 0}$ on $L^2(E,\mu)$ iff $\mathbb{P}$ is associated with the corresponding sub-Markovian s.c.c.s.~$(T_{t,1})_{t \geq 0}$ from $L^1(E,\mu)$.
\item[(iii)] Let $E$ be as above. Then there is at most one probability law $\mathbb{P}$ on $C([0,\infty);E)$ with initial probability distribution $\mu$ which is associated with $(T_t)_{t \geq 0}$. This follows with the same argumentation as in \cite[Rem.~2.1.5]{Con11} since $\mathcal{F}_C$ is already generated by the cylinder sets\index{cylinder sets} 
\begin{align*}
\{ X_{s_1} \in A_1 \} \cap \cdots \cap \{ X_{s_k} \in A_k\},\quad 0 \leq s_1 \leq \cdots \leq s_k < \infty,\quad k \in \mathbb{N},\quad A_{i} \in \mathcal{B}(E).
\end{align*}
The analogous uniqueness statement is valid in case $\mathbb{P}$ is assume to be a probability law on $C([0,T];E)$ for some $T>0$.
\end{itemize}
\end{Rm}

Next, we need two more well-known lemmas. For completeness, we recapitulate their proofs below. For the following lemma, see e.g.~\cite[Lem.~1.3.21, Lem.~2.1.14]{Con11} and \cite[Ch.~II, Prop.~4.1]{MR92}. Conservativity and $\mu$-invariance is defined on page \pageref{Def_conservativity_mu_invariance}.

\begin{Lm} \label{Lm_mu_invariance_and_sub_Marko_implies_regularity}
Let $(E,\mathcal{B},\mu)$ be a probability space. Let $(T_{t,2})_{t \geq 0}$ be a $\mu$-invariant sub-Markovian s.c.c.s.~on $L^2(E,\mu)$. Then $(T_{t,2})_{t \geq 0}$ is conservative and regular.
\end{Lm}

\begin{proof}
By the sub-Markovian property we have $T_{t,2}1 -1 \geq 0$, $t \geq 0$. But $\mu$-invariance implies $\mu(T_{t,2}1 -1)=0$, hence $T_{t,2}1=1$ for each $t \geq 0$; The regularity statement follows by the same calculation performed at the end of the proof of \cite[Ch.~II, Prop.~4.1]{MR92}. Indeed, let $f,g \in L^2(E,\mu)$ with $0\leq f,g \leq 1$. By the sub-Markovian property and $\mu$-invariance we get $0 \leq \int_E f \,T_{t,2}g \d \mu \leq \mu(T_{t,2}g)=\mu(g)$. Thus $0 \leq \int_E g \,\hat{T}_{t,2}f \d \mu \leq \int g \d \mu$. Hence easily $0 \leq \hat{T}_{t,2}f \leq 1$ for each $f \in L^2(E,\mu)$ with $0 \leq f \leq 1$ and all $t \geq 0$.
\end{proof}

The next lemma is obtained from \cite[Lem.~2.1.14]{Con11}.


\begin{Lm} \label{Lm_properties_semigroup_related_law}
Let $E$ be a separable metric space and $\mathbb{P}$ a probability law on $C([0,\infty);E)$ with initial probability distribution $\mu$. Assume that $\mathbb{P}$ is associated with a sub-Markovian s.c.c.s.~$(T_{t,2})_{t \geq 0}$ on $L^2(E,\mu)$. Then the following statements hold.
\begin{enumerate}
\item[(i)]
$(T_{t,2})_{t \geq 0}$ is conservative.
\item[(ii)]
$\mu$ is a an invariant measure for $\mathbb{P}$ iff $\mu$ is invariant for $(T_{t,2})_{t \geq 0}$.
\item[(iii)] Let $\mu$ be invariant for $\mathbb{P}$. Then the time-reversed law $\mathbb{P}_{T} \circ \tau_T^{-1}$ on $C([0,T];E)$, $T \geq 0$, is associated with $(\hat{T}_{t,2})_{t \in [0,T]}$. 
\item[(iv)] In the situation from (iii), $\mu$ is also an invariant measure for $\mathbb{P}_{T} \circ \tau_T^{-1}$.
\end{enumerate}
\end{Lm}

\begin{proof} We prove (i). The associatedness condition implies
\begin{align*}
\mu(T_{t,2}1)=\mathbb{E}\left[ 1(X_t) \right] =1 \quad \mbox{for all } t \in [0,\infty).
\end{align*}
We conclude that $\mu(1-T_{t,2}1)=0$ and thus $1-T_{t,2}1=0$ since $0 \leq T_{t,2}1 \leq 1$. We prove (ii). So, let $\mu$ be invariant for $\mathbb{P}$ and choose $A \in \mathcal{B}(E)$, $t \in [0,\infty)$. Then
\begin{align*}
\mu(1_A)=\mu(A) = \mathbb{P}(X_t^{-1}(A))=\mathbb{E}\left[ 1_A(X_t) \right]=\mu(T_{t,2} 1_A).
\end{align*}
Hence invariance of $\mu$ w.r.t.~$(T_{t,2})_{t \geq 0}$ follows. The other direction is obvious. Now let us prove (iii). First note that the associatedness statement makes sense since $(\hat{T}_{t,2})_{t \geq 0}$ is indeed sub-Markovian by Lemma \ref{Lm_mu_invariance_and_sub_Marko_implies_regularity}. So, let $f_1,\ldots,f_n \in \mathcal{L}^\infty(E,\mu)$ be nonnegative, $0 \leq t_1 \leq \cdots \leq t_n \leq T$, $n \in \mathbb{N}$. Then
\begin{align*}
\int_{C([0,T];E)} f_1(X_{t_1}) \cdots f_n(X_{t_n}) &\d \mathbb{P}_{T} \circ \tau_T^{-1} = \int_{C([0,T];E)} f_n(X_{T-t_n}) \cdots f_1(X_{T-t_1}) \d \mathbb{P}_{T}\\
&=\mu( T_{T-t_n,2} ( f_n T_{t_n-t_{n-1},2} (f_{n-1} \cdots T_{t_{3}-t_{2},2} ( f_{2} T_{t_{2}-t_{1},2} f_1))))\\
&=\mu( f_n T_{t_n-t_{n-1},2} (f_{n-1} \cdots T_{t_{3}-t_{2},2} ( f_{2} T_{t_{2}-t_{1},2} f_1)))\\
&=\mu( f_1 \hat{T}_{t_2-t_1,2} (f_2 \cdots \hat{T}_{t_{n-1}-t_{n-2},2} ( f_{n-1} \hat{T}_{t_{n}-t_{n-1},2} f_n)))\\
&=\mu( \hat{T}_{t_1,2} ( f_1 \hat{T}_{t_2-t_1,2} (f_2 \cdots \hat{T}_{t_{n-1}-t_{n-2},2} ( f_{n-1} \hat{T}_{t_{n}-t_{n-1},2} f_n)))).
\end{align*}
Here the third equality follows due to the invariance of $\mu$ w.r.t.~$(T_{t,2})_{t \geq 0}$. The last equality holds since
\begin{align} \label{conservativity_implies_invariance_of_dual_semigroup}
\mu(\hat{T}_{t,2}f) = (1,\hat{T}_{t,2} f)_{L^2(E,\mu)} = \left(T_{t,2} 1,f\right)_{L^2(E,\mu)} = \left(1,f\right)_{L^2(E,\mu)} = \mu(f)
\end{align}
is fulfilled for all $t \geq 0$ and all $f \in L^2(E,\mu)$. This also proves (iv).
\end{proof}

Now finally, we recapitulate a specific case of a result from \cite[Lem.~2.1.8]{Con11} (or see \cite[Lem.~5.1]{CG10} equivalently). As remarked in \cite[p.~62]{Con11} it is itself a combination of \cite[Prop.~1.4]{BBR06} and \cite[Theo.~4.6]{DIPP84}; in \cite{Con11} the statement is formulated on the space of c\`{a}dl\`{a}q paths taking values in a polish space. In the upcoming proposition we only assume that $E$ is a separable metric space and formulate the statement on $C([0,\infty);E)$. We remark that the proof carries over in exactly the same way without any modification.

\begin{Pp} \label{Pp_martingale_problem} (Martingale problem) \index{martingale problem}
Let $E$ be a separable metric space. Let $\mathbb{P}$ be a probability law on $C([0,\infty);E)$ with initial and invariant probability distribution $\mu$. Assume that $\mathbb{P}$ is associated with a sub-Markovian s.c.c.s.~$(T_{t,2})_{t \geq 0}$ on $L^2(E,\mu)$ (which is regular by Lemma \ref{Lm_mu_invariance_and_sub_Marko_implies_regularity} and Lemma \ref{Lm_properties_semigroup_related_law}). Denote the generator of $(T_{t,2})_{t \geq 0}$ on $L^2(E,\mu)$ by $(L_2,D(L_2))$. Then $\mathbb{P}$ solves the martingale problem for $(L_2,D(L_2))$, i.e., the process $(M_t^{[f],L_2})_{t \geq 0}$\index{$M^{[f],L_2}$}, $f \in D(L_2)$, defined by
\begin{align*}
M_t^{[f],L_2}=f(X_t) - f(X_0) - \int_0^t L_2f(X_s) \d s,\quad t \geq 0,
\end{align*}
is $\mathbb{P}$-integrable and $(M_t^{[f],L_2})_{t \geq 0}$ is an $(\mathcal{F}_t^{0})_{t \geq 0}$-martingale under $\mathbb{P}$. For all $f \in D(L_2)$ the martingale $(M_t^{[f],L_2})_{t \geq 0}$ is even square integrable and if furthermore $f^2 \in D(L_1)$, then 
\begin{align*}
N_t^{[f],L_1,L_2}\df\left(M_t^{[f],L_2}\right)^2 - \int_0^t \left(L_1(f^2)(X_s) - \left(2fL_2f\right)(X_s)\right) \d s,\quad t \geq 0,
\end{align*}
is $\mathbb{P}$-integrable \index{$N^{[f],L_1,L_2}$, $N^{[f],L_2}$}and defines also an $(\mathcal{F}_t^{0})_{t \geq 0}$-martingale under $\mathbb{P}$. Here $(L_1,D(L_1))$ denotes the generator of the sub-Markovian s.c.c.s.~$(T_{t,1})_{t \geq 0}$ on $L^1(E,\mu)$ associated to $(T_{t,2})_{t \geq 0}$, see Proposition \ref{Pp_regular_sccs}. Note that the requirement $f^2 \in D(L_1)$ is fulfilled if $f^2 \in D(L_2)$. In this case we also have $L_1 f^2 = L_2 f^2$ and we denote $N_t^{[f],L_1,L_2}$ for short by $N_t^{[f],L_2}$. 
\end{Pp}

\begin{proof} See \cite[Lem.~2.1.8]{Con11} or see \cite[Lem.~5.1]{CG10}.
\end{proof}

\begin{Rm}
In the situation where $E$ is a separable metric space we have that $\mathcal{F}_C$ is equal to $\mathcal{F}_\infty^0$. However, in the whole abstract framework (above and below) the assumption $E$ being separable can be dropped. Indeed, if we only assume $E$ to be metric, then every statement from above and below stays valid if we initially introduce a probability law as a probability measure on $((C[0,\infty);E),\mathcal{F}_\infty^0)$. However, this is not the common definition of a probability law used in literature and would probably be confusing. Hence we always assume separability of $E$ for convenience.
\end{Rm}

After this short recapitulation, we can now go on to prove the assertions and statements claimed in the abstract ergodicity method in Section \ref{Results_abstract_ergodicity_method}. 

\subsection{Proofs to Section \ref{Results_abstract_ergodicity_method} (The abstract ergodicity method)}\label{subsection_proofs_to_abstract_setting}

For the rest of this section assume Conditions (A) and (S) from Section \ref{Results_abstract_ergodicity_method} and we use the notations introduced in Section \ref{Results_abstract_ergodicity_method}. In particular, let 
\begin{align*}
E,~H=L^2(E,\mu),~(S,D),~(A,D),~(L,D),~(L_2,D(L_2)),~\left(T_{t,2}\right)_{t \geq 0},~\mathbb{P}
\end{align*}
and the respective dual objects $(\hat{L},D)$, $(\hat{L}_2,D(\hat{L}_2))$, $\left(\hat{T}_{t,2}\right)_{t \geq 0}$, $\hat{\mathbb{P}}$ be as in (A) and (S) (or (A)' and (S)') in the following. We start with the proof of Corollary \ref{Pp_computation_quadratic_variation}.

\begin{proof}[Proof of Corollary \ref{Pp_computation_quadratic_variation}]
Let $t \geq 0$ and $f \in D$. We have\phantomsection\label{proof_Corollary_quadratic_variation}
\begin{align*}
\mathbb{E}\left[\left( M_t^{[f],L}\right)^2\right] = \mathbb{E}\left[ 2 \int_0^t \Gamma_L(f,f)(X_s) \d s \right] =  2 \int_0^t \mathbb{E}\left[ \,\Gamma_L(f,f)(X_s)  \right] \d s
\end{align*}
since $N^{[f],L}$ is a $\mathbb{P}$-integrable $\left(\mathcal{F}_h^0\right)_{h \geq0}$-martingale which starts at zero. Here in the last equality Fubini's theorem is used, see Lemma \ref{Lm_well_definedness_integrability}\,(ii) for details. By using the invariance $\mathbb{P} \circ X_s^{-1} = \mu$ for each $s \geq 0$ and (A3) we conclude
\begin{align*}
\mathbb{E} \left[ \,\Gamma_L(f,f)(X_s) \,\right] = \int_E \Gamma_L(f,f) \, \mathrm{d}\mu = - \,(Lf,f)_H=-\,(Sf,f)_H,\quad s \geq 0.
\end{align*}
Here the last equality holds since $L=S-A$ on $D$ and $(A,D)$ is antisymmetric on $H$ by (A2). The statement in the dual case follows by using exactly the same arguments with $\mathbb{P}$, $\mathbb{E}$ and $L$ replaced by $\hat{\mathbb{P}}$, $\hat{\mathbb{E}}$ and $\hat{L}$ above.
\end{proof}

With the knowledge about regular s.c.c.s.~and associated laws we are now able to prove Theorem \ref{Thm_Martingale_Problem_via_regular semigroups}. We mainly have to apply Proposition \ref{Pp_martingale_problem}. 

\begin{proof}[Proof of Theorem \ref{Thm_Martingale_Problem_via_regular semigroups}]
First of all, we have the identity\phantomsection\label{proof_Theorem_A_strich_und_s_strich}
\begin{align*}
Lf = \lim_{h \downarrow 0} \frac{1}{h} \left(T_{h,2}f - f\right),\quad f \in D,
\end{align*}
since $L_2=L$ on $D$. Hence by invariance of $\mu$ w.r.t.~$(T_{t,2})_{t \geq 0}$ we conclude that $\mu(Lf)=0$ for all $f \in D$. Moreover, invariance of $\mu$ w.r.t.~$(T_{t,2})_{t \geq 0}$ also implies that $\mathbb{P}$ admits $\mu$ as invariant measure by Lemma \ref{Lm_properties_semigroup_related_law}\,(ii). And since $\hat{\mathbb{P}}_T = \mathbb{P}_T \circ \tau_T^{-1}$ for $T \geq 0$, Lemma \ref{Lm_properties_semigroup_related_law}\,(iii) implies that $\hat{\mathbb{P}}$ is associated with $(\hat{T}_{t,2})_{t \geq 0}$ which is also clearly a sub-Markovian s.c.c.s.~on $L^2(E,\mu)$. The probability measure $\mu$ is furthermore invariant w.r.t.~$(\hat{T}_{t,2})_{t \geq 0}$ since $(T_{t,2})_{t \geq 0}$ is conservative, see Identity \eqref{conservativity_implies_invariance_of_dual_semigroup}. Now Lemma\,\ref{Lm_mu_invariance_and_sub_Marko_implies_regularity} or Lemma\,\ref{Lm_properties_semigroup_related_law}\,(i) shows that $(\hat{T}_{t,2})_{t \geq 0}$ is also conservative. Summarizing, $(\hat{T}_t)_{t \geq 0}$ is also a regular conservative $\mu$-invariant sub-Markovian s.c.c.s.~on $H$ and again by Lemma \ref{Lm_properties_semigroup_related_law}\,(ii) we conclude that $\hat{\mathbb{P}}$ admits $\mu$ as invariant measure. The same argumentation as in the beginning of this proof shows that $\mu(\hat{L}f) =0$ for all $f \in D$ since $\hat{L}_2 =L$ on $D$. So, indeed (A3) holds. The final statements that $M^{[f],L}$ and $N^{[f],L}$ are $(\mathcal{F}_t^{0})_{t \geq 0}$-martingales under $\mathbb{P}$ for each $f \in D$ and that $M^{[f],\hat{L}}$ and $N^{[f],\hat{L}}$ are $(\mathcal{F}_t^{0})_{t \geq 0}$-martingales under $\hat{\mathbb{P}}$ for all $f \in D$ is a consequence of Proposition \ref{Pp_martingale_problem} since $D$ is an algebra and $(L,D) \subset (L_2,D(L_2))$ as well as $(\hat{L},D) \subset (\hat{L}_2,D(\hat{L}_2))$.
\end{proof}

Now we continue proving the claimed statements in our abstract framework from Section \ref{Results_abstract_ergodicity_method}. Therefore, we need one more lemma. Recall that $(G,\overline{D_P}^{G})$ is understood as an operator living on $H_P$.

\begin{Lm} \label{Lm_closedness_range_S_and_range_G}
\begin{itemize}
\item[(i)] Assume Condition (E1). Then the range $\mathcal{R}(S) $ of $(S,\overline{D}^S)$ is a closed subspace of $H$. Thus $\mathcal{R}(S)=\mathcal{N}(S)^\bot$ where $\mathcal{N}(S)$ is the kernel of $(S,\overline{D}^S)$.
\item[(ii)] Assume Condition (E3). Then the range $\mathcal{R}(G) $ of $(G,\overline{D_P}^G)$ is a closed subspace of $H_P$. Thus $\mathcal{R}(G)=\mathcal{N}(G)^\bot$ where $\mathcal{N}(G) \subset H_P$ denotes the kernel of $(G,\overline{D_P}^G)$.
\end{itemize}
\end{Lm}

\begin{proof}
We prove (i). Choose $h \in H$ such that $h =\lim_{n \to \infty} Sf_n$ where $f_n \in \overline{D}^S$, $n \in \mathbb{N}$. W.l.o.g.~we may assume that $Pf_n=0$ for each $n \in \mathbb{N}$. Thus the microscopic inequality in (E1), which carries over to each element from $\overline{D}^S$, yields
\begin{align*}
\|Sg\| \|g\| \geq - \left(Sg, g\right)_H \geq \Lambda_m \, \|g\|^2\quad \mbox{for all $g \in \overline{D}^S$ with $Pg=0$}.
\end{align*}
Thus we see that $(f_n)_{n \in \mathbb{N}}$ is a Cauchy sequence in $H$ with limit denoted by $f \in H$. By closedness of $(S,\overline{D}^S)$ we obtain $f \in \overline{D}^S$ and $h=Sf \in \mathcal{R}(S)$. Thus $\mathcal{R}(S)$ is a closed subset of $H$. The second part of (i) now follows by the well-known identity $\overline{\mathcal{R}(T)} =\mathcal{N}(T)^\bot$ which is satisfied for each selfadjoint operator $(T,D(T))$ on a Hilbert space, see e.g.~\cite[Lem.~8.19]{Gol85}. This finishes the proof of part (i). The proof of (ii) is similar. Indeed, let $h \in \overline{\mathcal{R}(G)}$. Here the closure is understood to be in $H_P$. Thus there exists $g_n \in D_P$, $n \in \mathbb{N}$, such that $Gg_n \to h$ in $H$ as $n \to \infty$. Since $1 \in \overline{D_P}^{G}$ and $G1=0$ we can infer that $G f_n=G g_n$ where 
\begin{align*}
f_n\df g_n - \left(g_n,1\right)_{H_P} \in \overline{D_P}^G \quad \mbox{for all } n \in \mathbb{N}.
\end{align*}
By the macroscopic inequality in (E3), which carries over to each element from $\overline{D}_P^{G}$, we can infer that 
\begin{align*}
\|f_n-f_m\|_{H} \leq \frac{1}{\Lambda_M} \| G \, (f_n-f_m)\|_{H} \quad \mbox{for all }n,m \in \mathbb{N}.
\end{align*}
Thus also $(f_n)_{n \in \mathbb{N}}$ is a Cauchy sequence in $H_P$ with limit denoted by $f \in H_P$. By closedness of $(G,\overline{D_P}^G)$ we obtain 
\begin{align*}
f \in \overline{D_P}^G \quad \mbox{and} \quad h=Gf \in \mathcal{R}(G).
\end{align*}
The second part of (ii) now follows by using additionally selfadjointness of $(G,\overline{D_P}^G)$. 
\end{proof}

We are now arriving at the proof of Proposition \ref{Pp_property_from_microscopic}. We remark that Proposition \ref{Pp_property_from_microscopic} is the generalization of \cite[Prop.~5.1]{GK08} and \cite[Lem.~6.6.10]{Con11} to our abstract setting. Thus also the proof below is the generalization of the proofs corresponding to \cite[Prop.~5.1]{GK08} and \cite[Lem.~6.6.10]{Con11}.

\begin{proof}[Proof of Proposition \ref{Pp_property_from_microscopic}]
By \phantomsection\label{proof_main_ingredient_proposition} Lemma \ref{Lm_closedness_range_S_and_range_G}\,(i) we have $\mathcal{N}(S)^\bot=\mathcal{R}(S)$ where $\mathcal{R}(S)$ is the range of $(S,\overline{D}^S)$. So, there exists $g \in \overline{D}^S$ such that $f=Sg$. Moreover, we may assume $g \in \overline{D}^S \cap \mathcal{N}(S)^\bot$ by using that $H=\mathcal{N}(S) \oplus \mathcal{N}(S)^\bot$. In particular, $Pg=0$. Further note that the microscopic inequality in (E1) is satisfied for each element from $\overline{D}^S$. So, by applying the microscopic inequality from (E1) to the previously chosen $g$ we obtain
\begin{align*}
\|f\| \|g\| \geq -\left(f,g\right)_H \geq \Lambda_m \,\|g\|^2.
\end{align*}
Hence $\|g \| \leq \tfrac{1}{\Lambda_m} \|f\|$. Altogether, it suffices to show that for each $g \in \overline{D}^S$ we have
\begin{align} \label{eq_2_property_from_microscopic}
\mathbb{E} \left[ \left( \frac{1}{t} \int_0^t Sg(X_s) \d s \right)^2 \right] \leq \frac{2}{t} \,\|g\| \|Sg\|.
\end{align}
We show this in the following. Therefore, let first $g \in D$. Note $2\,S = (L + \hat{L})$ on $D$. One easily verifies that
\begin{align} \label{equations_P_a_s}
2 \int_0^t Sg(X_s) \d s = \int_0^t Lg(X_s) \d s + \int_0^t \hat{L}g(X_s) \d s =- M_t^{[g],L} - {M}_t^{[g],\hat{L}} \circ \tau_t.
\end{align}
Above $\tau_t$ is the time-reversal on $C([0,t];E)$ as introduced in Definition \ref{Df_probability_law} and the occurring random variables in \eqref{equations_P_a_s} are considered on $(C([0,t];E),\mathcal{F}_t^0,\mathbb{P}_t)$. Here and in the following $\mathbb{P}_t$ and $\hat{\mathbb{P}}_t$ denote the laws induced by $\mathbb{P}$ and $\hat{\mathbb{P}}$ on $C([0,t];E)$, see Definition \ref{Df_probability_law}. Expectation is denoted by $\mathbb{E}_{\,t}$ and $\hat{\mathbb{E}}_{\,t}$. We remark that the equal signs in \eqref{equations_P_a_s} are understood $\mathbb{P}_t$-a.s.~on $(C([0,t];E),\mathcal{F}_t^0,\mathbb{P}_t)$.  Hence the Minkowski inequality implies
\begin{align*}
2\,\mathbb{E}_{\,t}\left[ \left(\int_0^t Sg(X_s) \d s\right)^2\right]^{\tfrac{1}{2}} &\leq \mathbb{E}_{\,t}\left[ \left(M_t^{[g],L}\right)^2 \right]^{\tfrac{1}{2}} + \mathbb{E}_{\,t}\left[ \left(M_t^{[g],\hat{L}} \circ \tau_t\right)^2 \right]^{\tfrac{1}{2}} \\
&=\mathbb{E}_{\,t}\left[ \left(M_t^{[g],L}\right)^2 \right]^{\tfrac{1}{2}} + \hat{\mathbb{E}}_{\,t}\left[ \left(M_t^{[g],\hat{L}}\right)^2 \right]^{\tfrac{1}{2}}.
\end{align*}
We used that $\hat{\mathbb{P}}_t= \mathbb{P}_t \circ \tau_t^{-1}$ by Condition (S). Now note that 
\begin{align*}
\mathbb{E}_{\,t}\left[ \left(M_t^{[g],L}\right)^2 \right]=\mathbb{E}\left[ \left(M_t^{[g],L}\right)^2 \right],\quad \hat{\mathbb{E}}_{\,t}\left[ \left(M_t^{[g],\hat{L}}\right)^2 \right]=\hat{\mathbb{E}}\left[ \left(M_t^{[g],\hat{L}}\right)^2 \right]
\end{align*}
which follows using Identity \eqref{Rm_identity_restriction_integral_ergodicity}. So, with the help of Corollary \ref{Pp_computation_quadratic_variation} we can further estimate
\begin{align*}
2\,\mathbb{E}_{\,t}\left[ \left(\int_0^t Sg(X_s) \d s\right)^2\right]^{\tfrac{1}{2}} &\leq  2 \,\sqrt{- 2\,t\, \left(Sg,g\right)_H}.
\end{align*}
Thus we get
\begin{align*}
\mathbb{E}\left[ \left(\int_0^t Sg(X_s) \d s\right)^2\right]=\mathbb{E}_{\,t}\left[ \left(\int_0^t Sg(X_s) \d s\right)^2\right] \leq -2 \, t\, \left(Sg,g\right)_H.
\end{align*}
So, \eqref{eq_2_property_from_microscopic} is satisfied for each $g \in D$. Finally, by using that $D$ is  a core for $(S,\overline{D}^S)$ and Estimate \eqref{eq_property_from_microscopic}, by approximation \eqref{eq_2_property_from_microscopic} is indeed fulfilled for all $g \in \overline{D}^S$. 
\end{proof}

Before proving Lemma \ref{Lm_after_H4} recall the definition of $G$ and $\left[LAP\right]$ from Section \ref{section_Results}. We shall mention explicitly that $\left[LAP\right]f$ for $f \in \overline{D_P}^G$ does not mean that $Pf \in \overline{D}^A$, $APf \in \overline{D}^L$ and $LAPf=\left[LAP\right]f$. However, there exists a sequence $f_n \in D$, $n \in \mathbb{N}$, such that $f_n \to f$ and $Gf_n \to Gf$ in $H$ as $n \to \infty$. Using \eqref{Eq_cauchy_APf} we see that $(APf_n)_{n \in \mathbb{N}}$ is also a Cauchy sequence in $H$ with limit denoted by $[AP]f \in H$. Since $LAPf_n \to [LAP]f$ in $H$ as $n \to \infty$, by closedness of $(L,\overline{D}^L)$ we obtain
\begin{align*}
[AP]f \in \overline{D}^L \quad \mbox{and} \quad L[AP]f=[LAP]f.
\end{align*} 
Nevertheless, this relation is nowhere required in the following.

\begin{proof}[Proof of Lemma \ref{Lm_after_H4}]
We prove (i).\phantomsection\label{proof_Lemma_before_maintheorem} We have the relation
\begin{align*}  
-G = PLAP \quad \mbox{on } D_P
\end{align*}
since $L=S-A$ on $AP(D)=AP(D_P)$ by \eqref{Technical_Condition_Ergodicity} and $PS=0$ on $\overline{D}^S$. The latter identity holds since $\mathcal{N}(S)^\bot =\mathcal{R}(S)$ by Lemma \ref{Lm_closedness_range_S_and_range_G}\,(i) and hence $P_{|\mathcal{R}(S)}=0$. Hence the first formula from \eqref{relation2_G_LAP} follows by definition of $\left[LAP\right]$ on $\overline{D_P}^G$. Note that Condition (E2) was essential used here for the construction of $\left[LAP\right]$. Now the second formula follows by using that $PG=G$ on $\overline{D_P}^G$ and the symmetry of $P$. Now we prove (ii). Therefore, let $f \in H_P$. It suffices to show that
\begin{align*}
P_G f = \left(f,1\right)_H
\end{align*}
where $P_G \colon H_P \to \mathcal{N}(G)$ denotes the orthogonal projection onto the kernel $\mathcal{N}(G)$ of $(G,\overline{D_P}^G)$. To prove this, note that if $f$ can be written as $f=g_1+g_2$ for some $g_1 \in \mathcal{N}(G)$ and some $g_2 \in \mathcal{N}(G)^\bot$, then $P_Gf=g_1$. Now we have
\begin{align*}
f=f-\left(f,1\right)_{H} + \left(f,1\right)_{H}.
\end{align*}
But $\left(f,1\right)_H \in \mathcal{N}(G)$ since constant functions are elements from $\mathcal{N}(G)$. So, it is left to show that
\begin{align*}
f-\left(f,1\right)_H \in \mathcal{N}(G)^\bot.
\end{align*}
Therefore, choose an arbitrary $g \in \mathcal{N}(G)$. The macroscopic coercivity inequality from (E3) (which is satisfied for all elements from $\overline{D}^G$) implies  $g=\left(g,1\right)_H$. As desired, we obtain
\begin{align*}
\left( f-\left(f,1\right)_H,g\right)_H = \left( f,g \right)_H - \left(f,1\right)_H \left(g,1\right)_H = 0.
\end{align*}
The proof is finished.
\end{proof}

Now we can prove our desired ergodicity theorem. The idea for the proof of Theorem \ref{main_ergodicity_theorem_with_rate_of_convergence} below is a generalization of the strategy developed originally in the concrete fiber lay-down setting in \cite{GK08}, see the proof of Theorem 5.3 therein. So, the following proof is also the generalization of the proof of \cite[Theo.~4.5]{CG10} (or from \cite[Theo.~6.6.5]{Con11} equivalently) which also relies on the strategy used for proving \cite[Theo.~5.3]{GK08}.

\begin{proof}[Proof of Theorem \ref{main_ergodicity_theorem_with_rate_of_convergence}]
By \phantomsection\label{proof_main_theorem_nr1}replacing $f$ through $f-\left(f,1\right)_H$, it suffices to prove the theorem for all $f \in H$ satisfying $\left(f,1\right)_H=0$. So, w.l.o.g.~we assume $\left(f,1\right)_H=0$ in the following. We further use the decomposition
\begin{align*}
f=f-Pf +Pf\quad \mbox{where}\quad f-Pf \in \mathcal{N}(S)^\bot,\quad Pf \in \mathcal{N}(S).
\end{align*}

\textit{Step 1:} Since $f-Pf$ is an element from $\mathcal{N}(S)^\bot$, we can apply Proposition \ref{Pp_property_from_microscopic} to $f-Pf$ and obtain
\begin{align} \label{Estimate_1_main_theorem}
\left\| \frac{1}{t} \int_0^t (f-Pf)(X_s) \d s  \right\|_{L^2(\mathbb{P})} \leq \frac{\sqrt{2}}{\sqrt{t \,\Lambda_m}} \left\| f-Pf \right\|_H \leq \frac{\sqrt{2}}{\sqrt{t \,\Lambda_m}} \left\| f \right\|_H
\end{align}

\textit{Step 2:} Let us consider $Pf$. Lemma \ref{Lm_after_H4}\,(ii) yields that $-Pf \in \mathcal{N}(G)^\bot$ since $P1=1$ which implies $\left(Pf,1\right)_H=0$. By Lemma \ref{Lm_closedness_range_S_and_range_G}\,(ii) we have $\mathcal{N}(G)^\bot=\mathcal{R}(G)$.  So, there exists $g \in \overline{D_P}^G$ such that
\begin{align*}
-Pf=Gg.
\end{align*}
W.l.o.g.~we may assume that $\left(g,1\right)_H=0$ since $G(g-\left(g,1\right)_H)=Gg$ and $\left(g,1\right)_H \in \overline{D_P}^G$ by the last assumption from (E3). By \eqref{relation2_G_LAP} we obtain
\begin{align*}
Pf=-Gg=P\left[LAP\right]g.
\end{align*}
Hence
\begin{align*}
P\left(Pf - \left[LAP\right]g\right)=0.
\end{align*}
So, the element $Pf - \left[LAP\right]g$ can also be estimated by Proposition \ref{Pp_property_from_microscopic} and we obtain
\begin{align*} 
\left\| \frac{1}{t} \int_0^t (Pf - \left[LAP\right]g)(X_s) \d s  \right\|_{L^2(\mathbb{P})} \leq \frac{\sqrt{2}}{\sqrt{t \,\Lambda_m}} \left\| Pf - \left[LAP\right]g \right\|_H.
\end{align*}
Next, we further estimate the right hand side of the previous inequality. Therefore, by  Lemma \ref{Lm_after_H4}\,(i) observe that
\begin{align*}
\left\| Pf - \left[LAP\right]g \right\|_H^{2} &= \left\| Pf \right\|_H^{2} + \left\| \left[LAP\right]g \right\|_H^{2} - 2 \left(Pf,\left[LAP\right]g \right)_H \\
&= \left\| \left[LAP\right]g \right\|_H^{2} -  \left\| Gg \right\|_H^2 \leq \left\| \left[LAP\right]g \right\|_H^{2}.
\end{align*}
Furthermore, by the Kato-bound from (E2) we can infer that
\begin{align*}
\left\| \left[LAP\right]g \right\|_H \leq c_1\, \|Gg\|_H + c_2\, \|g\|_H \leq \left(c_1 + \frac{c_2}{\Lambda_M}\right)\|f\|_H.
\end{align*}
In the last inequality, we have used that
\begin{align*}
\|Gg\|_H=\|Pf\|_H \leq \|f\|_H
\end{align*}
and the macroscopic coercivity inequality from (E3) which implies
\begin{align*}
-\left( Gg,g\right)_H \geq \Lambda_M \|g - \left(g,1\right)\|_H^2 = \Lambda_M \,\|g\|_H^{2}.
\end{align*}
Altogether, we obtain 
\begin{align} \label{Estimate_2_main_theorem}
\left\| \frac{1}{t} \int_0^t (Pf - \left[LAP\right]g)(X_s) \d s  \right\|_{L^2(\mathbb{P})} \leq \frac{\sqrt{2}}{\sqrt{t \,\Lambda_m}}\,\left(c_1 + \frac{c_2}{\Lambda_M}\right)\|f\|_H.
\end{align}

\textit{Step 3:} It is left to consider $\left[LAP\right]g$ where $g \in \overline{D_P}^G$ is chosen as in Step 2. But first, we need some preceding estimates. Therefore, choose $h \in D$ arbitrary. By using the invariance of $\mu$ w.r.t.~$\mathbb{P}$ (more precisely, see Lemma \ref{Lm_well_definedness_integrability}\,(iii)) and Corollary \ref{Pp_computation_quadratic_variation} we obtain
\begin{align*}
\left\| \frac{1}{t} \int_0^t Lh (X_s) \d s \right\|_{L^2(\mathbb{P})} &\leq \frac{1}{t}\left\| h(X_t) - h(X_0) \right\|_{L^2(\mathbb{P})} + \frac{1}{t} \left\| M_t^{[h],L} \right\|_{L^2(\mathbb{P})}\\
&\leq \frac{2}{t} \,\|h\|_H + \frac{\sqrt{2}}{\sqrt{t}} \,\sqrt{ - \left(Lh,h\right)_H}.
\end{align*}
An approximation yields that the last inequality even carries over to each $h \in \overline{D}^L$ by using \eqref{eq_property_from_microscopic}. And since $AP h \in \overline{D}^L$ for all $h \in D_P$, we get
\begin{align} \label{Equation_simplified:later_on_in _Corollary_with_H5}
\left\| \frac{1}{t} \int_0^t LAP h(X_s) \d s \right\|_{L^2(\mathbb{P})} &\leq \frac{2}{t} \,\|APh\|_H + \frac{\sqrt{2}}{\sqrt{t}} \,\sqrt{ - \left(LAPh,APh\right)_H} \\
&\leq  \frac{2}{t} \,\|APh\|_H + \frac{\sqrt{2}}{\sqrt{t}} \,\left( \|LAPh\|_H \| APh\|_H\right)^{\frac{1}{2}} \nonumber
\end{align}
Using the formula $\|APh\|_H^2 = -\left(Gh,h\right)_H$ for $h \in D_P$ and again an approximation via \eqref{eq_property_from_microscopic}, we obtain the desired estimation for $\left[LAP\right]g$ as
\begin{align*}
&\left\| \frac{1}{t} \int_0^t \left[LAP\right] g(X_s) \d s \right\|_{L^2(\mathbb{P})} \\
&\leq  \frac{2}{t} \, \sqrt{-\left(Gg,g\right)_H} + \frac{\sqrt{2}}{\sqrt{t}} \,\left( \|\left[LAP\right]g\|_H \, \sqrt{-\left(Gg,g\right)_H}\right)^{\frac{1}{2}}
\end{align*}
From the computation in  Step 2 we have
\begin{align*}
\|\left[LAP\right]g\|_H \leq \left(c_1 + \frac{c_2}{\Lambda_M}\right)\|f\|_H\quad \mbox{and}\quad \sqrt{-\left(Gg,g\right)_H} \leq \frac{1}{\sqrt{\Lambda_M}} \, \|f\|_H
\end{align*}
Thus we finally obtain
\begin{align} \label{Estimate_3_main_theorem}
&\left\| \frac{1}{t} \int_0^t \left[LAP\right] g(X_s) \d s \right\|_{L^2(\mathbb{P})} \leq \left(\frac{2}{t \, \sqrt{\Lambda_M}} + \frac{\sqrt{2\,c_1 + 2\,c_2\,\Lambda_M^{-1}}}{\sqrt{t}\, \sqrt[4]{\Lambda_M}} \right)\|f\|_H
\end{align}

\textit{Step 4:}  By using the decomposition
\begin{align*}
f=f-Pf +Pf - \left[LAP\right]g +\left[LAP\right]g
\end{align*}
with $g$ as in Step 2 and Estimates \eqref{Estimate_1_main_theorem}, \eqref{Estimate_2_main_theorem} and \eqref{Estimate_3_main_theorem}, the claim follows. Therefore, again use that $\left\|h\right\|_{L^2(\mu)}=\left\|h(X_0)\right\|_{L^2(\mathbb{P})}$ for each $h \in L^2(\mu)$.
\end{proof}

Theorem \ref{Cor_main_ergodicity_theorem_with_rate_of_convergence} now directly follows by modifying Step 3 of the previous proof through using the additionally introduced algebraic relation from (E4), see next. 

\begin{proof}[Proof of Theorem \ref{Cor_main_ergodicity_theorem_with_rate_of_convergence}]
By using (E4) in Step 3 of the proof of Theorem \ref{main_ergodicity_theorem_with_rate_of_convergence} and the relation $(LAPh,APh)_H=(SAPhAPh)_H$ for each $h \in D_P$ (since $(A,\overline{D}^A)$ is antisymmetric), Estimate \eqref{Equation_simplified:later_on_in _Corollary_with_H5} can be simplified as\phantomsection\label{proof_main_theorem_corollary}
\begin{align*}
\left\| \frac{1}{t} \int_0^t LAP h(X_s) \d s \right\|_{L^2(\mathbb{P})} &\leq \frac{2}{t} \,\|APh\|_H + \frac{\sqrt{2}}{\sqrt{t}} \,\sqrt{ c_3}\, \sqrt{- \left(Gh,h\right)_H},\quad h \in D_P.
\end{align*}
Now follow exactly the argumentation from the end of the proof to Step 3 of Theorem \ref{main_ergodicity_theorem_with_rate_of_convergence} to verify the claim.
\end{proof}




Before going on, we finally formulate Assumption (E3) in a different, but equivalent way. This equivalent formulation connects Assumption (E3) with the corresponding macroscopic coercivity assumption (H3) introduced in the hypocoercivity method in \cite[Sec.~1.3]{DMS13} or \cite{GS12B}. In this hypocoercivity setting, namely, the operator $PA^2P$ is considered as an operator living on $H$. For avoiding a bad notation, only for the upcoming lemma, the notation is changed: We denote the previously introduced operator $(G,\overline{D_P}^G)$ on $H_P$ more precisely as $(G_P,\overline{D_P}^{G_P})$ and further introduce $(G,D)$ as
\begin{align*}
G=PA^2P \quad \mbox{on }D.
\end{align*}
The latter is a densely defined, dissipative operator on $H$. Its closure on $H$ we denote by $(G,\overline{D}^G)$ in the upcoming lemma. 

\begin{Lm} \label{Lm_equivalent_condition_for_H4} Assume the technical condition from \eqref {Technical_Condition_Ergodicity}.
\begin{itemize}
\item[(i)] $(G,D)$ is essentially selfadjoint on $H$ if and only if $(G_P,D_P)$ is essentially selfadjoint on $H_P$ and $D_P \subset \overline{D}^G$.
\item[(ii)] Assume that $1 \in \overline{D}^S \cap \overline{D}^G$ and $S1=G1=0$. Then also $1 \in \overline{D_P}^{G_P}$ and $G_P 1=0$. Moreover, the macroscopic inequality in (E3) equivalently means that 
\begin{align} \label{spectral_gap_G}
-\left(Gf,f\right)_{H} \geq \Lambda_M \,\|Pf-\left(f,1\right)_H\|^2_{H} \quad \mbox{for all }f \in D.
\end{align}
\end{itemize}
\end{Lm}

\begin{proof} We start with (i). The proof is a generalization of the specific calculations done in the proof of \cite[Prop.~3.13]{GS12B}. First note that essential selfadjointness of $(G,D)$ on $H$ implies $D_P \subset \overline{D}^G$. Indeed, the identity $\left(Gf,h \right)_H =\left(Gf,g \right)_H =\left(f,Gg \right)_H$ for each $f \in \overline{D}^G$ and each $h=Pg$, $g \in D$, yields $h \in \overline{D}^G$ and $Gh=Gg$ by self-adjointness of $(G,\overline{D}^G)$. Thus for the rest of this proof we assume w.l.o.g.~$D_P \subset \overline{D}^G$; Now note that $(G,D)$ (and $(G_P,D_P)$ respectively) is a densely defined, symmetric and nonpositive definite linear operator on the respective Hilbert space. Thus essential selfadjointness is equivalent to essential m-dissipativity of these both operators on the respective Hilbert space, see e.g.~\cite[Rem.~1.1.20]{Con11}. Thus it remains to verify that $(I-G)(D)$ is dense in $H$ if and only if $(I-G_P)(D_P)$ is dense in $H_P$. Assume first that $(I-G)(D)$ is dense in $H$. Let $g \in H_P$ such that 
\begin{align*}
\left(g,(I-G_P)h\right)_{H_P}=0 \quad \mbox{for all }h=Pf,~f \in D.
\end{align*}
We have to show that $g$ must be zero. But this is clear since 
\begin{align*}
0=\left(g,(I-G_P)h\right)_{H_P}=\left(g,(I-G)f\right)_H
\end{align*}
for each such $f \in D$ with $h=Pf$. Here we have used that $G_Ph=Gf$, $PG=G$, $Pg=g$ and the symmetry of $P$. Thus $g=0$ since $(I-G)(D)$ is dense in $H$. Now let $(I-G_P)(D_P)$ be dense in $H_P$ and let $g \in H$ be chosen such that
\begin{align} \label{Eq_Lm_equivalent_condition_for_H4}
\left(g,(I-G)h\right)_H=0 \quad \mbox{for all }h \in D.
\end{align}
Again we need to show that $g=0$. Note that the previous equation carries over to each $h \in \overline{D}^G$, in particular is satisfied for all $h=Pf$ with $f \in D$ since $D_P \subset \overline{D}^G$. Thus for each such $h$ we obtain
\begin{align*}
\left(Pg,(I-G_P)h\right)_{H_P}=\left(Pg,(I-G)h\right)_H=\left(g,(I-G)h\right)_H=0.
\end{align*}
Here we used that  $D_P \subset \overline{D}^G$ easily implies the relation $G_{|D_P}=G_P$ on $D_P$. Thus $Pg=0$ since $(I-G_P)(D_P)$ is dense in $H_P$. Then Equation \eqref{Eq_Lm_equivalent_condition_for_H4} implies
\begin{align*}
\left(g,h\right)_H=\left(g,Gh\right)_H=\left(g,PGh\right)_H=\left(Pg,Gh\right)_H=0 \quad \mbox{for all } h \in D.
\end{align*}
This finally yields that $g=0$ because $D$ is dense in $H$. Part (ii) directly follows by a straightforward calculation.
\end{proof}


Now we go on and apply the abstract ergodicity framework to the $N$-particle Langevin dynamics.

\subsection{Proofs to Section \ref{Results_Applications_N_particle_Langevin} (The Langevin dynamics)} \label{subsection_proofs_to_Langevin_dynamics}

First recall the definitions and notations introduced in Section \ref{Results_Applications_N_particle_Langevin}. For the verification of the assumptions, we recapitulate sometimes calculations done in \cite{CG10} for a better understanding. We start with the proof of Proposition \ref{Pp_(P)_implies_A_and_S_Langevin}.

\begin{proof}[proof of Proposition \ref{Pp_(P)_implies_A_and_S_Langevin}]

Recall that the state space for the $N$-particle system is
\begin{align*}
E= \widetilde{\mathbb{R}^{dN}}\times \mathbb{R}^{dN} \quad \mbox{where}\quad \widetilde{\mathbb{R}^{dN}}\df \left\{ x \in \mathbb{R}^{dN}~\big|~\Phi(x) < \infty \right\}. 
\end{align*} 
We remark that this space is denoted with the symbol $\widetilde{E}$ in \cite{CG10}. Due to the first assumption from (C0), note that $E$ is an open subset of $\mathbb{R}^{dN} \times \mathbb{R}^{dN}$. Thus $E$ is equipped with the usual Euclidean metric. Accordingly to (A1) we set
\begin{align*}
H=L^2(E,\mu_{\Phi})=L^2(\mathbb{R}^{dN} \times \mathbb{R}^{dN},\mu_{\Phi,\beta}).
\end{align*}
Further note that $\mu_{\Phi,\beta}(E)=1$ by (P1). The set $D$ from (A2) is given by 
\begin{align*}
D\df C_c^\infty(E).
\end{align*}
The latter denotes the space of all infinitely often differentiable functions with compact support in $E$. Standard arguments imply that $D$ is dense in $H$. For the verification, use that $C_c(E)$ is dense in $L^2(E,\mu_{\Phi})$ (see \cite[Theo.~29.14]{Bau92}) and combine this with the fact that $C_c^\infty(E)$ is dense in $C_c(E)$ w.r.t.~the sup-norm. Here the last statement is implied by the extended version of the Stone-Weierstra"{s} theorem, see \cite[Sec.~7.38]{Sim63}.

Now the operator $(L,D)$ on $H$ is defined via \eqref{Eq_generator_N_particle_Langevin}. It is clearly well-defined by the assumptions from (C0). $S$ and $A$ are defined as
\begin{align} \label{Eq_generator_S_and_A_N_particle_Langevin}
S=\frac{\alpha}{\beta} \, \Delta_\omega - \alpha~ \omega \cdot \nabla_\omega \quad \mbox{and} \quad A=- \omega \cdot \nabla_x + \frac{1}{\beta}\,\nabla_x \Phi \cdot \nabla_\omega.
\end{align}
Integration by parts shows that $(S,D)$ and $(A,D)$ are satisfying the desired properties from (A2) and by using the same argument again, also (A3) is easily checked. For more details see \cite[Sec.~6.2]{Con11}. 

Of course, the construction of the laws $\mathbb{P}$ and $\hat{\mathbb{P}}$ on $C([0,\infty),E)$ solving the martingale problems as needed in (S) requires hard work. For the existence of the laws consider \cite[Theo.~3.1(iii)]{CG10} and \cite[Rem.~3.3]{CG10} which is based on modern methods from \cite{BBR06}. For the statements on the solutions to the required martingale problems see e.g.~\cite[Lem.~3.21]{CG10}. More precisely, the martingale problems are implied by Theorem \ref{Thm_Martingale_Problem_via_regular semigroups} in the Langevin case. Indeed, in \cite{CG10} a regular conservative $\mu_{\Phi,\beta}$-invariant s.c.c.s.~$(T_{t,2})_{t \geq 0}$ on $L^2(\mu_{\Phi,\beta})$ as required in (A)' is constructed which is associated with $\mathbb{P}$ as in (S)'. The existence of the desired semigroup and associatedness with the law $\mathbb{P}$ is proven in \cite{CG10}, see Theorem 3.1 therein, and is based on showing essential m-dissipativity of $(L,D)$ in $L^1(E,\mu_\Phi)$. Moreover, as desired, the generator $(L_2,D(L_2))$ indeed extends $(L,D)$ and the adjoint $(\hat{L}_2,D(\hat{L}_2))$ extends $(\hat{L},D)$. This is stated in \cite[Theo.~3.1]{CG10} and in \cite[Lem.~3.16(ii)]{CG10}.
\end{proof}

\begin{Rm} 
\begin{itemize}
\item[(i)] We mention that the expression regular is not used in \cite{CG10}. So, in this context, recall the definition of regularity from Section \ref{Regular sub-Markovian semigroups and associated laws} as well as Proposition \ref{Pp_regular_sccs}. With this at hand \cite[Theo.~3.1]{CG10} then really implies (A)'. The semigroup existence statement can equivalently also be found in \cite{Con11}, see Theorem 6.4.1 therein.
\item[(ii)] The reader may be confused since the definition of associatedness as in \cite{CG10} does not quite coincide with the definition of associatedness given in Section \ref{section_Results} (which is taken from \cite[Def.~2.1.3]{Con11}). However, the associatedness definition of $\mathbb{P}$ with $(T_{t,2})_{t \geq 0}$ from \cite{CG10} implies the desired associatedness condition as defined in Section \ref{section_Results}. This follows by a monotone class argument as in the proof of \cite[Lem.~2.1.4]{Con11}. Alternatively, for the existence of $\mathbb{P}$ and $\hat{\mathbb{P}}$ one may also consider \cite{Con11} directly, see Corollary 6.4.3 and Remark 6.4.4 therein. 
\end{itemize}
\end{Rm}

Now we verify (E1)-(E4). Some calculations are similar to calculations done in \cite{GS14} in a hypocoercivity setting or as in  \cite{CG10} where the orginal ergodicity elaboration for the $N$-particle Langevin dynamics can be found. In order to have a complete argumentation and presentation in this section, we stay detailed below. Of course, we always assume that Condition (C0) from Section \ref{Results_Applications_N_particle_Langevin} is satisfied without mention this explicitly. Conditions (C1)-(C3) are only needed in order to verify (E2) and (E3) and are therefore assumed later on. We start with (E1) and introduce  $P\colon H \to H$ in the ergodicity framework by
\begin{align*}
Pf=\int_E f \d \nu_\beta \quad \mbox{for each } f \in H.
\end{align*}
Note that $P$ can also be defined on $L^2(\mathbb{R}^{dN},\nu_\beta)$ and is an orthogonal projection in both cases. Recall the well-known fact that the Ornstein-Uhlenbeck operator $(S,C_c^\infty(\mathbb{R}^{dN}))$ is a nonpositive definite, essentially selfadjoint operator on $L^2(\mathbb{R}^{dN},\nu_\beta)$. Moreover, recall the Poincaré inequality for the Gaussian measure (see \cite{Be89}) which easily implies
\begin{align*}
\|\nabla_x f \|^2_{L^2(\nu_\beta)} \geq \beta \,\| f- Pf \|^2_{L^2(\nu_\beta)} \quad \mbox{for all } f \in C_c^\infty(\mathbb{R}^{dN}).
\end{align*}
We have the following (well-known) lemma. 

\begin{Lm} \label{Lm_Verifying_H1_N_particle_Langevin}
Assume that $\Phi$ satisfies (C0). The operator $(S,D)$ is nonpositive definite and essentially selfadjoint on $H$. For the kernel $\mathcal{N}(S)$ of its closure on $H$ we have
\begin{align*}
\mathcal{N}(S)=\mathcal{R}(P)=L^2(\widetilde{\mathbb{R}^{dN}},e^{-\Phi}\mathrm{d}x)
\end{align*}
Here $L^2(\widetilde{\mathbb{R}^{dN}},e^{-\Phi}\mathrm{d}x)$ is viewed as canonically be embedded in $L^2(E,\mu_{\Phi,\beta})$. In other words, $P$ is really the orthogonal projection onto $\mathcal{N}(S)$. Moreover,
\begin{align*}
-\left(Sf,f\right)_H \geq \alpha\,\| f- Pf \|^2_{H} \quad \mbox{for all } f \in D.
\end{align*}
In particular, Condition (E1) holds with $\Lambda_m=\alpha$.
\end{Lm}

\begin{proof}
As seen in \cite{CG08}, essential selfadjointness of $(S,C_c^\infty(\mathbb{R}^{dN}))$ in $L^2(\mathbb{R}^{dN},\nu_\beta)$ implies essential selfadjointness of $(S,C_c^\infty(\widetilde{\mathbb{R}^{dN}}) \otimes C_c^\infty(\mathbb{R}^{dN}))$ in $H$ by \cite[Theo.~VIII.33]{RS80}. Thus essential selfadjointness of $(S,D)$ in $H$ follows since selfadjoint operators do not possess proper symmetric extensions. Now the Poincaré inequality above yields
\begin{align*}
-\left(Sf,f\right)_H = \frac{\alpha}{\beta} \, \|\nabla_x f \|^2_{H} \geq \alpha \| f- Pf \|^2_{H}\quad \mbox{for all } f \in D.
\end{align*}
Clearly, the last inequality then carries over to each $f \in \overline{D}^S$. Hence for some $f \in \overline{D}^S$ with $Sf=0$ it follows that $f=Pf$. Vice versa, standard approximation shows that 
\begin{align*}
L^2(\widetilde{\mathbb{R}^{dN}},e^{-\Phi}\mathrm{d}x) \subset \overline{D}^S\quad \mbox{and} \quad Sf=0 \quad \mbox{for each }f \in L^2(\widetilde{\mathbb{R}^{dN}},e^{-\Phi}\mathrm{d}x).
\end{align*}
The proof is finished.
\end{proof}

Now we verify the technical condition in \eqref{Technical_Condition_Ergodicity}. In this context we also prove (E4). Part (i) of the upcoming lemma is similar to \cite[Lem.~3.7]{CG10}. Since notations differ below, we present the full proof. 

\begin{Lm} \label{Lm_verification_H2_langevin}
Assume that $\Phi$ satisfies (C0).
\begin{itemize}
\item[(i)]
Then $P(D)=C_c^\infty(\widetilde{\mathbb{R}^{dN}}) \subset \overline{D}^A$ and $Af=-\omega \cdot \nabla_x f$ for all $f \in C_c^\infty(\widetilde{\mathbb{R}^{dN}})$. Moreover, for each such $f$ we also have that 
\begin{align*}
g\df Af \in \overline{D}^A \cap \overline{D}^S \cap \overline{D}^L \quad \mbox{and} \quad Lg=Sg-Ag
\end{align*}
with the natural representation formulas for $Sg$ and $Ag$ as in  \eqref{Eq_generator_S_and_A_N_particle_Langevin}. 
\item[(ii)] Condition (E4) is fulfilled with $c_3=\alpha$.
\end{itemize}
\end{Lm}

\begin{proof}
We prove (i). Let $f \in C_c^\infty(\widetilde{\mathbb{R}^{dN}})$. Choose some $\varphi \in C_c^\infty(\mathbb{R}^{dN})$ such that $0 \leq \varphi \leq 1$, $\varphi =1$ on $B_1(0)$ and $\varphi=0$ outside $B_2(0)$ where $B_r(z)$ denotes the open ball w.r.t.~the Euclidean norm of radius $r>0$ around the point $z \in \mathbb{R}^{dN}$.  Define 
\begin{align*}
\varphi_n(z) \df \varphi(\frac{z}{n}) \quad \mbox{for each $z \in \mathbb{R}^{dN}$, $n \in \mathbb{N}$}.
\end{align*}
Then there exists a constant $C < \infty$, independent of $n \in \mathbb{N}$, such that
\begin{align} \label{eq_boundedness_derivatives_varphi_n}
|\partial_{i} \varphi_n(z)| \leq  \frac{C}{n}, \quad |\partial_{ij} \varphi_n(z )| \leq \frac{C}{n^2}\quad \mbox{for all } z \in \mathbb{R}^{dN},~1 \leq i,j \leq dN.
\end{align}
Moreover, clearly $0 \leq \varphi_n \leq 1$ for all $n \in \mathbb{N}$ and $\varphi_n \to 1$ pointwisely on $\mathbb{R}^{dN}$ as $n \to \infty$. Define $f_n(x,\omega)\df f(x) \,\varphi_n(\omega)$ for $(x,\omega) \in E$ and $n \in \mathbb{N}$. Then dominated convergence implies
\begin{align*}
\omega \cdot \nabla_x f_n =  \varphi_n ~\omega \cdot \nabla_{x} f  \stackrel{n \rightarrow \infty}{\longrightarrow} \omega \cdot \nabla_{x} f,\quad \nabla_x \Phi \cdot \nabla_\omega f_n =  f ~\nabla_x \Phi \cdot \nabla_{\omega} \varphi_n  \stackrel{n \rightarrow \infty}{\longrightarrow} 0
\end{align*} 
with convergence in $H=L^2(e^{-\Phi} \mathrm{d} x \otimes \nu_\beta)$. Here we have used that $|\omega| \in L^2(\nu_\beta)$, \eqref{eq_boundedness_derivatives_varphi_n} and $\nabla_x \Phi \in L^2(e^{-\Phi} \mathrm{d}x)$. Thus $f \in \overline{D}^A$ and the formula for $Af$ is shown.\\
Now let $g$ be of the form $g=\omega_i \, h$ where $h \in C_c^\infty(\widetilde{\mathbb{R}^{dN}})$. Here $\omega_i$ denotes the coordinate function $\mathbb{R}^{dN} \ni \omega \mapsto \omega _i \in \mathbb{R}$ for some $1 \leq i \leq dN$. Define $g_n$, $n \in \mathbb{N}$, by 
\begin{align*}
g_n(x,\omega)=\varphi_n(\omega) \, \omega_i \, h(x) \quad \mbox{for all } (x,\omega) \in E.
\end{align*}
Then again by dominated convergence in combination with $|\omega|,|\omega|^2 \in L^2(\nu_\beta)$ and \eqref{eq_boundedness_derivatives_varphi_n} we can infer that
\begin{align*}
Sg_n = &\frac{\alpha}{\beta} \,h \left( 2\,\partial_{\omega_i} \varphi_n + \omega_i \,\Delta_\omega \,\varphi_n \right) - \alpha \,g_n - \alpha \, h\,\omega_i ~ \omega \cdot \nabla_\omega \varphi_n \\
&\stackrel{n \rightarrow \infty}{\longrightarrow} -\alpha \,g=\left(\frac{\alpha}{\beta} \,\Delta_\omega - \alpha \,\omega \cdot \nabla_\omega \right)g
\end{align*}
with convergence in $H$. Similarly, we obtain
\begin{align*}
\omega \cdot \nabla_x \,g_n = \varphi_n \, \omega_i ~\omega \cdot \nabla_x h \stackrel{n \rightarrow \infty}{\longrightarrow} \omega \cdot \nabla_x g
\end{align*}
as well as
\begin{align*}
\nabla_x \Phi \cdot \nabla_\omega g_n =h\,\omega_i\,\nabla_x\Phi \cdot \nabla_\omega \varphi_n + \varphi_n \, h \, \partial_{x_i} \Phi \stackrel{n \rightarrow \infty}{\longrightarrow} \nabla_x \Phi \cdot \nabla_\omega g
\end{align*}
with convergence in $H$ in each case. Now we prove (ii). Thus for $f \in D$ we denote
\begin{align} \label{Df_f_S_Langevin_case}
f_S\df Pf \in C_c^\infty(\widetilde{\mathbb{R}^{dN}}).
\end{align}
By part (i) we obtain for each $f \in D$ that
\begin{align*}
SAPf=-S\left(\omega \cdot \nabla_x f_S\right)=\alpha~ \omega \cdot \nabla_x f_S = -\alpha \,APf.
\end{align*}
In particular, the algebraic relation from Condition (E4) is fulfilled with $c_3=\alpha$. 
\end{proof}

Next we verify (E2) and (E3). First we compute $(G,D_P)=(PA^2P,D_P)$ and still assume that (C0) is satisfied. So, by using the notations from Section \ref{Results_abstract_ergodicity_method} we have
\begin{align*}
D_P=P(D)=C_c^\infty(\widetilde{\mathbb{R}^{dN}}),\quad H_P=P(H)=L^2(\widetilde{\mathbb{R}^{dN}},e^{-\Phi}\mathrm{d}x)=H_{\Phi}
\end{align*} 
Furthermore, for $f \in D$ we get
\begin{align*}
A^2Pf=- A \left( \omega \cdot \nabla_x f_S \right) = \left( \omega, \nabla^2_x f_S \,\omega\right)_{\text{euc}} - \frac{1}{\beta}\,\nabla_x \Phi  \cdot \nabla_x f_S
\end{align*}
where $\nabla^2_x$ denotes the Hessian matrix in Euclidean space and $f_S$ is defined in \eqref{Df_f_S_Langevin_case}. Thus observe 
\begin{align*}
PA^2Pf=\frac{1}{\beta} \,\Delta_x f_S - \frac{1}{\beta} \,\nabla_x \Phi  \cdot \nabla_x f_S \quad \mbox{for all } f \in D.
\end{align*}
And therefore
\begin{align} \label{Relation_G_G_psi}
G= \frac{1}{\beta} \left( \Delta_x - \nabla_x \Phi  \cdot \nabla_x \right) = \frac{1}{\beta} G_{\Phi} \quad \mbox{on } D_P=C_c^\infty(\widetilde{\mathbb{R}^{dN}}).
\end{align}
This connects $G$ with $G_{\Phi}$ and (E2), (E3) can easily be verified, see next. However, first note that (C1) implies that there exists constants $K_i(\Phi) \in [0,\infty)$, $1 \leq i \leq 4$, independent of $f \in C_c^\infty(\widetilde{\mathbb{R}^{dN}})$ and only depending on the choice of $\Phi$ such that
\begin{align} \label{Eq_Kato_bound_E3}
&\sum_{i,j=1}^{dN} \left\| \partial_{x_i} \partial_{x_j} f \right\|_{H_\Phi} \leq K_1(\Phi) \left\| G_\Phi f \right\|_{H_\Phi} + K_3(\Phi) \left\| f \right\|_{H_\Phi},\\
&\sum_{i=1}^{dN} \left\| \left(\partial_{x_i}\Phi \right) \, \partial_{x_i} f \right\|_{H_\Phi} \leq K_2(\Phi) \left\| G_\Phi f \right\|_{H_\Phi} + K_4(\Phi) \left\| f \right\|_{H_\Phi}. \nonumber
\end{align}

\begin{Lm} \label{Lm_Verifying_H3_H4_N_particle_Langevin}
Assume that $\Phi$ satisfies (C0)-(C3). Then the ergodicity conditions (E2) and (E3) are fulfilled. Moreover, the constants $c_1$ and $c_2$ from (E2) can be computed as
\begin{align*}
c_1= \frac{\alpha\, \sqrt{\beta}}{\sqrt{\text{\normalfont{gap}}(G_{\Phi})}} + \sqrt{3} \,  K_1(\Phi) +  K_2(\Phi),\quad c_2=  \frac{\sqrt{3}}{\beta} \, K_3(\Phi) + \frac{1}{\beta} \,K_4(\Phi).
\end{align*}
Here the constants $K_i(\Phi) \in [0,\infty)$ are obtained by the Kato-bound from (C1) in Equation \eqref{Eq_Kato_bound_E3} and are only depending on the choice of $\Phi$. Moreover,
\begin{align*}
\Lambda_M = \frac{\text{\normalfont{gap}}(G_{\Phi})}{\beta}.
\end{align*}
\end{Lm}

\begin{proof}
We first verify (E3). By the relation of $G$ with $G_\Phi$ from \eqref{Relation_G_G_psi} essential selfadjointness of $(G,D_P)$ as required in (E3) follows by (C2). Clearly, $1 \in H_P$. Now also $1 \in \overline{D_P}^G$ with $G1=0$. This follows as explained in \cite[p.~175]{Con11}. Indeed, integration by parts imply $\left(1,Gf\right)_{H_P}=0$ for each $f \in D_P$. Thus $1 \in \overline{D_P}^G$ and $G1=0$ since $(G,\overline{D_P}^G)$ is selfadjoint on $H_P$. The spectral gap condition in (C3) can now equivalently be restated as the macroscopic coercivity inequality required in (E3). Then $\Lambda_M$ can be chosen as $\Lambda_M=\frac{\text{\normalfont{gap}}(G_{\Phi})}{\beta}$. Next we prove (E2). We mention that the Kato-boundedness condition in (E2) is exactly the statement of \cite[Lem.~4.9]{CG10}. For completeness and in view of computing the rate of convergence we recall the argument here. Let $f \in C_c^\infty(\widetilde{\mathbb{R}^{dN}})$, thus $f=Pg=g_S$ for a suitable $g \in D$. Then the previous calculations give
\begin{align*}
LAPf=LAPg = \alpha~ \omega \cdot \nabla_x f - \sum_{i,j=1}^{dN} \omega_i \, \omega_j \,\partial_{x_i}\partial_{x_j} f   + \frac{1}{\beta}\,\nabla_x \Phi  \cdot \nabla_x f
\end{align*}
Further note $\| \nabla_x f \|^2_{H_P}= - \beta \, \left(G f, f \right)_{H_P}$ and $\| \omega \cdot \nabla_x f\|^2_H=\frac{1}{\beta} \|\nabla_x f\|_{H_P}^2$. Hence 
\begin{align*}
\| \omega \cdot \nabla_x f\|^2_H= -\left(Gf,f\right)_{H_P} = -\left(Gh,h\right)_{H_P} \leq \frac{1}{\Lambda_M} \, \| G h \|^2_{H_P} = \frac{\beta}{\text{\normalfont{gap}}(G_{\Phi})} \, \|Gf\|^2_H.
\end{align*}
where $h \df f - \left (f,1\right)_{H_\Phi} \in \overline{D_P}^G$. Moreover, one easily computes that 
\begin{align*}
\left\|\omega_i^2\right\|_{L^2(\mathbb{R}^{dN},\nu_\beta)}^2=\frac{3}{\beta^2},\quad \left\|\omega_i\, \omega_j\right\|_{L^2(\mathbb{R}^{dN},\nu_\beta)}^2=\frac{1}{\beta^2} \mbox{ for } i \not=j,\quad 1 \leq i,j \leq dN.
\end{align*}
Altogether, this yields
\begin{align*}
\|LAPf\|_H \leq &\frac{\alpha \, \sqrt{\beta}}{\sqrt{\text{\normalfont{gap}}(G_{\Phi})}} \|Gf\|_H  + \frac{\sqrt{3}}{\beta} \,\sum_{i,j=1}^{dN} \left\| \partial_{x_i}\partial_{x_j} f \right\|_{H_{P}} + \frac{1}{\beta} \, \left\|\nabla_x \Phi  \cdot \nabla_x f\right\|_{H_{P}}
\end{align*}
By the Kato-bound from (C1) (see Equation \eqref{Eq_Kato_bound_E3}) and Relation \eqref{Relation_G_G_psi} Condition (E2) is satisfied with the claimed values for $c_1$ and $c_2$.
\end{proof}

Altogether, we are are able to verify Theorem \ref{Ergodicity_theorem_N_particle_Langevin}.

\begin{proof}[proof of Theorem \ref{Ergodicity_theorem_N_particle_Langevin}] 
Apply Corollary \ref{Cor_main_ergodicity_theorem_with_rate_of_convergence} and therefore use Proposition \ref{Pp_(P)_implies_A_and_S_Langevin} together with Lemma \ref{Lm_Verifying_H1_N_particle_Langevin}, Lemma \ref{Lm_verification_H2_langevin} and Lemma \ref{Lm_Verifying_H3_H4_N_particle_Langevin}. Then ergodicity with rate of convergence follows. The quantitative description of the rate in dependence of $\alpha,\beta \in (0,\infty)$ follows by a straightforward calculation. Indeed, one only has to use the concrete formulas for $\kappa_1$ and $\kappa_2$ from Corollary \ref{Cor_main_ergodicity_theorem_with_rate_of_convergence} and has to plug in the explicit expressions for $c_1, c_2, c_3,\Lambda_m$ and $\Lambda_M$ that are calculated above. Then the constants $A(\Phi)$ and $B(\Phi)$ from the statement are given by
\begin{align} \label{Specification_of_ni}
A(\Phi) = \sqrt{6} \, K_1(\Phi) + \sqrt{2}  \big( K_2(\Phi) +1 \big),\quad B(\Phi) = \sqrt{6} \, K_3(\Phi) + \sqrt{2} \, K_4(\Phi)
\end{align}
where the $K_i(\Phi) \in [0,\infty)$ are the constants occurring in the Kato-bound from \eqref{Eq_Kato_bound_E3}.
The proof is finished.
\end{proof}

\begin{Persp} \label{Outlook_ergodicity}
We have seen that our abstract ergodicity method applies to the $N$-particle Langevin dynamics with singular potentials. It is of interest to establish ergodicity also for the manifold-valued version of the Langevin dynamics. This manifold-valued version of the Langevin equation is derived e.g.~in \cite{GS12} and in \cite{LRS12} (where it is called the constrained Langevin dynamics). The interest for discussing ergodicity of the latter equation arised in \cite{LRS12} in which an ergodic statement for the constrained Langevin dynamics is outlined without convergence rate, see \cite[Prop.~3.2]{LRS12}.
\end{Persp}

Finally, let us already prove Proposition \ref{Lm_sufficient_criteria_implying_P_E_fiber_case} since it also fits to the situation considered here.

\begin{proof}[proof of Proposition \ref{Lm_sufficient_criteria_implying_P_E_fiber_case}]  First, we note that $\nabla_x \Phi \in H_\Phi$  is indeed satisfied due to \cite[Lem.~A24]{Vil09}. Thus Condition (C0) clearly holds. Moreover, our assumptions imply that (C2) is fulfilled, see \cite[Theo.~7]{BKR97} or \cite[Theo.~3.1]{Wie85}. (C3) obviously follows from the Poincar\'e inequality assumed for the potential $\Phi$. For the verification of (C1), apply \cite[Lem.~4.8]{CG10} which explicitly uses the bound $c< \infty$ on the growth behavior for $\nabla_x^2\Phi$.
\end{proof}

\subsection{Proofs to Section \ref{Results_Applications_generalized_fiber_lay_down} (The generalized fiber lay-down dynamics)} \label{subsection_proofs_to_generalized_fiber_lay_down}

Now we follow the definitions and notations for the generalized fiber lay-down dynamics introduced in Section \ref{Results_Applications_generalized_fiber_lay_down} and we always let $d \in \mathbb{N}$, $d \geq 2$. We start recapitulating some properties concerning the fiber lay-down generator proven in \cite{GS12B}. Moreover, we sometimes recall calculations done in \cite{GS12B} for a better understanding. The following theorem is proven in \cite{GS12B}.

\begin{Thm} \label{Thm_basic_properties_fiber_lay_down_generator}
Assume that $\Phi\colon \mathbb{R}^d \rightarrow \mathbb{R}$ satisfies (C0) and let $\sigma > 0$. Define
\begin{align*}
D:=C_c^\infty(\mathbb{M}),\quad H:=L^2(\mathbb{M},\mu_\Phi),\quad \mathbb{M}=\mathbb{R}^d \times \mathbb{S}, \quad \mu_\Phi=e^{-\Phi} \d x \otimes \nu..
\end{align*} 
On the predomain $D$ the generator $L$ associated to the fiber lay-down dynamics (see \eqref{Fiber_Operator_Intro}) is decomposed into $L=S-A$. Here $S$ and $A$ are defined on $D$ via
\begin{align*}
S=\frac{1}{2} \,\sigma^2 \, \Delta_{\mathbb{S}},~A= - \omega \cdot \nabla_x + \text{\normalfont{grad}}_{\mathbb{S}} \Psi \cdot \nabla_\omega~\mbox{ where }~\Psi(x,\omega)= \frac{1}{d-1} \,\nabla_x \Phi(x) \cdot \omega.
\end{align*}
Then the following properties are fulfilled.
\begin{itemize}
\item[(i)] $(S,D)$ is a symmetric and nonpositive definite linear operator on $H$. $(A,D)$ is an antisymmetric linear operator on $H$.
\item[(ii)] The probability measure $\mu_\Phi$ is invariant w.r.t.~$(S,D)$ and $(A,D)$. 
\item[(iii)] It holds $1 \in \overline{D}^L$ and $L1=0$. 
\item[(iv)] $(L,D)$ is essentially m-dissipative on $H$.
\end{itemize}
\end{Thm}

For the proof of (i),(ii) and (iii) see \cite[Lem.~3.5]{GS12B} and \cite[Lem.~3.9]{GS12B}. These are just simple calculations. The hardest part is of course to prove (iv), see \cite[Sec.~4]{GS12B}. Moreover, we refer to \cite[Sec.~3]{GS12B} for more details on the notations regarding the definition of $L$ and some basic calculation rules. 

With these properties at hand we can now start proving ergodicity with rate of convergence of the fiber lay-down dynamics. Therefore, we need to verify first the necessary data conditions required for applying the abstract ergodicity method from Section \ref{Results_abstract_ergodicity_method}; basically, the analytic dynamical system assumptions (A) therein are easily be implied (or are already shown) by the statements from Theorem \ref{Thm_basic_properties_fiber_lay_down_generator} above. However, it is left to verify Assumption (S). Therefore, we construct probability laws $\mathbb{P}$, $\hat{\mathbb{P}}$ as required in (A)' and (S)' such that $\mathbb{P}$ is associated with the fiber lay-down semigroup $(T_{t,2})_{t \geq 0}$. The construction of the laws uses an abstract scheme from the theory of generalized Dirichlet forms. However, the arguments below for verifying these abstract conditions are standard. Detailed definitions are not needed in the sequel and are therefore not introduced in the proof below. However, we give precise references to the literature where definitions and details can be found. Thus the interested reader who is not familiar with the theory of generalized Dirichlet\index{Dirichlet forms} forms may skip the arguments in first reading. 

\begin{proof}[proof of Proposition  \ref{Pp_(P)_implies_A_and_S_fiber_lay_down}]
In order to verify (A) and (S), we aim to apply Theorem \ref{Thm_Martingale_Problem_via_regular semigroups} and verify Conditions (A)' and (S)' from Section \ref{Results_abstract_ergodicity_method}. We first prove (A)'. Recall that $D=C_c^\infty(\mathbb{M})$ and $H=L^2(\mathbb{M},\mu_\Phi)$ with $\mu_\Phi=e^{-\Phi} \mathrm{d}x \otimes \nu$. Adopting the notations from (A1)-(A3), the manifold $\mathbb{M}=\mathbb{R}^d \times \mathbb{S}$ plays the role of $E$, $\mu_\Phi$ plays the role of $\mu$ therein and we equip $\mathbb{M}$ with the Euclidean metric induced by $\mathbb{R}^{2d}$. We remark that that we could also endow $\mathbb{M}$ with the metric induced by its Riemannian manifold structure. However, it is well-known that in the latter case $\mathbb{M}$ again becomes a separable metric space whose topology coincides with the relative topology induced by $\mathbb{R}^{2d}$. Hence we can infer that $(C([0,\infty);\mathbb{M}),\mathcal{F}_C)$ with $\mathcal{F}_C=\sigma \{ X_t ~|~t \geq 0\}$ does not depend on one of these metric structures our manifold $\mathbb{M}$ is endowed with. 

By Theorem \ref{Thm_basic_properties_fiber_lay_down_generator} we obtain that Conditions (A1), (A2) (and (A3)) are obviously satisfied. In the following, the closures 
\begin{align*}
(L,\overline{D}^L),\quad (S,\overline{D}^S),\quad (A,\overline{D}^A)
\end{align*}
as introduced after (A) on page \pageref{pageref_Conditions_analytic_A} are denoted by $(L,D(L))$, $(S,D(S))$, $(A,D(A))$ for the rest of this proof. Due to Theorem \ref{Thm_basic_properties_fiber_lay_down_generator}, the closure $(L,D(L))$ of $(L,D)$ in $H$ generates a s.c.c.s in $H$ that we call the fiber lay-down semigroup $(T_{t,2})_{t \geq 0}$.  Recall the identity
\begin{align*}
T_{t,2}f-f = \int_0^t T_{s,2} Lf \, \mathrm{d}s = \int_0^t L T_{s,2}f \, \mathrm{d}s,\quad t \geq 0,\quad f \in D(L).
\end{align*}
Thus conservativity of $(T_{t,2})_{t \geq 0}$ follows from Theorem \ref{Thm_basic_properties_fiber_lay_down_generator}\,(iii) and invariance of $\mu_\Phi$ w.r.t.~$(T_{t,2})_{t \geq 0}$ is satisfied since $\mu_\Phi$ is also invariant for $(L,D(L))$. Furthermore, it is easy to see that $(L,D)$ is an abstract diffusion operator on $L^2(\mathbb{M},\mu_\Phi)$ as defined in \cite[Def.~1.5]{Eb99}. This basically follows since $(L,D)$ is a second order partial differential operator without zero order term. This together with the property that $\mu_\Phi$ is invariant for $(L,D)$ implies that $(T_{t,2})_{t \geq 0}$ is indeed sub-Markovian, see \cite[Lem.~1.9]{Eb99}. The fact that the generator of the adjoint semigroup $(\hat{T_{t,2}})_{t \geq 0}$ extends $(\hat{L},D)$ is obvious. Altogether, the conditions from (A)' are shown and one obtains the desired semigroup $(T_{t,2})_{t \geq 0}$. But before proving (S)' we need some more preliminary considerations that will be used at the end of the proof. Therefore, consider the mapping
\begin{align*}
U\colon L^2(\mathbb{M},\mu_\Phi) \to  L^2(\mathbb{M},\mu_\Phi)
\end{align*}
defined by $Uf(x,\omega):=f(x,-\omega)$ for $(x,\omega) \in \mathbb{M}$. Clearly, $U$ is an unitary isomorphism mapping the constant functions to itself and $(L,D)$ is transformed under $U$ to $(\widehat{L},D)$ where $\hat{L}=S+A$ on $D$. Hence $(\hat{L},D)$ is again essentially m-dissipative on $\H$ and the constant functions are again elements of the kernel of its closure. Thus by using the same arguments as at the beginning of the proof it follows that the closure of $(\hat{L},D)$ in $H$ generates a conservative $\mu_\Phi$-invariant sub-Markovian s.c.c.s.~$(S_t)_{t \geq 0}$ on $H$. However, $(S_t)_{t \geq 0}$ coincides with the dual semigroup of $(T_{t,2})_{t \geq 0}$ in $H$. This follows since one easily sees that the generator of both semigroups is given by the closure of $(\hat{L},D)$ in $H$ by using the well-known fact that m-dissipative operators do not posses proper dissipative extensions.

Next we check (S)' and construct the desired laws $\mathbb{P}$ and $\hat{\mathbb{P}}$. We start with the construction of $\mathbb{P}$. We proceed similar as in the proof of \cite[Theo.~3]{CG08}, \cite[Theo.~2.5]{CG10} (or \cite[Theo.~6.3.2]{Con11} equivalently) and \cite[Cor.~2.7]{CG10} by applying the theory of generalized Dirichlet forms. Consider \cite{St99}, \cite{Tru03} and \cite{Tru05} (or \cite[Sec.~2.2]{Con11}) for basic notations used in the rest of this proof. The conditions that we need to verify below are obtained from several references and are summarized in \cite[Sec.~2.2]{Con11} in a nice overview. First of all, \cite[Lem.~1.9]{Eb99} implies that $(L,D)$ is even a m-dissipative Dirichlet operator. For the definition of a Dirichlet operator see e.g.~\cite[(1.18)]{Eb99}. Thus \cite[Prop.~I.4.7]{St99} (or consider \cite[Lem.~3.4]{Con05}) implies that $(L,D(L))$ generates a generalized Dirichlet form on $H$\index{Dirichlet forms}. And since $D=C_c^\infty(\mathbb{M})$ is a core for $(L,D(L))$ which is an algebra consisting of continuous functions and separating the points of $\mathbb{M}$, one easily verifies that this Dirichlet form is quasi-regular (see \cite[Def.~IV.1.7]{St99}) and satisfies Condition D3 from \cite[Ch.~IV.2]{St99} (or from \cite{Tru05} equivalently) by using \cite[Prop.~IV.2.1]{St99}. Thus the existence of a special standard process 
\begin{align*}
\mathbf{M}=\left( \Omega, \mathcal{M}, (\mathcal{M}_t)_{t \geq 0},  (x_t,\omega_t)_{t \geq 0}, \mathbb{P}_{(x,\omega) \in \mathbb{M}_\Delta} \right)
\end{align*} 
associated with $(T_{t,2})_{t \geq 0}$ follows from \cite[Theo.~IV.4.2]{St99}; the definition of associatedness in this situation is given below. However, since $(T_{t,2})_{t \geq 0}$ is conservative, we may assume that $\mathbf{M}$ has infinite lifetime $\mathbb{P}_{(x,\omega)}$-a.s.~for any initial point $(x,\omega) \in \mathbb{M}$, see the proof of \cite[Theo.~2.5]{CG10}. Thus in fact we do not need to join the cemetry $\Delta$ to $\mathbb{M}$ and may assume that the state space of $\mathbf{M}$ is just $\mathbb{M}$.\index{Hunt process} We further mention that $\mathbf{M}$ is even a Hunt process (cf.~\cite[Sec.~2.2]{Con11}). Moreover, since $(L,D)$ is a linear  partial differential operator of second order without zero order terms, it follows automatically that $\mathbf{M}$ has continuous sample paths $\mathbb{P}_{(x,\omega)}$-a.s.~for quasi-every $(x,\omega) \in \mathbb{M}$. The last property follows by \cite[Theo.~3.3]{Tru03} or \cite[Theo.~3.58]{Con05}.\index{Dirichlet forms}

Let us further recall that associatedness of the special standard process\index{associatedness!with a right process} $\mathbf{M}$ with $(T_{t,2})_{t \geq 0}$ here means that for each $t > 0$ and each $f \in L^\infty(\mathbb{M},\mu_\Phi)$ with bounded $\mu_\Phi$-version $\hat{f} \colon \mathbb{M} \to \mathbb{R}$ it holds that $p_t \hat{f}$ is $\mu_\Phi$-version of $T_{t,2}f$, see e.g.~\cite[Lem.2.2.8]{Con11}. Here $p_t \hat{f}(x,\omega)=\mathbb{E}_{(x,\omega)}[\hat{f}(x_t,\omega_t)]$, $(x,\omega) \in \mathbb{M}$, is the transition kernel of $\mathbf{M}$ and $\mathbb{E}_{(x,\omega)}$ denotes expectation w.r.t.~$\mathbb{P}_{(x,\omega)}$.

Now we follow the construction scheme \cite[Rem~2.2.9]{Con11}. First of all, one can always define the probability measure $\mathbb{P}\df \mathbb{P}_{\mu_\Phi}$ on $(\Omega,\mathcal{M})$ via
\begin{align*}
\mathbb{P}_{\mu_\phi} = \int_\mathbb{M} \mathbb{P}_{(x,\omega)} \, \mathrm{d}\mu_\Phi (x,\omega).
\end{align*}
where measurability of the integrand is ensured by the defining properties of a special standard process $\mathbf{M}$. Thus 
\begin{align*} 
\mathbb{P} \left(\big\{ \nu \in \Omega~\big|~ t \mapsto (x_t(\nu),\omega_t(\nu)) \in \mathbb{M} \mbox{~~is continuous on }[0,\infty) \big\}\right)=1.
\end{align*}
So, one can easily construct a $\mathcal{M} / \mathcal{F}_C$-measurable mapping
\begin{align*}
\iota \colon \Omega \to C([0,\infty);\mathbb{M})
\end{align*}
by using that $\mathcal{F}_C$ coincides with $\sigma \{ (x_t,\omega_t)~|~t \geq 0 \}$; here and in the following the evaluation of paths at time $t$ is also denoted by $(x_t,\omega_t)$ instead of $X_t$. The image measure of $\mathbb{P}$ under the mapping $\iota$ is denoted with the same symbol. Then $\mathbb{P}$ defines our desired probability law associated with $(T_{t,2})_{t \geq 0}$ as required in (A)' where associatedness is now understood as in Section \ref{Results_abstract_ergodicity_method}. Indeed, for all bounded nonnegative $f_1, \ldots, f_n \colon \mathbb{M} \to \mathbb{R}$, $n \in \mathbb{N}$, and all $0 \leq t_1 \leq \ldots \leq t_n < \infty$ the Markov property implies
\begin{align*}
\mathbb{E}_{(\cdot)}\left[ f_1(x_{t_1},\omega_{t_1}) \cdots f_n(x_{t_n},\omega_{t_n})\right] = p_{t_1} (f_1 p_{t_2-t_1}(f_2 \cdots p_{t_n - t_{n-1}}f_n) \cdots)
\end{align*}
and the right hand side is a $\mu_\Phi$-version of $T_{t_1,2} (f_1 T_{t_2-t_1,2}(f_2 \cdots T_{t_n - t_{n-1},2}f_n)$; for more details on this construction scheme we refer to \cite[Rem.~2.2.9]{Con11}. Finally, exactly the same arguments and construction scheme applies to $(\hat{L},D(\hat{L}))$ (the closure of $(\hat{L},D)$ in $H$). Hence there exists a probability law $\hat{\mathbb{P}}$ on $C([0,\infty);\mathbb{M}),\mathcal{F}_C)$ which is associated with $(S_t)_{t \geq 0} = (\hat{T}_{t,2})_{t \geq 0}$. In particular, $\hat{\mathbb{P}}_T$ is associated with $(\hat{T}_{t,2})_{t \in [0,T]}$ for each $T \geq 0$. However, also $\mathbb{P}_T \circ \tau_T^{-1}$ is associated with $(\hat{T}_{t,2})_{t \in [0,T]}$ by Lemma \ref {Lm_properties_semigroup_related_law} where $\tau_T$ is the time-reversal on $C([0,T];\mathbb{M})$. Consequently, by uniqueness (cf.~Remark \ref{Rm_associatedness}\,(iii)) we conclude 
\begin{align*}
\hat{\mathbb{P}}_T=\mathbb{P}_T \circ \tau_T^{-1}\quad \mbox{on $(C([0,T];\mathbb{M}),\mathcal{F}_C)$ for each $T \geq 0$}.
\end{align*}
Summarizing, also Condition (S)' is verified and the claim follows by Theorem \ref{Thm_Martingale_Problem_via_regular semigroups}.
\end{proof}

\begin{Rm} \label{Rm_martingale_problem_fiber_semigroup}\index{martingale problem}
By Proposition \ref{Pp_computation_quadratic_variation} it follows that the coordinate process $(x_t,\omega_t)_{t \geq 0}$ provides a martingale solution for the operator $(L,D)$ under the previously constructed law $\mathbb{P}=\mathbb{P}_{\mu_\Phi}$. Furthermore, under $\mathbb{P}$ the process $(x_t,\omega_t)_{t \geq 0}$ has initial distribution $\mu_\Phi$ and $\mu_\Phi$ as invariant measure and $\mathbb{P}$ is uniquely determined via associatedness with $(T_{t,2})_{t \geq 0}$. We mention that one can even prove pointwise statements. We do not need such pointwise statements in the sequel, however, we shall mention what else can be shown. Indeed, in the proof of Proposition \ref{Pp_(P)_implies_A_and_S_fiber_lay_down} we constructed a diffusion process
\begin{align*}
\mathbf{M}=\left( \Omega, \mathcal{M}, (\mathcal{M}_t)_{t \geq 0},  (x_t,\omega_t)_{t \geq 0}, \mathbb{P}_{(x,\omega) \in \mathbb{M}} \right)
\end{align*} 
whose transition kernel\index{transition semigroup/kernel} coincides with the semigroup $(T_{t,2})_{t \geq 0} $ $\mu_\Phi$-a.e.~on $\mathbb{M}$. And moreover, $\mathbf{M}$ even solves the martingale problem for the operator $(L,C_c^2(\mathbb{M}))$ under $\mathbb{P}_{(x,\omega)}$ for quasi-any starting point $(x,\omega) \in \mathbb{M}$. This can be shown as for the Langevin dynamics, see \cite[Theo.~3, Cor.~1]{CG08} and \cite[Theo.~5]{CG10} for details and notations. Following \cite{CG08}, even the construction of a weak solution to the underlying stochastic differential equation \eqref{Fiber_Model_Intro} with the help of this functional analytic approach based on the theory of Dirichlet forms seems to be possible. 

Summarizing, this remark show that we can connect the semigroup $(T_{t,2})_{t \geq 0}$ on the Hilbert space $L^2(\mathbb{M},\mu_\Phi)$ with a diffusion process which solves (in a suitable sense) our underlying stochastic fiber lay-down equation. Hence all our analytic considerations are indeed natural and one sees once more the strength of the theory of Dirichlet forms.
\end{Rm}

Now again only assume that $\Phi \colon \mathbb{R}^d \to \mathbb{R}$ satisfies Condition (C0) from Section \ref{Results_Applications_generalized_fiber_lay_down}. We define $P\colon H \rightarrow H$ in the fiber lay-down case as
\begin{align*}
Pf:=\int_{\mathbb{S}} f \, \mathrm{d}\nu,\quad f \in H.
\end{align*}
Then $P$ is really an orthogonal projection and the following statement holds.

\begin{Pp} \label{Lm_Verifying_H1_fiber_lay_down}
Let $d \in \mathbb{N}$, $d \geq 2$ and $\sigma \in (0,\infty)$. Assume that the potential $\Phi\colon \mathbb{R}^{d} \to \mathbb{R} $ satisfies (C0). The operator $(S,D)$ is a nonpositive definite, essentially self-adjoint operator in $H$. For the kernel $\mathcal{N}(S)$ of its closure $(S,\overline{D}^S)$ we have
\begin{align*}
\mathcal{N}(S)=\mathcal{R}(P)=L^2(\mathbb{R}^{d},e^{-\Phi}\mathrm{d}x)
\end{align*}
In other words, $P$ is really the orthogonal projection onto $\mathcal{N}(S)$. Moreover,
\begin{align} \label{Eq_microscopic_inequality_fiber_lay_down}
-\left(Sf,f\right)_H \geq \frac{1}{2} \sigma^2 (d-1) \,\| f- Pf \|^2_{H},\quad f \in D.
\end{align}
In particular, Condition (E1) holds with $\Lambda_m=\frac{1}{2} \sigma^2 (d-1)$.
\end{Pp}

\begin{proof}
The argument is similar as the one used for proving Proposition \ref{Lm_Verifying_H1_N_particle_Langevin}. First it is well known that $(\Delta_{\mathbb{S}},C^\infty(\mathbb{S}))$ is essentially self-adjoint in $L^2(\mathbb{S},\nu)$, see e.g.~\cite{Tri72}. Thus essential self-adjointness of $(S,D)$ in $H$ follows by \cite[Theo.~VIII.33]{RS80} (or use elementary arguments). As already argued in \cite[Prop.~3.12]{GS12B} the Poincaré inequality on $\mathbb{S}$ directly implies Inequality \eqref{Eq_microscopic_inequality_fiber_lay_down}. Then clearly \eqref{Eq_microscopic_inequality_fiber_lay_down} carries over to each $f \in \overline{D}^S$. Hence for some $f \in \overline{D}^S$ with $Sf=0$ it follows that $f=Pf$. Vice versa, standard approximation shows that $L^2(\mathbb{R}^{d},e^{-\Phi}\mathrm{d}x) \subset \overline{D}^S$ and $Sf=0$ for each $f \in L^2(\mathbb{R}^{d},e^{-\Phi}\mathrm{d}x)$, see also \cite[Lem.~3.8]{GS12B}.
\end{proof}

We further mention that the technical condition required in \eqref{Technical_Condition_Ergodicity} is obviously satisfied since 
\begin{align*}
P(D)=C_c^\infty(\mathbb{R}^{d}) \subset C_c^\infty(\mathbb{M})=D.
\end{align*}
This yields the formula
\begin{align*}
Af=-\omega \cdot \nabla_x f\quad \mbox{for all } f \in C_c^\infty(\mathbb{R}^{d}).
\end{align*}
Hence for each such $f$ we have $Af \in D$ and \eqref{Technical_Condition_Ergodicity} is fulfilled. Next, we verify first the algebraic relation required in (E4).

\begin{Pp} \label{Lm_verification_E3_fiber_lay_down}
Let $d \in \mathbb{N}$, $d \geq 2$. Assume that $\Phi$ satisfies (C0) and let $\sigma \in (0,\infty)$. Then Condition (E4) is fulfilled with $c_3=\frac{1}{2} \sigma^2 (d-1)$.
\end{Pp}

\begin{proof}
As seen before, we have
\begin{align*}
Af_S=-\omega \cdot \nabla_x f_S, \quad P f=f_S \in C_c^\infty(\mathbb{R}^{d})\quad \mbox{for all } f \in D.
\end{align*}
Furthermore, recall the identity $\Delta_\mathbb{S} \,\omega =- (d-1) \,\omega$ by \cite[Lem.~7.1]{GKMS12}. Thus we can infer 
\begin{align} \label{Eq_verification_E3_fiber_lay_down}
SAf= \frac{1}{2} \,(d-1)\, \sigma^2 \,\omega \cdot \nabla_x f =-\frac{1}{2} (d-1) \sigma^2 \, Af\quad \mbox{for each } f \in C_c^\infty(\mathbb{R}^{d}).
\end{align}
Hence (E4) is fulfilled with the claimed value for $c_3$.
\end{proof}

Now we verify (E2) and (E3). Therefore, let first $\Phi$ be as in the preceding proposition recall the definition of $G_\Phi$ and $H_\Phi$ from the introduction. We further note that Conditions (C0)-(C3) for proving ergodicity of the fiber lay-down dynamics are similar to the conditions for proving ergodicity of the $N$-particle Langevin dynamics. This is simply due to the fact that both dynamics admit the same macroscopic evolution operator $PA^2P$, see below.  But first we need some formulas that are already computed in \cite{GKMS12} and \cite{GS12B}.  We recall that
\begin{align} \label{formula_AAP}
A^2Pf = - A \left( \omega \cdot \nabla_x f_S\right)=\left(\omega, \nabla^2 f_S \,\omega\right)_{\text{euc}} - \frac{1}{d-1} \left( (I-\omega \otimes \omega) \nabla \Phi, \nabla f_S\right)_{\text{euc}}.
\end{align}
Above $f \in D$ and again $f_S:=Pf \in C_c^\infty(\mathbb{R}^d)$. By using the Gaussian integral formula this implies
\begin{align*}
Gf=PA^2Pf= \frac{1}{d} \left( \Delta_x - \nabla_x \Phi \cdot \nabla_x \right) f_S.
\end{align*}
So, by using the notations from Section \ref{Results_abstract_ergodicity_method} we have
\begin{align*}
D_{P}=P(D)=C_c^\infty(\mathbb{R}^{d}),\quad H_{P}=P(H)=L^2(\mathbb{R}^{d},e^{-\Phi}\mathrm{d}x)=H_{\Phi}
\end{align*}
as well as
\begin{align} \label{Relation_G_G_psi_fiber_case}
G= \frac{1}{d} \left( \Delta_x - \nabla_x \Phi  \cdot \nabla_x \right)= \frac{1}{d} \, G_{\Phi} \quad \mbox{ on } D_{P}.
\end{align}
Thus $G$ indeed looks like the operator from the $N$-particle Langevin dynamics and consequently, (E2) and (E3) can similarly be verified in the upcoming proposition. But first note that (C1) implies that there exists constants $K_1(d,\Phi),K_2(d,\Phi) \in [0,\infty)$  independent of $f \in C_c^\infty(\mathbb{R}^{d})$ and only depending on the choice of $\Phi$ (and on $d$) such that
\begin{align} \label{Eq_Kato_bound_E3_fiber}
\sum_{i,j=1}^{d} \left\| \partial_{x_i} \partial_{x_j} f \right\|_{H_\Phi} + \frac{1}{d-1} \, \sum_{i,j =1}^{d} \left\| \left(\partial_{x_i}\Phi \right) \partial_{x_j}f \right\| \leq K_1(d,\Phi) \left\| G_\Phi f \right\|_{H_\Phi} + K_2(d,\Phi) \,\left\| f \right\|_{H_\Phi}
\end{align}

The desired proposition reads as follows.

\begin{Pp} \label{Lm_Verifying_H3_H4_fiber_lay_down}
Assume that $\Phi$ satisfies (C0), (C1), (C2) and (C3). Then also (E2) and (E3) are fulfilled. Moreover, the constants $c_1$ and $c_2$ from (E2) are of the form
\begin{align*}
c_1= \frac{\sqrt{d}\,(d-1)}{2\, \sqrt{\text{\normalfont{gap}}(G_{\Phi})}} \,\sigma^2 + d \, K_1(d,\Phi),\quad c_2= K_2(d,\Phi).
\end{align*}
Here the constant $K_1(d,\Phi),K_2(d,\Phi) \in [0,\infty)$ are obtained by the Kato-bound from (C1) in \eqref{Eq_Kato_bound_E3_fiber} and are only depending on the choice of $\Phi$ (and on the dimension $d$). Moreover,
\begin{align*}
\Lambda_M = \frac{\text{\normalfont{gap}}(G_{\Phi})}{d}.
\end{align*}
\end{Pp}

\begin{proof}
We verify (E3). First of all, Conditions (C2) and (C3) together easily imply (E3) as in the proof of Lemma \ref{Lm_Verifying_H3_H4_N_particle_Langevin}. For convenience we recall the argument here. Following the notation in (E3) we have that $(G,D(G))\df (G,\overline{D_P}^G)$ denotes the closure of $(G,D_P)$ in $H_P$ where $P$ is given by $P=P$ in the ergodicity framework. Essential selfadjointness of $(G,D_P)$ in $H_{P}$ is implied by Condition (C2) by using the relation of $G$ with $G_\Phi$ from \eqref{Relation_G_G_psi_fiber_case}. As in \cite[Lem.~4.7]{CG10} it follows  $1\in D(G)$ and $G1=0$. Indeed, integration by parts implies $\left(1,Gf\right)_{H_P}=0$ for each $f \in D_P$. Thus $1 \in D(G)$ and $G1=0$ since $(G,D(G))$ is selfadjoint on $H_P$. By Relation \eqref{Relation_G_G_psi_fiber_case}, Condition (C3) yields the macroscopic coercivity inequality required in (E3) and the formula for $\Lambda_M$. We verify (E2). Let $f \in C_c^\infty(\mathbb{R}^d)=D_{P}$. Then by \eqref{formula_AAP} and \eqref{Eq_verification_E3_fiber_lay_down} we obtain
\begin{align*}
LAP f = \,&\frac{1}{2} \,(d-1) \,\sigma^2\,\omega \cdot \nabla_x f - \sum_{i,j=1}^{d} \omega_i \, \omega_j  \,\partial_{x_i}\partial_{x_j} f  + \frac{1}{d-1} \sum_{i,j=1}^{d} (\delta_{ij} - \omega_i \omega_j) \,\partial_{x_i} \Phi \,\partial_{x_j} f.
\end{align*}
The Gaussian integral formula yields $\| \omega \cdot \nabla_x f \|^2_{H}= \frac{1}{d}\, \| \nabla_x f \|^2_{H_\Phi}$, see \cite[Lem.~7.3]{GKMS12}. Define $\widetilde{f} \df f - \left(f,1\right)_{H_\Phi} \in D(G)$. We get
\begin{align*}
\| \omega \cdot \nabla_x f \|_{H}^2 = - \left( G f, f \right)_{H_\Phi} =  -\left( G \widetilde{f}, \widetilde{f} \right)_{H_\Phi} \leq \frac{1}{\Lambda_M} \| G \widetilde{f} \|_{H_\Phi}^2 = \frac{d}{\text{\normalfont{gap}}(G_{\Phi})} \| G f\|_H^2.
\end{align*}
This yields the estimate
\begin{align*}
\left\| LAPf \right\|_H \leq \,&\frac{\sqrt{d}\,(d-1)}{2\,\sqrt{\text{\normalfont{gap}}(G_{\Phi})}} \,\sigma^2 \,\|Gf\|_H  + \sum_{i,j=1}^{d} \left\| \partial_{x_i} \partial_{x_j} f \right\|_{H_\Phi}  + \frac{1}{d-1} \, \sum_{i,j=1}^{d} \left\| \left(\partial_{x_i}\Phi \right) \, \partial_{x_j} f \right\|_{H_\Phi}.
\end{align*}
By the Kato-bound provided in (C1) (see \eqref{Eq_Kato_bound_E3_fiber}) and the relation of $G$ with $G_\Phi$ from \eqref{Relation_G_G_psi_fiber_case} also the Kato-boundedness condition from (E2) is satisfied with the claimed values for $c_1$ and $c_2$ from the statement. 
\end{proof}

So, we are able to verify the ergodicity theorem with rate of convergence for the fiber lay-down dynamics.

\begin{proof}[Proof of Theorem \ref{Ergodicity_theorem_generalized_fiber_lay_down}] \phantomsection\label{proof_ergodicity_fiber}
We aim to apply Theorem \ref{Cor_main_ergodicity_theorem_with_rate_of_convergence}.  Now the dynamical system assumptions are established in Proposition \ref{Pp_(P)_implies_A_and_S_fiber_lay_down}. Uniqueness of the constructed law $\mathbb{P}=\mathbb{P}_{\mu_\Phi}$ associated with $(T_{t,2})_{t \geq 0}$ follows from Remark \ref{Rm_associatedness}. Condition (E1) is shown in Proposition \ref{Lm_Verifying_H1_fiber_lay_down}, (E2) and (E3) are shown in Proposition \ref{Lm_Verifying_H3_H4_fiber_lay_down} and (E4) is verified in Proposition \ref{Lm_verification_E3_fiber_lay_down}. Consequently, Theorem \ref{Cor_main_ergodicity_theorem_with_rate_of_convergence} implies ergodicity with rate of convergence. The quantitative description of the constants occurring in the rate of convergence are obtained by a straightforward calculation. Indeed, one only needs plugging the  constants $\Lambda_m$, $\Lambda_M$, $c_1$, $c_2$ and $c_3$ from the previous statements into the concrete rate predicted by Theorem \ref{Cor_main_ergodicity_theorem_with_rate_of_convergence}. Then $A(\Phi)$ and $B(\Phi)$ are calculated as
\begin{align} \label{Specification_of_mi}
A(\Phi) = \frac{2}{\sqrt{d-1}} \left(d \,K_1(d,\Phi) +1\right),\quad B(\Phi) = \frac{2\,d\, K_2(d,\Phi)}{\sqrt{d-1}}.
\end{align}
where $K_1(d,\Phi)$ and $K_2(d,\Phi)$ are the constants occurring in the Kato-bound from \eqref{Eq_Kato_bound_E3_fiber}. Finally, if $\Phi$ satisfies the stronger Assumptions (C1)-(C3) with the respective constants $\Lambda \in (0,\infty)$ and $c \in [0,\infty)$, then already by Lemma \ref{Lm_sufficient_criteria_implying_P_E_fiber_case} we obtain that $\Phi$ fulfills (C1)'-(C3)'. The statement $\text{\normalfont{gap}}(G_{\Phi}) \geq \Lambda$ is obvious. We remark that (C1)' only requires (C3), cf.~\cite[Lem.~4.8]{CG10}. And \cite[Lem.~4.8]{CG10} also shows that $K_1(d,\Phi)$ and $K_2(d,\Phi)$ then only depend on the value of $c$ and the dimension $d$. Thus $A(\Phi)$ and $B(\Phi)$ only depend on the value of $c$ and the dimension $d$. This finishes the proof.
\end{proof}

\begin{Rm}
Of course, as mentioned in \cite{CG10} (regarding the two-dimensional version of the fiber lay-down dynamics), more general assumptions on $\Phi$ may be allowed for covering also singular potentials analogously to the ones assumed for the $N$-particle Langevin dynamics from Section \ref{Results_Applications_N_particle_Langevin}. The construction scheme for the semigroup and the associated law then follows the scheme from \cite{CG10}. However, we are not interested in such a generalization for singular potentials for the fiber lay-down process and therefore do not discuss further details on this generalization.
\end{Rm}

\subsection*{Acknowledgment}
This work has been supported by Bundesministerium f\"ur Bildung und Forschung, Schwerpunkt \glqq Mathematik f\"ur Innovationen in Industrie and Dienstleistungen\grqq, Verbundprojekt ProFil, 05M10.


\begin{thebibliography}{GKMW07}

\bibitem[Bau92]{Bau92}
H.~Bauer.
\newblock {\em Ma\ss - und {I}ntegrationstheorie}.
\newblock de Gruyter Lehrbuch. Walter de Gruyter \& Co., Berlin, second
  edition, 1992.

\bibitem[Bau13]{Bau13}
F.~Baudoin.
\newblock Bakry-{E}mery meet {V}illani.
\newblock \textit{{A}rXiv preprint},~1308.4938, 2013.

\bibitem[BBCG08]{BBCG08}
D.~Bakry, F.~Barthe, P.~Cattiaux, and A.~Guillin.
\newblock A simple proof of the {P}oincar\'e inequality for a large class of
  probability measures including the log-concave case.
\newblock {\em Electron. Commun. Probab.}, 13:60--66, 2008.

\bibitem[BBR06]{BBR06}
L.~Beznea, N.~Boboc, and M.~R{\"o}ckner.
\newblock Markov processes associated with {$L^p$}-resolvents and applications
  to stochastic differential equations on {H}ilbert space.
\newblock {\em J. Evol. Equ.}, 6(4):745--772, 2006.

\bibitem[BCG08]{BCG08}
D.~Bakry, P.~Cattiaux, and A.~Guillin.
\newblock Rate of convergence for ergodic continuous {M}arkov processes:
  {L}yapunov versus {P}oincar\'e.
\newblock {\em J. Funct. Anal.}, 254(3):727--759, 2008.

\bibitem[Bec89]{Be89}
W.~Beckner.
\newblock A generalized {P}oincar\'e inequality for {G}aussian measures.
\newblock {\em Proc. Amer. Math. Soc.}, 105(2):397--400, 1989.

\bibitem[BKR97]{BKR97}
V.~I. Bogachev, N.~V. Krylov, and M.~R{\"o}ckner.
\newblock Elliptic regularity and essential self-adjointness of {D}irichlet
  operators on {$\mathbf{R}^{n}$}.
\newblock {\em Ann. Scuola Norm. Sup. Pisa Cl. Sci. (4)}, 24(3):451--461, 1997.

\bibitem[CG08]{CG08}
F.~Conrad and M.~Grothaus.
\newblock Construction of {$N$}-particle {L}angevin dynamics for
  {$H^{1,\infty}$}-potentials via generalized {D}irichlet forms.
\newblock {\em Potential Anal.}, 28(3):261--282, 2008.

\bibitem[CG10]{CG10}
F.~Conrad and M.~Grothaus.
\newblock Construction, ergodicity and rate of convergence of {$N$}-particle
  {L}angevin dynamics with singular potentials.
\newblock {\em J. Evol. Equ.}, 10(3):623--662, 2010.

\bibitem[CKW04]{CKW04}
W.~T. Coffey, Yu.~P. Kalmykov, and J.~T. Waldron.
\newblock {\em The {L}angevin {E}quation: {W}ith {A}pplications to {S}tochastic
  {P}roblems in {P}hysics, {C}hemistry and {E}lectrical {E}ngineering},
  volume~14 of {\em World Scientific Series in Contemporary Chemical Physics}.
\newblock World Scientific Publishing Co. Inc., River Edge, NJ, second edition,
  2004.

\bibitem[Con05]{Con05}
F.~Conrad.
\newblock {\em Non-sectorial diffusions and an application to continuous
  {N}-particle Langevin dynamics for a general class of interaction
  potentials}.
\newblock Diploma thesis, Department of Mathematics, University of
  Kaiserslautern, 2005.

\bibitem[Con11]{Con11}
F.~Conrad.
\newblock {\em Construction and analysis of Langevin dynamics in continuous
  particle systems}.
\newblock PhD thesis, University of Kaiserslautern. Published by Verlag
  Dr.~Hut, M{\"u}nchen, 2011.

\bibitem[DKMS13]{DKMS11}
J.~Dolbeault, A.~Klar, C.~Mouhot, and C.~Schmeiser.
\newblock Exponential {R}ate of {C}onvergence to {E}quilibrium for a {M}odel
  {D}escribing {F}iber {L}ay-{D}own {P}rocesses.
\newblock {\em Applied {M}athematics {R}esearch e{X}press}, 2013(2):165--175,
  2013.

\bibitem[DMIPP84]{DIPP84}
A.~De~Masi, N.~Ianiro, A.~Pellegrinotti, and E.~Presutti.
\newblock A survey of the hydrodynamical behavior of many-particle systems.
\newblock In {\em Nonequilibrium phenomena, {II}}, Stud. Statist. Mech., XI,
  pages 123--294. North-Holland, Amsterdam, 1984.

\bibitem[DMS14]{DMS13}
J.~Dolbeault, C.~Mouhot, and C.~Schmeiser.
\newblock Hypocoercivity for linear kinetic equations conserving mass.
\newblock \textit{{A}rXiv preprint},~1005.1495 (2010). To appear in
  Transactions of the American Mathematical Society, 2014.

\bibitem[Dua11]{Dua11}
R.~Duan.
\newblock Hypocoercivity of linear degenerately dissipative kinetic equations.
\newblock {\em Nonlinearity}, 24(8):2165--2189, 2011.

\bibitem[Ebe99]{Eb99}
A.~Eberle.
\newblock {\em Uniqueness and non-uniqueness of semigroups generated by
  singular diffusion operators}, volume 1718 of {\em Lecture Notes in
  Mathematics}.
\newblock Springer-Verlag, Berlin, 1999.

\bibitem[EK86]{EK86}
S.~Ethier and T.~Kurtz.
\newblock {\em Markov processes}.
\newblock Wiley Series in Probability and Mathematical Statistics: Probability
  and Mathematical Statistics. John Wiley \& Sons Inc., New York, 1986.

\bibitem[F{\=O}T94]{Fuk94}
M.~Fukushima, Y.~{\=O}shima, and M.~Takeda.
\newblock {\em Dirichlet forms and symmetric {M}arkov processes}, volume~19 of
  {\em de Gruyter Studies in Mathematics}.
\newblock Walter de Gruyter \& Co., Berlin, 1994.

\bibitem[Fuk80]{Fuk80}
M.~Fukushima.
\newblock {\em Dirichlet forms and {M}arkov processes}, volume~23 of {\em
  North-Holland Mathematical Library}.
\newblock North-Holland Publishing Co., Amsterdam, 1980.

\bibitem[GK08]{GK08}
M.~Grothaus and A.~Klar.
\newblock Ergodicity and rate of convergence for a nonsectorial fiber lay-down
  process.
\newblock {\em SIAM J. Math. Anal.}, 40(3):968--983, 2008.

\bibitem[GKMS12]{GKMS12}
M.~Grothaus, A.~Klar, J.~Maringer, and P.~Stilgenbauer.
\newblock Geometry, mixing properties and hypocoercivity of a degenerate
  diffusion arising in technical textile industry.
\newblock \textit{{A}rXiv preprint},~1203.4502, submitted for publication,
  2012.

\bibitem[GKMW07]{GKMW07}
T.~G{\"o}tz, A.~Klar, N.~Marheineke, and R.~Wegener.
\newblock A stochastic model and associated {F}okker-{P}lanck equation for the
  fiber lay-down process in nonwoven production processes.
\newblock {\em SIAM J. Appl. Math.}, 67(6):1704--1717, 2007.

\bibitem[Gol85]{Gol85}
J.~A. Goldstein.
\newblock {\em Semigroups of {L}inear {O}perators and {A}pplications}.
\newblock Oxford Mathematical Monographs. The Clarendon Press Oxford University
  Press, New York, 1985.

\bibitem[GS12]{GS12B}
M.~Grothaus and P.~Stilgenbauer.
\newblock {H}ypocoercivity for {K}olmogorov backward evolution equations and
  applications.
\newblock \textit{{A}rXiv preprint},~1207.5447, submitted for publication,
  2012.

\bibitem[GS13]{GS12}
M.~Grothaus and P.~Stilgenbauer.
\newblock Geometric {L}angevin equations on submanifolds and applications to
  the stochastic melt-spinning process of nonwovens and biology.
\newblock {\em Stochastics and Dynamics}, 13(4), 2013.

\bibitem[GS14]{GS14}
M.~Grothaus and P.~Stilgenbauer.
\newblock Hilbert space hypocoercivity for the {L}angevin dynamics revisited.
\newblock In preparation, 2014.

\bibitem[H{\'e}r06]{HeNi06}
F.~H{\'e}rau.
\newblock Hypocoercivity and exponential time decay for the linear
  inhomogeneous relaxation {B}oltzmann equation.
\newblock {\em Asymptot. Anal.}, 46(3-4):349--359, 2006.

\bibitem[H{\'e}r07]{Her07}
F.~H{\'e}rau.
\newblock Short and long time behavior of the {F}okker-{P}lanck equation in a
  confining potential and applications.
\newblock {\em J. Funct. Anal.}, 244(1):95--118, 2007.

\bibitem[HN04]{HeNi04}
F.~H{\'e}rau and F.~Nier.
\newblock Isotropic hypoellipticity and trend to equilibrium for the
  {F}okker-{P}lanck equation with a high-degree potential.
\newblock {\em Arch. Ration. Mech. Anal.}, 171(2):151--218, 2004.

\bibitem[HN05]{HN05}
B.~Helffer and F.~Nier.
\newblock {\em Hypoelliptic estimates and spectral theory for {F}okker-{P}lanck
  operators and {W}itten {L}aplacians}, volume 1862 of {\em Lecture Notes in
  Mathematics}.
\newblock Springer-Verlag, Berlin, 2005.

\bibitem[KMW12]{KMW12}
A.~Klar, J.~Maringer, and R.~Wegener.
\newblock A 3{D} model for fiber lay-down in nonwoven production processes.
\newblock {\em Math. Models Methods Appl. Sci.}, 22(9):1250020, 18, 2012.

\bibitem[KNR08]{KNR08}
S.~H. Kulkarni, M.~T. Nair, and G.~Ramesh.
\newblock Some properties of unbounded operators with closed range.
\newblock {\em Proc. Indian Acad. Sci. Math. Sci.}, 118(4):613--625, 2008.

\bibitem[LNP13]{LNP13}
T.~Lelièvre, F.~Nier, and G.A. Pavliotis.
\newblock {O}ptimal {N}on-reversible {L}inear {D}rift for the {C}onvergence to
  {E}quilibrium of a {D}iffusion.
\newblock {\em Journal of Statistical Physics}, 152(2):237--274, 2013.

\bibitem[LRS10]{LRS10}
T.~Leli{\`e}vre, M.~Rousset, and G.~Stoltz.
\newblock {\em {F}ree {E}nergy {C}omputations: {A} {M}athematical
  {P}erspective}.
\newblock Imperial College Press, London, 2010.

\bibitem[LRS12]{LRS12}
T.~Leli{\`e}vre, M.~Rousset, and G.~Stoltz.
\newblock Langevin dynamics with constraints and computation of free energy
  differences.
\newblock {\em Math. Comp.}, 81(280):2071--2125, 2012.

\bibitem[MR92]{MR92}
Z.~M. Ma and M.~R{\"o}ckner.
\newblock {\em Introduction to the theory of (nonsymmetric) {D}irichlet forms}.
\newblock Universitext. Springer-Verlag, Berlin, 1992.

\bibitem[MS02]{MS02}
J.~C. Mattingly and A.~M. Stuart.
\newblock Geometric ergodicity of some hypo-elliptic diffusions for particle
  motions.
\newblock {\em Markov Process. Related Fields}, 8(2):199--214, 2002.

\bibitem[Paz83]{Paz83}
A.~Pazy.
\newblock {\em Semigroups of linear operators and applications to partial
  differential equations}, volume~44 of {\em Applied Mathematical Sciences}.
\newblock Springer-Verlag, New York, 1983.

\bibitem[Ris89]{Ris89}
H.~Risken.
\newblock {\em The {F}okker-{P}lanck equation}, volume~18 of {\em Springer
  Series in Synergetics}.
\newblock Springer-Verlag, Berlin, second edition, 1989.
\newblock Methods of solution and applications.

\bibitem[RS80]{RS80}
M.~Reed and B.~Simon.
\newblock {\em Methods of modern mathematical physics. {I}. {F}unctional
  {A}nalysis}.
\newblock Academic Press Inc., New York, second edition, 1980.

\bibitem[Sch06]{Sch06}
F.~Schwabl.
\newblock {\em Statistical mechanics}.
\newblock Springer-Verlag, Berlin, second edition, 2006.

\bibitem[Sim63]{Sim63}
G.~F. Simmons.
\newblock {\em Introduction to topology and modern analysis}.
\newblock McGraw-Hill Book Co., Inc., New York, 1963.

\bibitem[Sta99]{St99}
W.~Stannat.
\newblock The theory of generalized {D}irichlet forms and its applications in
  analysis and stochastics.
\newblock {\em Mem. Amer. Math. Soc.}, 142(678):viii+101, 1999.

\bibitem[Sti14]{Sti14}
P.~Stilgenbauer.
\newblock {\em {T}he {S}tochastic {A}nalysis of {F}iber {L}ay-{D}own {M}odels:
  {A}n {I}nterplay between {P}ure and {A}pplied {M}athematics involving
  {L}angevin {P}rocesses on {M}anifolds, {E}rgodicity for {D}egenerate
  {K}olmogorov {E}quations and {H}ypocoercivity}.
\newblock PhD thesis, University of Kaiserslautern. Published by Verlag
  Dr.~Hut, M{\"u}nchen, 2014.

\bibitem[Tri72]{Tri72}
H.~Triebel.
\newblock {\em H\"ohere {A}nalysis}.
\newblock VEB Deutscher Verlag der Wissenschaften, Berlin, 1972.
\newblock Hochschulb{\"u}cher f{\"u}r Mathematik, Band 76.

\bibitem[Tru03]{Tru03}
G.~Trutnau.
\newblock On a class of non-symmetric diffusions containing fully nonsymmetric
  distorted {B}rownian motions.
\newblock {\em Forum Math.}, 15(3):409--437, 2003.

\bibitem[Tru05]{Tru05}
G.~Trutnau.
\newblock On {H}unt processes and strict capacities associated with generalized
  {D}irichlet forms.
\newblock {\em Infin. Dimens. Anal. Quantum Probab. Relat. Top.},
  8(3):357--382, 2005.

\bibitem[Vil09]{Vil09}
C.~Villani.
\newblock Hypocoercivity.
\newblock {\em Mem. Amer. Math. Soc.}, 202(950):iv+141, 2009.

\bibitem[Wan99]{Wan99}
F.~Y. Wang.
\newblock Existence of the spectral gap for elliptic operators.
\newblock {\em Ark. Mat.}, 37(2):395--407, 1999.

\bibitem[Wie85]{Wie85}
N.~Wielens.
\newblock The essential self-adjointness of generalized {S}chr{\"o}dinger
  operators.
\newblock {\em {J}ournal of {F}unctional {A}nalysis}, 61(1):98--115, 1985.

\bibitem[Wu01]{Wu01}
L.~Wu.
\newblock Large and moderate deviations and exponential convergence for
  stochastic damping {H}amiltonian systems.
\newblock {\em Stochastic Process. Appl.}, 91(2):205--238, 2001.

\end{thebibliography}
\end{document}